\documentclass[11pt,a4paper]{amsart}
\usepackage[utf8]{inputenc}
\usepackage{amsmath}
\usepackage{amsthm}
\usepackage{hyperref}
\usepackage[T1]{fontenc}
\usepackage{float}
\usepackage{color}
\hypersetup{linktocpage}
\usepackage{enumerate}
\usepackage{amsfonts}
\usepackage{rotating}
\usepackage{pdflscape}
\usepackage{enumitem}
\usepackage{makecell}
\usepackage{booktabs}
\numberwithin{table}{section}
\usepackage{amssymb}
\usepackage{lipsum}
\usepackage{xcolor}
\usepackage{mathtools}
\usepackage[most]{tcolorbox}
\usepackage{array}
\usepackage{graphicx}
\usepackage[left=2cm,right=2cm,top=2cm,bottom=2cm]{geometry}
\DeclarePairedDelimiter\floor{\lfloor}{\rfloor}
\DeclarePairedDelimiter\ceil{\lceil}{\rceil}

\newcolumntype{L}[1]{>{\raggedright\let\newline\\\arraybackslash\hspace{0pt}}m{#1}}
\newcolumntype{C}[1]{>{\centering\let\newline\\\arraybackslash\hspace{0pt}}m{#1}}
\newcolumntype{R}[1]{>{\raggedleft\let\newline\\\arraybackslash\hspace{0pt}}m{#1}}
\newcommand\numberthis{\addtocounter{equation}{1}\tag{\theequation}}
\newtheorem{theorem}{Theorem}[section]
\newtheorem{lemma}[theorem]{Lemma}
{\theoremstyle{definition}
\newtheorem{remark}[theorem]{Remark}
}\newtheorem{proposition}[theorem]{Proposition}
\newtheorem{corollary}[theorem]{Corollary}

\theoremstyle{definition}

\newtheorem{wexample}[theorem]{Worked Example}
\newcommand{\om}{\omega_}
\newcommand{\m}{\mu_}
\newcommand{\codim}{\mathrm{codim}}
\newcommand{\lam}{\lambda_}
\newcommand{\emax}{\mathcal{E}_{{\rm max}}}
\newcommand{\ep}{\epsilon}
\newcommand{\pgl}{\mathrm{PGL}}
\newcommand{\gl}{\mathrm{GL}}
\newcommand{\athree}{4 q^{\lfloor n^2/2 + 2 \rfloor}+\frac{43}{3}q^{2n^2/3+2}+2q^{2n^2/3+1}}
\newcommand{\atwoodd}{2q^{ n^2/2 + 3/2}}
\newcommand{\atwoeven}{4q^{n^2/2+2}}

\newcommand{\field}{2 \log (\log_2 q +2) q^{n^2/2+d/2}}
\newcommand{\itwo}{4(q^{\binom{n+1}{2}-1}+q^{\binom{n+1}{2}-2}) }
\renewcommand{\a}{\alpha}
\newcommand{\F}{\mathbb{F}}
\newcommand{\overf}{\overline{\mathbb{F}}}

\numberwithin{equation}{section}
\setlist[enumerate,1]{label={(\roman*)}}

\newcommand{\noop}[1]{}

\begin{document}
\author{Melissa Lee}
 \address{Department of Mathematics,
    Imperial College, London SW7 2BZ, UK}
 \email{m.lee16@imperial.ac.uk}
\title{Regular orbits of quasisimple linear groups II}

\begin{abstract}
Let $V$ be a finite-dimensional vector space over a finite field, and suppose $G \leq \Gamma \mathrm{L}(V)$ is a group with a unique subnormal quasisimple subgroup $E(G)$ that is absolutely irreducible on $V$.  A \textit{base} for $G$ is a set of vectors $B\subseteq V$ with pointwise stabiliser $G_B=1$. If $G$ has a base of size 1, we say that it has a \textit{regular orbit} on $V$.
 In this paper we investigate the minimal base size of groups $G$ with $E(G)/Z(E(G)) \cong \mathrm{PSL}_n(q)$ in defining characteristic, with an aim of classifying those with a regular orbit on $V$.
\end{abstract}

\maketitle

%\tableofcontents
%\pagebreak
%TODO Review any approximations of size of GL - doesn't work for q=2, see Tim Rmk 3.1.16
\section{Introduction} 
A permutation group $G \leq \mathrm{Sym}(\Omega)$ is said to have a \textit{regular orbit} on $\Omega$ if there exists $\omega \in \Omega$ with trivial stabiliser in $G$. 
The study of regular orbits arises in a number of contexts, particularly where $\Omega$ is a vector space $V$, and $G \leq \mathrm{GL}(V)$.
For instance, if $V$ is a finite vector space, $G\leq \mathrm{GL}(V)$, and all orbits of $G$ on $V\setminus \{0\}$ are regular, then the affine group $GV$ is a Frobenius group with Frobenius complement $G$, and such $G$ were classified by Zassenhaus \cite{MR3069653}. Regular orbits also underpin parts of the proof of the $k(GV)$-conjecture \cite{MR2078936}, asserting that the number of conjugacy classes $k(GV)$ of $GV$, with $|G|$ coprime to $|V|$, is at most $|V|$. One of the major cases was where $G$ is almost quasisimple and acts irreducibly on $V$. In this case, the existence of a regular orbit of $G$ on $V$ was sufficient to prove the $k(GV)$-conjecture. A classification of pairs $(G,V)$ where $G$ has a regular orbit on $V$ and $(|G|,|V|)=1$ was completed by K\"{o}hler and Pahlings \cite{MR1829482}, building on work of Goodwin \cite{goodwin1,goodwin2} and Liebeck \cite{MR1407889}.

A subset $B \subset V$ is a \textit{base} for $G$ if its pointwise stabiliser in $G$ is trivial. The minimal size of a base for $G$ is called the \textit{base size} and denote it by $b(G)$. For example, if $G$ has a regular orbit on $V$, then $b(G)=1$.
Each element $g\in G$ is characterised by its action on a base, so $b(G) \geq \lceil \log |G|/\log |V| \rceil$.
In recent years, there has been a number of advancements towards classifying finite primitive groups $H$ with small bases. A primitive group with $b(H)=1$ is cyclic, so the smallest case of interest is where $b(H)=2$. There has been a number of contributions to a classification here, including a partial classification for diagonal type groups \cite{MR2998958}, a complete classification for primitive actions of $S_m$ and $\mathrm{Alt}_m$ \cite{MR2781219, MR2214474} and sporadic groups \cite{MR2684423} and also substantial progress for almost simple classical groups \cite{MR3219555}.

A finite group $G$ is said to be \textit{quasisimple} if it is perfect and $G/Z(G)$ is a non-abelian simple group. We further define $G$ to be \textit{almost quasisimple} if $G$ has a unique quasisimple subnormal subgroup, which forms the layer $E(G)$ of $G$, and the quotient $G/F(G)$ of $G$ by its Fitting subgroup $F(G)$ is almost simple.
%$G/Z(G)$ is an almost simple group, or equivalently if it has a unique quasisimple subnormal subgroup, which forms the layer $E(G)$ of $G$.
A primitive affine group $H=GV$ has base size 2 if and only if $G \leq \mathrm{GL}(V)$ has a regular orbit in its irreducible action on $V$.
In classifying $G$ with a regular orbit, one would naturally use Aschbacher's theorem \cite{asch_thm} to determine the possibilities for irreducible subgroups of $\mathrm{GL}(V)$. In this paper, we investigate the case where $G$ is a member of the $\mathcal{C}_9$ class of Aschbacher's theorem. Slightly more broadly than this, we consider $G \leq \Gamma \mathrm{L}(V)$ almost quasisimple such that $E(G)$ acts absolutely irreducibly on $V$. 

%An open case in the classification of primitive groups $H$ with $b(H)=2$ is where $H=GV$ is an affine type group, $(|G|,|V|)>1$ and $G \leq \mathrm{GL}(V)$ is almost quasisimple, and whose quasisimple layer $E(G)$ acts absolutely irreducibly on $V$. 
The pairs $(G,V)$ where $G$ has a regular orbit on $V$ have been classified for $E(G)/Z(E(G))$ a sporadic or alternating group \cite{MR3893366,MR3500766}, and for $G$ of Lie type with $(|G|,|V|)=1$ by the aforementioned proof of the $k(GV)$-conjecture. The authors of \cite{ourpaper} showed that $b(G)\leq 6$ for $G \leq \gl(V)$ with $E(G)$ quasisimple, as long as $V$ is not the natural module for $E(G)$.
Moreover, Guralnick and Lawther \cite{bigpaper} classified $(G,V)$ where $G$ is a simple algebraic group over an algebraically closed field of characteristic $p>0$ that has a regular orbit on the irreducible $G$-module $V$. They also determine the generic stabilisers in each case. Their methods, which provide the foundation for the techniques in this paper, rely heavily on detailed analyses of highest weight representations of these simple algebraic groups.

This paper is the second in a series of three, which analyse base sizes of pairs $(G,V)$, where $G\leq \Gamma \mathrm{L}(V)$ is almost quasisimple and  $E(G)$ is a group of Lie type that acts absolutely irreducibly on $V$, with $(|G|,|V|)>1$. The first paper \cite{crosscharpaper} dealt with the cross-characteristic representations of $G$, while the present paper considers the case where $E(G)/Z(E(G)) \cong \mathrm{PSL}_n(q)$ and $V$ is an absolutely irreducible module for $E(G)$ in defining characteristic.
The final paper in the series will consider the remaining almost quasisimple groups $G \leq \Gamma\mathrm{L}(V)$  of Lie type in defining characteristic.

%In this paper, as with the others, we will assume that $E(G)/Z(E(G)) \not\cong \mathrm{Alt}_m$, since this has been covered in \cite{MR3500766}.
We say that two representations $\rho_1, \rho_2$ of $E(G)$ (and their corresponding modules) are quasiequivalent if there exists $g \in \mathrm{Aut}(E(G))$ such that $\rho_1$ is equivalent to $g \rho_2$. For an almost quasisimple group of Lie type $G=G(q)$, and a natural number $i$, we define $\ep_i=1$ if $i \mid q$, and $\ep_i=0$ otherwise.

The main theorem of this work is as follows.
\begin{theorem}
\label{main}
Let $V=V_d(q_0)$ be a $d$-dimensional vector space over $\mathbb{F}_{q_0}$ with $q_0=p^f$, and let $G \leq \Gamma \mathrm{L}(V)$ be an almost quasisimple group such that $E(G)/Z(E(G)) \cong  \mathrm{PSL}_n(q)$ with $n\geq 2$ and $q=p^{e}$. Suppose that the restriction of $V$ to $E(G)$ is an absolutely irreducible module $V(\lambda)$ of highest weight $\lambda$. Set $k=f/e$. Either $G$ has a regular orbit on $V$, or up to quasiequivalence, $\lambda$ and $n$ appear in Table \ref{rem1} if $k=1$ and Table \ref{rem2} otherwise.
\end{theorem}
%\begin{table}[!htbp]
%\centering
%\begin{tabular}{cccc}
%\toprule 
%$\lambda$ & Dimension & Notes & $n$\\ 
%\midrule
%$\lambda_1$ &$n$ &$b(G)= n$ & $[2,\infty)$\\
%$\lambda_2$ &$\binom{n}{2}$ &$b(G) \geq 3$ &$[3,\infty)$\\
%$2\lambda_1$ &$\binom{n+1}{2}$&$b(G)\geq 2$  & $[2,\infty)$\\
%$3\lam1$ & 4 &$b(G)\geq 2$  & $2$\\
%$\lam1+\lam2$,$q=3^e$ &16& $b(G)=2$ &  4\\
%$\lambda_3$ &$\binom{n}{3}$& $b(G)=2$ & $[7,8]$ \\
% &20&  $2\leq b(G)\leq 3$& $6$ \\ 
% $\lam1+\lam {n-1}$ &$n^2-1-\ep_n$&$b(G)=2$ & $[4,\infty)$\\
%$(p^a +1 )\lambda_1$ &$n^2$& $b(G)\geq 2$ &$[2,\infty)$  \\ 
%$p^a\lambda_1+ \lambda_{n-1}$ &$n^2$&$b(G)\geq 2$ & $[3,\infty)$\\ 
%\bottomrule
%\end{tabular}
%\caption{A description of modules $V(\lambda)$ for $G$ with no regular orbit. \label{rem}
%}
%\end{table}

{
\renewcommand{\arraystretch}{1.5}
\begin{table}[!htbp]
\centering
\begin{tabular}{@{}llll@{}}
	\toprule
	$\lambda$ & Dimension & Notes & $n$ \\ \midrule
	$\lambda_1$ & $n$ & $b(G)\leq n+1^*$ & $[2,\infty)$ \\
	$2\lambda_1$, $p\geq 3$ & $\binom{n+1}{2}$ & $2  \leq b(G) \leq 3 $ & $[2,\infty)$ \\
	$\lambda_2$ & $\binom{n}{2}$ & $b(G) = 3 $ & $[7,\infty)$ \\
	&  & $3 \leq b(G) \leq 4$ & $[5,6]$ \\
	&  & $3 \leq b(G) \leq 5$ & $4$ \\
	$3\lam1$, $p\geq 5$ & 4 & $b(G)= 2^*$ & $2$ \\
	$\lambda_3$ & $\binom{n}{3}$ & $b(G)=2$ & $[7,8]$ \\
	& 20 & $2\leq b(G)\leq 4^*$ & $6$ \\
	$\lam1+\lam {n-1}$ & $n^2-1-\ep_n$ & $b(G)=2$ & $[3,\infty)$ \\
	$\lam1+\lam2$, $q=3^e$ & 16 & $b(G)=2$ & 4 \\
	$(p^a +1 )\lambda_1$ & $n^2$ & $b(G)= 2^*$ & $[2,\infty)$ \\
	$\lambda_1+p^a \lambda_{n-1}$ & $n^2$ & $b(G)= 2^*$ & $[3,\infty)$ \\
	\midrule 
	$2\lam2$, $p\geq 3$ &$20-\epsilon_3$ & $1\leq b(G)\leq 2^*$ & 4\\
	$\lam3$ &84& $1\leq b(G) \leq 2^*$ & 9\\
	$3\lam1$, $p\geq 5$ &10 &  $1\leq b(G) \leq 2^*$ & 3\\
	$\lam4$ &70& $1\leq b(G)\leq 2^*$ & 8\\
 \bottomrule
\end{tabular}
\caption{A description of modules $V(\lambda)$, with $k=1$, where $G$ may not have a regular orbit. \label{rem1}}
\begin{tabular}{lllll}
	\toprule 
	$\lambda$ & Dimension &$k$& Notes & $n$\\ 
	\midrule
	$\lambda_1$ &$n$ &$k\geq 2$ &$b(G)\leq  \lceil n/k \rceil +1^*$ & $[2,\infty)$\\
	$2\lambda_1$, $p\geq 3$ &$\binom{n+1}{2}$&$2$&$1  \leq b(G) \leq 2 $  & $[2,\infty)$\\
			$\lambda_2$ &$\binom{n}{2}$ &$2$ &$b(G) = 2$ &$[5,\infty)$\\
& &$2$ &$2 \leq b(G) \leq 3$ &$4$\\
& &$3$ &$1 \leq b(G) \leq 2$ &$[4,6]$\\
&&$4$ & $1\leq b(G) \leq 2$ & $4$\\
	$\lambda_3$ &20&$2,3$&  $2/k\leq b(G)\leq \lceil4/k \rceil^*$& $6$ \\ 
	$\lam1+\lam l$ & $8-\ep_3$ &2& $1\leq b(G)\leq 2^*$ & 3\\
	$(p^{\frac{e}{2}} +1 )\lambda_1$ &$n^2$&$\frac{1}{2}$& $2 \leq b(G)\leq 3$ &$[2,\infty)$  \\
	$(p^{e/2} +1 )\lambda_2$ &$\binom{n}{2}^2$&$\frac{1}{2}$& $1 \leq b(G) \leq  2^*$ &$4$  \\ 
		$2(p^{e/2} +1 )\lambda_1$, $p\geq 3$ &$\binom{n+1}{2}^2$&$\frac{1}{2}$& $1 \leq b(G) \leq  2$ &$2$  \\ 
	 $(p^{\frac{e}{3}} +p^{\frac{2e}{3}}+1 )\lambda_1$&$n^3$  &$\frac{1}{3}$& $1 \leq b(G) \leq  2$ &$3$  \\
	&&$\frac{1}{3}$& $b(G)= 2$ &$2$  \\ 
	$(p^{\frac{e}{4}} +p^{\frac{2e}{4}}+p^{\frac{3e}{4}}+1 )\lambda_1$	 &$n^4$&$\frac{1}{4}$& $1\leq b(G) \leq 2$ &$2$  \\ 
	\bottomrule
\end{tabular}
\caption{A description of modules $V(\lambda)$, with $k\neq 1$, where $G$ may have no regular orbit. \label{rem2}
}
\end{table}
}

There are several corollaries that we can deduce from Theorem \ref{main}.
\begin{corollary}
Let $G$, $V$ be as in Theorem \ref{main}. If $n\geq 7$ and $\log_q|V| > n^2+n$, then $G$ has a regular orbit on $V$.
\end{corollary}

The authors of \cite{ourpaper} show that if $G$ as in Theorem \ref{main} is contained in $\gl(V)$, then $b(G)\leq 6$, unless $V$ is the natural module for $E(G)$. The next corollary gives an improvement to this result.
\begin{corollary}
Let $G$, $V$ be as in Theorem \ref{main}. Then either $V$ is the natural module for $G$, or $b(G) \leq 5$.
\end{corollary}

If $q=q_0$, that is, $k=1$,  and $G$ is quasisimple, we can say precisely when $G$ has a regular orbit on $V$.
\begin{corollary}
Let $G$, $V$ be as in Theorem \ref{main}, and further suppose $G$ is quasisimple. If $q=q_0$, then either $G$ has a regular orbit on $V$, or $b(G)>1$ and $V$ appears in Table \ref{rem1}.
\end{corollary}
We now give some remarks on the content of Tables \ref{rem1} and \ref{rem2}.
\begin{remark}
\begin{enumerate}
\item The highest weights in Tables \ref{rem1} and \ref{rem2} with coefficients including $p^a$ have $1\leq a \leq e-1$, and $\lambda_i$ denotes the $i$th fundamental dominant weight of the root system associated with $G$.
\item The groups $\mathrm{PSL}_2(4)$ and $\mathrm{PSL}_2(5)$ are isomorphic, and the proof of Theorem \ref{main} for these groups is given in \cite{crosscharpaper}.
\item The rows in Tables \ref{rem1} and \ref{rem2} for $\lambda=\lam1$ are asterisked because if $G$ contains no field automorphisms, then $b(G)=\lceil n/k \rceil +c$, where $c=1$ if $i = (k,n)>1$ and $G$ contains scalars in $\mathbb{F}_{q^i}^{\times}\setminus \F_q^\times$, and $c=0$ otherwise. If $G$ contains field automorphisms, then $b(G)\leq \lceil n/k \rceil +1$, with equality if $G$ contains scalars in $\mathbb{F}_{q^i}^{\times}\setminus \F_q^\times$ or $\log_q|G|>kn \lceil n/k \rceil$.
%
%The quantity $c$ in the row for $\lambda =\lam1$ in Table \ref{rem2} is equal to one if $i=(k,n)>1$ and $G$ contains scalars in $\mathbb{F}_{q^i}^\times \setminus \F_q^\times$ and is zero otherwise. 
\item Notice that if $G$ has no regular orbit on $V$, then the possibilities for the parameter $k$ are very restricted. Indeed, by Tables \ref{rem1} and \ref{rem2}, excluding the natural module for $E(G)$, the parameter $k$ is either an integer $k \in [1,3]$, or can be written as $k=1/c$, with $c \in [2,4]$. 
\item The entry for $\lambda=3\lam1$ and $n=2$ is asterisked because by Proposition \ref{3l1}, $G$ has a regular orbit on $V$ for $G$ a subgroup of index at least 3 in $\mathrm{GL}_2(q)/K$, where $K$ is the kernel of the action of $\gl_2(q)$. Otherwise, $b(G)\leq 2$.

%\item The upper bound on $b(G)$ in  the corresponding to $\lambda = \lam 2$, $n\in [3,6]$ originates from \cite{ourpaper}.
\item The entries for $\lambda  =\lam 3$ with $n=6$ in Tables \ref{rem1} and \ref{rem2} are asterisked because by Propositions \ref{l3} and \ref{k=1_leftovers}, $2/k \leq b(G) \leq (3+\delta)/k$, where $\delta$ is 1 if $G$ contains a graph automorphism and zero otherwise. 
\item The row for $\lambda= \lam1+\lam l$ in Table \ref{rem2} is asterisked because by Proposition \ref{adj_ext}, there is a regular orbit of $G$ on $V$ if $G$ does not contain graph automorphisms.
%\item Similarly, the entry in Table \ref{rem2} for $\lambda= (p^{e/2} +1 )\lambda_2$ and $n=4$ is asterisked because there is a regular orbit of $G$ on $V$ if $G$ does not contain graph or graph-field automorphisms.
% $= \langle \nu_3 I\rangle$ if there is an element $\nu_3$ of order 3 in $\mathbb{F}_q$, and $K$ is trivial otherwise.
\item If $\lambda = (p^a+1)\lam1$ or $\lam1+p^a\lam {n-1}$, then by Proposition \ref{tensornatural}, there is no regular orbit of $G$ on $V$ unless $(a,e)=1$, $(n,p)=(2,2)$ or $(2,3)$ and $G= \mathrm{SL}_2(q)/K$, where $K$ is the kernel of the action of $\mathrm{SL}_2(q)$.
\item The entries in the lower section of Table \ref{rem1} are asterisked because if $G$ is quasisimple, then $G$ has a regular orbit on $V$. Otherwise, we determine that $b(G)\leq 2$.
\item Each row in Table \ref{rem2} with $\lambda = \left( \sum_{i=0}^{m-1} p^{ie/m}\right) \lam1$ for some $m \in \{2,3,4\}$ describes an $\mathrm{SL}_n(q)$-module over a subfield $\mathbb{F}_{q^{1/m}} \subset \mathbb{F}_q$. In each of these cases, $k=1/m$. The construction of such modules is described in Section \ref{subfields}. 
\item The papers \cite{MR3893366,MR3500766,goodwin1,goodwin2,MR1829482} mentioned in the introduction each assume that $V$ (as in Theorem \ref{main}) is a faithful $d_0$-dimensional $\F_{r_0}G$-module, with $r_0$ prime such that $E(G)$ acts irreducibly, but not necessarily absolutely irreducibly, on $V$. We can reconcile this with our study of absolutely irreducible modules for $E(G)$ over finite fields of arbitrary prime power order, following \cite[\S3]{goodwin1}.  Define $k = End_{F_{r_0}G}(V)$, $K = End_{F_{r_0}E(G)}(V)$, $t = |K:k|$ and $d = dim_K(V)$. Then $E(G) \le GL_d(K)$ is absolutely irreducible, and $G \le GL_d(K)\langle \phi \rangle \leq \Gamma \mathrm{L}_d(K)$, where $\phi$ is a field automorphism of order $t$.
\end{enumerate}
\end{remark}

Our main approach for proving Theorem \ref{main} relies on the simple observation that if $G$ has no regular orbit on the absolutely irreducible $E(G)$-module $V$, then every vector $v \in V$ is fixed by a non-trivial prime order element in $G$. For each $V$, we then set out to give a proof by contradiction by showing that $|V|$ exceeds sum of the sizes of the fixed point spaces $C_V(x)$ of prime order elements $x \in G\setminus 1$  (this is formalised in Proposition \ref{tools}). This technique requires a method of determining reasonably precise upper bounds for the dimensions of fixed point spaces of elements of $G$.
 
To compute these upper bounds, we adopt the techniques pioneered by Guralnick and Lawther in \cite{bigpaper}, and also based on the work of Kenneally \cite{kenneally}.
Their methods rely heavily on weight theory in simple algebraic groups of Lie type. They obtain upper bounds on eigenspace dimensions by defining a set of equivalence relations on the weights of the representation $\rho$ corresponding to $V$. These equivalence relations are derived from subsystems $\Psi$ of the root system $\Phi$ associated to a simple algebraic group $\overline{G}$. Larger subsystems $\Psi$ generally give tighter upper bounds, but apply to fewer conjugacy classes due to our underlying assumptions. The technique is described in more detail in Section \ref{prelim}.

The techniques implemented by Guralnick and Lawther provide a starting point for the proof of Theorem \ref{main}. However, our application to finite groups presents additional challenges. We are also required to sum over conjugacy classes of our group, and this often means that more delicate upper bounds on dimensions of fixed point spaces are required. In some cases, this method does not work at all, and different considerations are needed. 
 
\begin{remark}
There are striking examples where Guralnick and Lawther \cite{bigpaper} show that $\overline{G}$, a simple algebraic group of type $A_l$, has no regular orbit on an irreducible $\overline{\F}_q\overline{G}$-module $V=V(\lambda)$, but according to Theorem \ref{main}, the corresponding finite group $G(q)$ has a regular orbit on $V$ realised over $\F_q$, where $q=p^e$. 	
For example, when $l=1$, $V=V((p^a+1)\lam1)$ and $(a,e)=1$, then \cite[Proposition 3.1.8]{bigpaper} asserts that there is no regular orbit under the action of $\mathrm{SL}_2(\overf_q)$, but we prove in Proposition \ref{tensornatural} that there is a regular orbit on $V$ under the action of $\mathrm{SL}_2(q)$ when $p=2$ or 3.
\end{remark}

%Although many of our computations are based on those in \cite{bigpaper}, our application to finite groups presents additional challenges. Guralnick and Lawther reduce  the existence of a regular orbit to showing that the co-dimension of $C_V(x)$ is greater than $\dim x^{\overline{G}}$ for all prime order $x \in \overline{G}$. However, while this condition is necessary for our application to the finite groups, we also need to sum over all conjugacy classes of elements of prime order, meaning that we are often required to compute more precise upper bounds to complete a proof by contradiction.

% If $b(G) = 1$, we say that $G$ has a \textit{regular orbit} on $\Omega$. 
%In this work, we only include the proofs for $soc(H/Z(H)) \cong \mathrm{PSL}_n(q)$.
The rest of this work is set out as follows. In Section \ref{prelim}, we present some preliminary results, which will provide the machinery for the bulk of the proofs in Sections \ref{prest} and \ref{tensor} which will prove Theorem \ref{main}. We also include some explanation on the techniques of the proofs and a guide on how information is presented in tables preceding each calculation.
The proof of Theorem \ref{main} is split across four sections. 
In the notation of Theorem \ref{main}, the modules $V=V(\lambda)$ where $k=1$ are dealt with in Section \ref{prest} for $\lambda$ $p$-restricted and Section \ref{tensor} otherwise.
 Section \ref{subfields} deals with absolutely irreducible modules with $k<1$ that are not realised over a proper subfield of $\F_{q^k}$. Finally, Section \ref{ext_section} completes the proof of Theorem \ref{main} by considering field extensions of the modules discussed in Sections \ref{prest}, \ref{tensor} and \ref{subfields}.
%The techniques used in Sections \ref{prest} and \ref{tensor} to prove Theorem \ref{main} are adapted from those used in \cite[Chapters 2, 3]{bigpaper}.
\section{Background}
\label{prelim}

Let $\overline{G}$ be a simple algebraic group over $\overline{\mathbb{F}_p}$, $p$ prime. 
%Denote the rank of $\overline{G}$ by $l=n-1$. 
Let $T \leq  \overline{G}$ be a fixed maximal torus of $\overline{G}$, and $\Phi$ be the root system of $\overline{G}$ with respect to $T$, with base $\Delta = \{\alpha_1, \dots , \alpha_l \}$ of simple roots, and corresponding fundamental dominant weights $\lambda_1, \dots , \lambda_l$. We define a partial ordering on weights $\lambda, \mu$ by saying $\mu \preceq \lambda$ if and only if $ \lambda-\mu$ is a non-negative linear combination of simple roots.
By a theorem of Chevalley, the irreducible modules of $\overline{G}$ in defining characteristic $p$ are characterised by their unique highest weights $\lambda$ under $\preceq$ and conversely, every dominant weight $\lambda$ gives rise to an irreducible $\overline{G}$-module, denoted $V(\lambda)$.  

If the weight $\lambda =a_1 \lam1 + \dots a_l \lam l$ with $0 \leq a_i < p$, then we say $\lambda$ is \textit{$p$-restricted}. Given a module $V=V(\lambda)$ with corresponding representation $\rho: \overline{G} \to \mathrm{GL}(V)$ and a Frobenius endomorphism of $\overline{G}$ given by $\sigma: x \mapsto x^{(p^a)}$, we may define a new module $V(\lambda)^{(p^a)}$ with corresponding representation $\rho^{(p^a)}$ by setting $v\rho^{(p^a)}(g) = v\rho(\sigma(g))$. Notice that $V(\lambda)$ and $V(\lambda)^{(p^a)}$ are quasiequivalent. We will sometimes say that $V(\lambda)^{(p^a)}$ is the image of $V(\lambda)$ under a field twist.
The following celebrated theorem provides a method of representing each irreducible $\overline{\mathbb{F}}_p\overline{G}$-module as a tensor product of field twists of $p$-restricted modules.
\begin{theorem}[{\cite{MR0230728}}]
\label{steintens}
Let $\overline{G}$ be a simple algebraic group of simply connected type over $\overline{\mathbb{F}}_p$. If $\omega_0, \dots \omega_k$ are $p$-restricted weights, then 
\[
V(\omega_0+p \omega_1 + \dots +p^k \omega_k) \cong V(\omega_0) \otimes V(\omega_1)^{(p)} \otimes \dots \otimes V(\omega_n)^{(p^k)}
\]
\end{theorem}

Let $\sigma$ be a Frobenius endomorphism of $\overline{G}$, such that $\overline{G}_\sigma$ is an untwisted group of Lie type over $\mathbb{F}_q$.
%and $\mathrm{soc} (\overline{G}_\sigma/Z(\overline{G}_\sigma)) = E(G)/Z(E(G))$.
Steinberg proved that the restriction of $V(\lambda)$  for $\lambda \in \{a_1 \lam1 + \dots a_l \lam l \mid 0\leq a_i \leq q-1 \}$ to the finite group $\overline{G}_\sigma$ gives a complete set of representatives of irreducible $\overline{\mathbb{F}}_p\overline{G}_\sigma$-modules up to equivalence. 

So if $G$ is an almost quasisimple group such that $E(G) =\overline{G}_\sigma'$, the finite absolutely irreducible $\mathbb{F}_qE(G)$-modules are the realisations of the $V(\lambda)$ over $\mathbb{F}_q$.

%We will assume that $n\geq 3$, since the case where $n=2$ was done in separate, previous work (if not satisfactory, will do again separately). %TODONE n=2?

Our primary method for proving Theorem \ref{main} comes from the observation that if $G$ has no regular orbit on $V$, then $V$ is the union of the fixed points spaces $C_V(g)$ for $g\in G\setminus F(G)$. Since conjugates in $G$ have conjugate fixed point spaces, we find that
\begin{equation}
\label{ogeqn}
|V| \leq \sum_{x\in \mathcal{X}} |x^G||C_V(x)|, 
\end{equation}
where $\mathcal{X}$ is a set of non-identity conjugacy class representatives of $G$.

We say that an element $g\in G$ is of \textit{projective prime order} if the coset $gF(G)$ has prime order in $G/F(G)$. Note that this implies that there exists a scalar $\beta \in \overline{\mathbb{F}}_q$ such that $\beta g$ is of prime order. We call the order of $\beta g$ the \textit{projective order} of $g$. For every non-central element of $g \in G$, there exists $j \in \mathbb{N}$ such that $g^j$ is of projective prime order. Moreover, $C_V(g) \subseteq C_V(g^j)$. Therefore, we may instead sum over classes of projective prime order elements in \eqref{ogeqn}. 

\subsection{Properties of prime order elements in $\mathrm{Aut}(\mathrm{PSL}_n(q))$}
From now on, let $G$ be an almost quasisimple group with $E(G)/Z(E(G)) \cong \mathrm{PSL}_n(q)$, with $q=p^e$ for $p$ prime; let $\overline{G} = \mathrm{SL}_n(\overline{\F}_p)$ with $l=n-1$ and let $\sigma$ be a Frobenius endomorphism of $\overline{G}$ such that $\overline{G}_\sigma = \mathrm{SL}_n(q)$.

For $G_0$ a simple group, and non-identity $x \in {\rm Aut}(G_0)$, define $\a(x)$ to be the minimal number of $G_0$-conjugates of $x$ needed to generate the group $\langle G_0, x\rangle$. For a projective prime order element $g$ of an almost quasisimple group $G$, define $\a(g)=\a(gF(G))$, and also define
\[
\a(G) = {\rm max}\,\{ \a(x)\,\mid\,1\ne x \in G/F(G) \} .
\]
The following results give upper bounds for $\a(G)$ when $\mathrm{soc}(G/F(G)) \cong \mathrm{PSL}_n(q)$. 
%Recall that we do not consider alternating groups, since they are handled in \cite{MR3500766}.
%\begin{proposition}[{\cite[Lemma 3.1]{gs}}]
%\label{L2lemma} 
%Let $H$ be an almost simple  group with $\mathrm{soc}(H) \cong \mathrm{PSL}_2(q)$, and let $x\in H$ have prime order $r$. Then $\alpha(x) \leq 3$ unless:
%\begin{enumerate}[i)]
%\item $x$ is an involutory field automorphism and $\alpha(x) \leq 4$, except $\alpha(x)=5$ for $q=9$, or 
%\item $q=5$, $x$ is an involutory diagonal automorphism, and $\alpha(x)=4$.
%\end{enumerate}
%Moreover, if the order of $x$ is odd, then $\alpha(x) = 2$, unless $q=9$, $r=3$ and $\alpha(x) = 3$.
%\end{proposition}

%\begin{proposition}[{\cite[Lemma 3.1]{gs}}]
%\label{L2lemma} 
%Let $H$ be an almost simple  group with $\mathrm{soc}(H) \cong \mathrm{PSL}_2(q)$, and let $x\in H$ have prime order $r$. Then $\alpha(x) \leq 3$ unless $x$ is an involutory field automorphism and $\alpha(x) \leq 4$.
%Moreover, if the order of $x$ is odd, then $\alpha(x) = 2$.
%\end{proposition}

\begin{proposition}[{\cite[Lemmas 3.1, 3.2, Theorem 4.1]{gs}}]\label{alphas} Let $H$ be an almost simple group with socle $H_0 = \mathrm{PSL}_n(q)$ for $n\geq 2$. Then for $x\in H\setminus\{1\}$, one of the following holds:
\begin{enumerate}
\item$\a(x) \leq n$,
\item $H_0 = \mathrm{PSL}_4(q)$ with $q\geq 3$, and $x$ is an involutory graph automorphism and $\alpha(x)\leq 6$,
\item $H_0 = \mathrm{PSL}_3(q)$, $x$ is an involutory graph-field automorphism and $\alpha(x)\leq 4$,
\item $H_0 = \mathrm{PSL}_2(q)$, with $q\neq 9$, $x$ is an involution and $\a(x) \leq 3$, except that $\a(x) \leq 4$ if $x$ is an involutory field automorphism or if $q=5$ and $x$ is an involutory diagonal automorphism.
\item $H_0 = \mathrm{PSL}_2(9)$, $x$ has order 2 or 3 and $\alpha(x) \leq 3$, or $x$ is an involutory field automorphism and $\alpha(x) = 5$.
\item  $H_0 = \mathrm{PSL}_4(2)$, $x$ is a graph automorphism and $\alpha(x) =7$.
\end{enumerate}
\end{proposition}

For a group $G$, denote by $i_r(G)$ the number of elements of order $r$ in $G$.
\begin{proposition}[{\cite[Prop. 1.3]{MR1922740}}]
\label{invols}
Let $\overline{G}$ be a simply connected simple algebraic group over $\overline{\mathbb{F}}_p$ with associated root system $\Phi$. Let $\sigma$ be a Frobenius endomorphism of $\overline{G}$ such that $G_0=\mathrm{soc}(\overline{G}{_\sigma}/Z(\overline{G}{_\sigma}))$ is a finite simple group of Lie type over $\mathbb{F}_q$.  Assume that $G_0$ is not of type $^2F_4$, $^2G_2$ or $^2B_2$. Then
\begin{enumerate}
\item $i_2(\mathrm{Aut}(G_0)) <2(q^{N_2} + q^{N_2-1})$, where $N_2 = \dim \overline{G} - \frac{1}{2}|\Phi|$, and 
\item $i_3(\mathrm{Aut}(G_0)) <2(q^{N_3} + q^{N_3-1})$, where $N_3 = \dim \overline{G} - \frac{1}{3}|\Phi|$.
\end{enumerate}
\end{proposition}

Let $V=V_d(q^k)$ be an absolutely irreducible module for $\mathrm{SL}_n(q)$, where $q=p^e$ for some prime $p$. We observe that by \cite[Proposition 5.4.6]{KL}, $k$ can be written as $k=a/b$, where $a,b \in \mathbb{N}$, $b \mid e$ and $a \geq 1$.
 \begin{proposition}
\label{field}
Let $V=V_d(q^k)$ be a $d$-dimensional vector space over $\mathbb{F}_{q^k}$, and assume that $G \leq \Gamma\mathrm{L}(V)$ is almost quasisimple, with $E(G)/Z(E(G)) \cong \mathrm{PSL}_n(q)$. Further suppose that the restriction of $V$ to $E(G)$ is absolutely irreducible, and that $d \geq n^2/k$. Set $H=G/F(G)$, for $x\in G$ let $\bar{x}=xF(G)$,  and let $\mathcal{F}$ denote a set of conjugacy class representatives of projective prime order elements of $G$ that induce field automorphisms of $E(G)/Z(E(G))$.
Then
%Let $G$ be an almost quasisimple with $E(G)/Z(E(G)) \cong \mathrm{PSL}_n(q^k)$, with $q=p^e$. Let $H=G/F(G)$ ,and suppose $V = V_d(q)$ is an irreducible $\mathbb{F}_qG$-module in defining characteristic. Let $\mathcal{F}$ denote a set of conjugacy class representatives of the field automorphisms of prime order in $\mathrm{Aut}(H)$. Then 
\[
\sum_{x\in \mathcal{F}} |\bar{x}^H| |C_V(x)| \leq
  \log (\log_2 q +2) q^{n^2/2+\frac{k}{2}(d+\zeta d^{1/2}) } ,
\]
where $\zeta=1$ if there exists $x \in \mathcal{F}$ that acts linearly on $V$, and $\zeta=0$ otherwise.
\end{proposition}
\begin{proof}
%Let $\phi : g \mapsto g^p$ be a Frobenius automorphism, so that $|\langle \phi \rangle | = e$. 

Suppose first that all $x \in \mathcal{F}$ of projective prime order $r$ acts as a field automorphism on $V$. For such $x$ there exists a basis of $V$ such that $x$ fixes all vectors with coefficients lying in $\F_{q^{1/r}}$ and no other vectors.
%Suppose first that $k\geq 1$.
 This observation, together with \cite[Proposition 4.9.1(d)]{GLS} and \cite[Table B.3]{bg}, give
\[
\sum_{x\in \mathcal{F}} |\bar{x}^H| |C_V(x)| \leq \sum_{r \mid e}  \frac{|\mathrm{PGL}_n(q)|}{|\mathrm{PGL}_n(q^{1/r})|} q^{kd/r} < \sum_{r \mid e} \frac{(q^{1/r}-1)}{(q-1)} q^{n^2(1-1/r)+kd/r}
\]
%where the last inequality holds since $\frac{1}{2(q-1)} q^{n^2} < |\mathrm{PGL}_n(q)| < \frac{1}{q-1} q^{n^2}$. 
Now, $\frac{(q^{1/r}-1)}{(q-1)} <1$ for all $q$ and $r$. We also have $n^2(1-1/r)+kd/r \leq  n^2/2 +kd/2$ if $d \geq n^2/k$. So
\[
\sum_{x\in \mathcal{F}} |\bar{x}^H| |C_V(x)| \leq  \log (\log_2 q +2) q^{n^2/2+kd/2}, 
\]
since the number of distinct prime divisors of $N \in \mathbb{N}$ is at most $\log(N+2)$.

Now suppose that there exists $x \in \mathcal{F}$ of projective prime order $r$ that acts linearly on $V$.
By \cite[Proposition 5.4.6]{KL}, we can write $
V \otimes \overline{\mathbb{F}}_q \cong M \otimes M^{(q)}\otimes \dots \otimes M^{(q^{r-1})}$ for some irreducible $\mathrm{SL}_n(q)$ module $M$. Then $x$ permutes the tensor factors of $V$ cyclically by one place.

%Write $m=t/r$ so that $x: a \mapsto a^{(q^{m})}$ is a field automorphism  
%which cyclically permutes the factors of $V$ by $m=t/r$ places. 
%%We may assume that a field automorphism of prime order is a map of the form $g \mapsto g^{p^{ke/r}}$ for $g \in G$. Let $ke/r = me+a$.
%Suppose $\varphi: a \mapsto a^{(q^{m})}$ is a field automorphism  
%which cyclically permutes the factors of $V$ by $m=t/r$ places. 
Let $B = \{e_{i_1} \otimes \dots \otimes e_{i_r} \mid 1\leq i_j \leq d\}$ be a basis of $V$, and set $d_0 = \dim M$ so that $|B| = d_0^r$. Now $\langle x \rangle$ has $d_0$ orbits of size 1 on $B$, consisting of basis vectors $e_{i_1} \otimes \dots \otimes e_{i_r}$ where $i_j = i_{j+1}$ for $1\leq j \leq r-1$. The remaining basis vectors lie in $(d_0^r-d_0)/r$ orbits of size $r$. Therefore, $x$ has one eigenspace of dimension $d_0+(d_0^r-d_0)/r \leq (d+d^{1/2})/2$, while the other eigenspaces of $x$ have dimension $(d_0^r-d_0)/r  \leq (d-d^{1/2})/2$.

%Now suppose that $\phi: x \mapsto x^{p^{me+a}}$ is a field automorphism of prime order.
%The 1-eigenspace $E_1(\phi)$ of our original field automorphism $\phi$ is the $\mathbb{F}_{p^a}$ span of the elements of $B$ fixed by $\varphi$, as well as $\sum_{i=0}^{t-1} \nu^{(p^a)^i}v_{i+1}$, where $\nu$ is $t$th root of unity in $\mathbb{F}_q$, and  $\{v_1, v_2, \dots ,v_t \}$ is a single orbit of $\langle \varphi \rangle $ with $\varphi(v_i) = v_{i+1}$ for $1\leq i \leq t-1$, except that $v(t)=v(1)$. Therefore, so $\log_q|E_1(\phi) |\leq \frac{a}{e} \frac{1}{t}(d_0^k+(t-1)d_0^{k/t})$. Since $gcd(m,k) = gcd(\lfloor k/r \rfloor, k) \leq k/r$, we have $t \geq r$ and thus
%\[
%\log_q|E_1(\phi) |\leq \frac{a}{e} \frac{1}{r}(d+(r-1)d^{k/r}) 
%\]
%
%The $\nu$-eigenspace of $\phi$ is the $\mathbb{F}_{p^a}$ span of vectors $\sum_{i=0}^{t-1} \nu^{(p^a)^i +i+1}v_{i+1}$, where $\{v_1, v_2, \dots ,v_t \}$ lie in a single orbit under $\langle \varphi \rangle $ and $\varphi(v_i) = v_{i+1}$ for $1\leq i \leq t-1$, except that $v(t)=v(1)$. Therefore $\log_q|E_\nu(\phi)| \leq\frac{a}{e} (d_0^k-d_0^{k/t})/t$, which is less than $\log_q|E_1(\phi)|$.
Therefore, repeating the same computation as for the previous case, we see that 
\begin{align*}
\sum_{x\in \mathcal{F}} |\bar{x}^H| |C_V(x)|\leq& \sum_{r \mid e}  \frac{1}{r-1}\frac{|\mathrm{PGL}_n(q)|}{|\mathrm{PGL}_n(q^{1/r})|} (q^{\frac{k}{2}(d+d^{1/2})}+ (r-1) q^{\frac{k}{2}(d-d^{1/2})})< \sum_{r \mid e}2 \frac{(q^{1/r}-1)}{(q-1)} q^{\frac{1}{2}n^2 +\frac{k}{2}(d+d^{1/2})}\\
&<   \log (\log_2 q +2) q^{n^2/2+k(d+d^{1/2})/2 },
\end{align*}
as required.
\end{proof} 

Let $\tau$ be a graph automorphism of $\mathrm{SL}_n(q)$, and let $V=V_d(q^k)$ be an absolutely irreducible $\F_{q^k}\mathrm{SL}_n(q)$-module for some $k \in \mathbb{N}$. We say that $\tau$ \textit{preserves} $V$ if it acts as a linear transformation on $V$. Moreover, if $V=V(\lambda)$ is an absolutely irreducible highest weight module for $\mathrm{SL}_n(q)$, then $\tau$ preserves $V$ if and only if the Dynkin diagram automorphism induced by $\tau$ preserves $\lambda$.

\begin{proposition}
	\label{graphfix}
	Let $\tau \in \mathrm{Aut}(\mathrm{SL}_n(q))$, with $q=p^e$, be an involutory graph automorphism, preserving the $d$-dimensional irreducible $\mathrm{SL}_n(q)$ module $V=V(\lambda)$ over a finite field $\F_{q_0}$ of characteristic $p$. Define $S$ to be the set of all weights $\mu$ in $V(\lambda)$ fixed by the Dynkin diagram automorphism induced by $\tau$, and let $d=\dim V(\lambda)$. Then 
	\[
	\dim C_V(\tau) \leq  \frac{d}{2} + \frac{1}{2}\sum_{\mu \in S} \dim V_{\mu}.
	\] 
	Furthermore, if $\tau x$ is an involutory graph-field automorphism, then $|C_V(\tau x)| \leq q_0^{\frac{1}{2}(d+d^{1/2})}$ if $\tau x$ acts linearly on $V$, and $|C_V(\tau x)| \leq q_0^{d/2}$ otherwise.
\end{proposition}
\begin{proof}
	Let $\overline{V} = V \otimes \overline{\mathbb{F}}_q$. Now,  $\overline{V}$ can be written as a direct sum of weight spaces with respect to some maximal torus of $\overline{G}$.
	 If $\tau$ swaps two weight spaces $V_\nu$, $V_{\tau(\nu)}$ with $\nu \notin S$, then $\tau$ has a fixed point space on $V_\nu\oplus V_{\tau(\nu)}$ of dimension at most $ \frac{1}{2}(\dim V_\nu+\dim V_{\tau(\nu)})$. Therefore, 
	\[
	\dim C_V(\tau)\leq \sum_{\mu \in S} \dim V_{\mu}+ \frac{1}{2} \left(d-\sum_{\mu \in S} \dim V_{\mu}\right) =  \frac{d}{2} + \frac{1}{2}\sum_{\mu \in S} \dim V_{\mu}.
	\]
The result for graph-field automorphisms follows from the proof of Proposition \ref{field}.

% Therefore, by Proposition \ref{field}, $\log_q|C_V(\tau x)| \leq d/2$.
\end{proof}

 \begin{proposition}[{\cite[Lemma 2.8]{crosscharpaper}}]
 \label{graph}
 The number of graph automorphisms of order 2 in $\mathrm{Aut}(\mathrm{PSL}_n(q))$ for $n\geq 3$ is at most $2 q^{\frac{n^2-n}{2}-1}$, and the number of graph-field automorphisms of order 2 in $\mathrm{Aut}(\mathrm{PSL}_n(q))$ is less than $q^{(n^2-1)/2}$.
 \end{proposition}

\section{Techniques}
\label{techniques}

The following proposition gives the main results used to prove Theorem \ref{main}. For an almost quasisimple group $G$, let $G_s$ denote the set of elements of $G$ of projective prime order $s$, and $G_{s'}$ denote the set of elements of $G$ with projective prime order coprime to $s$. 
\begin{proposition}[{\cite[Proposition 3.1]{crosscharpaper}}]
	\label{tools}
	Let $G\leq \mathrm{GL}(V)$ be an almost quasisimple group, acting irreducibly on the $d$-dimensional module $V=V_d(r)$ over $\mathbb{F}_r$. Set $H=G/F(G)$ and let $\mathcal{P}$ be a set of conjugacy class representatives of elements of projective prime order in $G$. For $x \in G$, let $\bar{x} = xF(G) \in H$, and denote the order of $\bar{x}$ by $o(\bar{x})$. Also let $\overline{V} = V\otimes\overline{\mathbb{F}}_r$. Then the following statements hold.
	
	\begin{enumerate}
		\item \label{alphabound} For $\nu \in \overline{\mathbb{F}}_r$, the $\nu$-eigenspace $E_\nu(x)$ of $x \in G$ on $\overline{V}$  satisfies 
		\[ 
		\dim_{\overline{\mathbb{F}}_r}(E_\nu(x)) \leq \left\lfloor\dim_{\overline{\mathbb{F}}_r}(\overline{V}) \left( 1-\frac{1}{\a(x)}\right)\right\rfloor.
		\] 
	\end{enumerate}
	Further, if $G$ has no regular orbit on $V=V_d(r)$ with $r=s^e$ for $s$ prime, then:
	\begin{enumerate} \setcounter{enumi}{1}
		\item \label{eigsp1} \[
		|V| \leq \sum_{x\in \mathcal{P}} \sum_{\kappa \in \mathbb{F}_r} \frac{1}{o(\bar{x})-1} |\bar{x}^H| |C_V(\kappa x)|.
		\]
		\item \label{eigsp2}
		\[
		|V| \leq \sum_{x\in \mathcal{P} \cap G_{s'}} \frac{o(\bar{x} )}{o(\bar{x} )-1} |\bar{x} ^H|\mathrm{max}\{|C_V(\kappa x)| \mid \kappa \in \mathbb{F}_r^{\times}\}+ \sum_{x\in \mathcal{P} \cap G_{s}} \frac{1}{o(\bar{x} )-1}  |\bar{x}^H| |C_V(x)|.
		\]
		\item \label{qsgood} \[
		|V| \leq 2\sum_{x\in \mathcal{P} \cap G_{s'}} |\bar{x} ^H|\mathrm{max}\{|C_V(\kappa x)| \mid \kappa \in \mathbb{F}_r^{\times}\} + \sum_{x\in \mathcal{P} \cap G_{s}} \frac{1}{o(\bar{x})-1}  |\bar{x}^H| |C_V(x)|.
		\]
		\item \label{crude} \begin{equation}
		\label{crudeeqn}
		|V| = r^d \leq 2\sum_{x\in \mathcal{P} \cap G_{s'}} |\bar{x}^H| r^{\lfloor(1-1/\alpha(x))d\rfloor} + \sum_{x\in \mathcal{P} \cap G_{s}} \frac{1}{o(\bar{x})-1}  |\bar{x} ^H|r^{\lfloor(1-1/\alpha(x))d\rfloor},
		\end{equation}
		
		and for fixed $r$, if this inequality fails for a given $d$, then it fails for all $d_1 \geq d$.
	\end{enumerate}
\end{proposition}

\begin{lemma}
\label{extfield}
Let $G$, $H$ and $V = V_d(q_0)$ be as in Proposition \ref{tools}. If one of Proposition \ref{tools}\ref{eigsp1})--\ref{crude}) fails for $G\leq \Gamma \mathrm{L}(V)$, then $\mathbb{F}_{q_0^k}^\times \circ G$ has a regular orbit on $\hat{V} = V \otimes F_{q_0^k}$ for each integer $k\geq 1$.
\end{lemma}
\begin{proof}
If one of Proposition \ref{tools} \ref{eigsp1})--\ref{crude}) fails for $G$ acting on $V$ (proving that $G$ has a regular orbit on $V$), then in particular Proposition \ref{tools}\ref{eigsp1}) fails, since the right hand side of the inequality is at most the right hand sides of parts \ref{eigsp2}), \ref{qsgood}) and  \ref{crude}). Therefore 
\[
q_0^d \geq \sum_{x\in \mathcal{P}} \sum_{\kappa \in \overline{\mathbb{F}}_{q_0}} \frac{1}{o(x)-1} |xF(G)^H| q_0^{\dim C_{\overline{V}}(\kappa x)},
\]
 and multiplying both sides of the inequality by $q_0^{(k-1)d}$, we have 
\[
q_0^{kd} \geq \sum_{x\in \mathcal{P}} \sum_{\kappa \in \overline{\mathbb{F}}_{q_0}} \frac{1}{o(x)-1} |xF(G)^H| q_0^{\dim C_{\overline{V}}(\kappa x)+k(d-1)} > \sum_{x\in \mathcal{P}} \sum_{\kappa \in \overline{\mathbb{F}}_{q_0}} \frac{1}{o(x)-1} |xF(G)^H| q_0^{k \dim C_{\overline{V}}(\kappa x)}
\]
and the result follows.
\end{proof}
The next result presents an additional method of bounding $|C_V(g)|$ for $g$ of projective prime order in $G$. 
For $g \in G\leq \mathrm{GL}(V)$ with $V=V_d(q_0)$, let $\emax^V(g)$ (or just $\emax(g)$) denote the dimension of the largest eigenspace of $g$ on $\overline{V} = V \otimes \overline{\mathbb{F}_{q_0}}$. 

\begin{proposition}[{\cite[Lemma 3.7]{MR1639620}}]
\label{tensorcodim}
Let $V_1$ and $V_2$ be vector spaces over $\mathbb{F}_{q_0}$ of dimension $d_1$ and $d_2$ respectively. Assume that $g = g_1\otimes g_2 \in \mathrm{GL}(V_1) \otimes \mathrm{GL}(V_2)$ is an element of projective prime order that acts on $V=V_1\otimes V_2$. Then 
\[
\emax^V(g) \leq \mathrm{min} \left\{ d_2\emax^{\overline{V_1}}(g_1), d_1\emax^{\overline{V_2}}(g_2) \right\}.
\]
where $\overline{V_i} = V_i \otimes \overline{\mathbb{F}_{q_0}}$.
\end{proposition}
%\begin{proof}
%This follows from {\cite[Lemma 3.7]{MR1639620}} with the observation that any element of projective prime order is of the form $\lambda x$ for $\lambda \in  \overline{\mathbb{F}_{q_0}}$ and $x$ an element of prime order in $\mathrm{GL}(V_1) \otimes \mathrm{GL}(V_2)$.
%\end{proof}

The following lemma shows that computing the base size of  $V\otimes \mathbb{F}_{q_0^k}$ can enable us to bound (and sometimes determine) the base size of $V$.
\begin{lemma}
\label{fieldext}
Suppose $V$ is a $d$-dimensional vector space over $\mathbb{F}_r$, and let $G\leq \mathrm{GL}(V)$. 
If $G$ has a base of size $c$ on $V \otimes \mathbb{F}_{r^i}$, then $G$ has a base of size at most $ci$ on $V$.
\end{lemma}
\begin{proof}
Let $\{v_1, \dots v_c\}$ be a base for $G$ acting on $V \otimes \mathbb{F}_{r^i}$. Let $\eta_1, \dots \eta_i$ be a basis for $\mathbb{F}_{r^i}$ over $\mathbb{F}_r$. Write each $v_j$ as $v_j = \sum_{k=1}^i \eta_k w_{j,k}$, where $w_{j,k} \in V$. Then $\{w_{j,k}\mid 1\leq j \leq c, 1\leq k\leq i\}$ is a base for $G$, since any non-trivial $g\in G$ stabilising the set must also stabilise $\{v_1, \dots v_c\}$.
\end{proof}

\begin{proposition}
	\label{permsquare}
	Let $\{a_1, a_2, \dots a_t\}$ be a non-decreasing sequence of natural numbers. The permutation $\nu\in S_t$ that maximises $\sum_{i=1}^t a_i a_{\nu(i)}$ is the identity permutation.
\end{proposition}
\begin{proof}
	We proceed by induction on $t$. The results holds true for $t=2$ since $(a_1-a_2)^2 \geq 0$. Suppose the result holds for $t=k$, and suppose $t=k+1$. If $\nu$ can be written as a product of disjoint cycles of length less than $k+1$, then the result follows by the inductive hypothesis. Otherwise, $\nu$ is a $(k+1)$-cycle. Let $\eta$ be the permutation defined by $\eta(i) = \nu(i)$, except that $\eta(\nu^{-1}(k+1))=\nu(k+1)$, and $\eta(k+1)=k+1$. Then, by the inductive hypothesis and recalling that the sequence of $a_i$s is increasing, we have
\[
 \sum_{i=1}^{k} a_ia_{\eta(i)} \leq \left(\sum_{i=1}^{k}a_i^2 \right) + (a_{k+1}-a_{\nu^{-1}(k+1)})(a_{k+1}-a_{\nu(k+1)}),
	\]
	so that 
\[
 \sum_{i=1}^{k} a_ia_{\eta(i)} + a_{\nu^{-1}(k+1)}a_{k+1}+ a_{k+1}a_{\nu(k+1)}-a_{\nu^{-1}(k+1)}a_{\nu(k+1)}\leq \sum_{i=1}^{k+1}a_i^2.
\]
	The left-hand side of the final inequality is equal to $\sum_{i=1}^{k+1} a_ia_{\nu(i)}$, so the result follows.
\end{proof}
We use Proposition \ref{permsquare} to bound the dimension of fixed point spaces of semisimple elements acting on certain modules. For a finite dimensional vector space $V$, Let $V_t(s)$ denote the $t$-eigenspace of an element $s \in \gl (V)$.
\begin{proposition}
	\label{eigsp_from_permsquare}
	Let $G$ be almost quasisimple with $E(G)/Z(E(G))\cong \mathrm{PSL}_n(q)$, let $s \in G$ be semisimple of projective prime order with eigenvalues $t_1, t_2, \dots, t_m$ on the natural module $W$ for $E(G)$ over $\overline{\mathbb{F}}_q$, arranged so that the multiplicities $a_i$ of the $t_i$ are weakly decreasing. Suppose $t \in \overline{\mathbb{F}}_q$. Then:
	\begin{enumerate}
		\item if $V$ is the symmetric square of $W$, then $\dim V_t(s) \leq a_1+ \frac{1}{2}  \sum a_i^2$,
		\item if $V$ is the exterior square of $W$, then $\dim V_t(s) \leq \frac{1}{2} \sum a_i^2$, and
		\item if $V$ is the exterior cube of $W$, then $\dim V_t(s) \leq \frac{1}{6} n \sum a_i^2$.

	\end{enumerate}
\end{proposition}
\begin{proof}
We will prove the proposition by considering the action of $s$ on $\overline{V} = V \otimes \overf_q$ for the various $V$.
Throughout, let $\overline{W}=W \otimes \overf_q$, and denote by  $\overline{W}^{\otimes a}$ the tensor product of $a$ copies of $\overline{W}$.  Also let $\{e_i \mid 1\leq i \leq n\}$ be a basis of $\overline{W}$ consisting of eigenvectors of $s$.

First suppose that $V$ is the symmetric square of $W$. Notice that $\{e_i \otimes e_j \mid 1\leq i\leq j \leq n\}$ is a basis of $\overline{W}^{\otimes 2}$ comprising eigenvectors of $s$, and let
$B_1=\{ e_i \otimes e_j+e_j\otimes e_i \mid 1\leq i<j\leq n\}$ and $B_2=\{ e_i\otimes e_i \mid 1\leq i \leq n\}$. Notice that $B_1 \cup B_2$ is a basis of $\overline{V}$ consisting of eigenvectors for $s$. The number of elements of $B_1$ lying in the $t$-eigenspace $\overline{V}_t(s)$ of $s$ on $\overline{V}$ is at most $\frac{1}{2} \sum a_i^2$ by Proposition \ref{permsquare}. Moreover, the number of vectors in $B_2$ contained in $\overline{V}_t(s)$ is at most $a_1$ if $s$ has odd projective prime order, and $2a_1$ if $s$ has projective prime order $2$. Therefore, the result follows for elements of odd projective prime order. If $s$ is of projective prime order 2, it has two eigenvalues $\pm \zeta$ of multiplicities $a_1,a_2$. Therefore, the eigenspaces of $s$ on $\overline{V}$ have dimensions $a_1a_2$ and $\frac{1}{2}(a_1^2+a_2^2+a_1+a_2)$. Both of these quantities are less than $\dim V_t(s) \leq a_1+ \frac{1}{2} \sum a_i^2$, as required.

Now suppose $\overline{V}$ is the exterior cube of $\overline{W}$. Note that  $\{e_i\otimes e_j \otimes e_k \mid 1\leq i\leq j\leq k\leq n\}$ is a basis of  $\overline{W}^{\otimes 3}
$ comprising eigenvectors of $s$. We can write $\overline{V} = \overline{W}^{\otimes 3}/U$, where $U$ is the subspace of $\overline{W}^{\otimes 3}$ generated by simple tensors with a repeated factor. Notice that  $e_i\otimes e_j \otimes e_k  + U = e_{\sigma(i)}\otimes e_{\sigma(j)} \otimes e_{\sigma(k)} +U$ for all $\sigma \in S_3$, and that vectors of the form $e_i\otimes e_j \otimes e_k  + U$ form a complete set of eigenvectors for $s$ on  $\overline{V}$. Therefore, by Propositions \ref{permsquare} and \ref{tensorcodim}, 
$
\dim \overline{V}_t(s) \leq  \frac{1}{6} \dim  W^{\otimes 3}_t(s) \leq \frac{1}{6} n \sum a_i^2
$
as required. The proof for the exterior square of $W$ is similar.
\end{proof}

 \section{Proof of Theorem \ref{main},  I : First steps}
In this section, we focus on the case where $G$ and $V$ have the same underlying field, $\mathbb{F}_q$. We reduce the proof of Theorem \ref{main} here to a finite list of cases given in Tables \ref{fulllist} and \ref{comptensor}. Here, as well as in future sections, we will consider the realisation of absolutely irreducible $\F_q\mathrm{SL}_n(q)$-modules $V(\lambda)$ up to quasiequivalence, i.e., up to duality and images under field automorphisms of $\mathrm{SL}_n(q)$.

\begin{proposition}
\label{n^3}
Let $V=V_d(q)$ be a $d$-dimensional vector space over $\F_q$. 
Suppose $G\leq \Gamma \mathrm{L}(V)$ is an almost quasisimple group with $E(G)/Z(E(G)) \cong \mathrm{PSL}_n(q)$, for $n\geq 2$ and $q=p^e$, such that the restriction of $V$ to $E(G)$ is absolutely irreducible. Also assume that $E(G)/Z(E(G))$ is not isomorphic to an alternating group i.e., $(n,q) \neq (2,4)$ or $(2,5)$. If $d  \geq n^3$, then $G$ has a regular orbit on $V$.
\end{proposition}
\begin{proof}
Throughout, we will assume that $d\geq n^3$.
First suppose $n\geq 5$.
For $x \in G\setminus F(G)$ we have, by Propositions \ref{alphas} and  \ref{tools}\ref{alphabound}, $\dim C_V(x) \leq \lfloor \frac{n-1}{n} \dim V\rfloor $.
If $G$ has no regular orbit on $V$, then by Propositions \ref{field} and \ref{graph}, as well as inequality \eqref{ogeqn},
\begin{align*}
q^{d} &\leq( |\pgl_n(q)|+q^{(n^2+n)/2-1}+q^{(n^2-1)/2})(q^{\lfloor \frac{n-1}{n}d \rfloor} + q^{\lceil \frac{d}{n} \rceil} )+2\log(\log_2q+2)q^{n^2/2+d/2}\\
& \leq q^{n^2-1}(q^{\lfloor \frac{n-1}{n}d \rfloor} + q^{\lceil \frac{d}{n} \rceil} )+2\log(\log_2q+2)q^{(\frac{1}{2}+\frac{1}{2n})d}\\
& < \frac{3}{2}q^{d-1},
\end{align*}
So this is a contradiction and $G$ has a regular orbit on $V$ when $d\geq n^3$.
%
%
%\[
%q^{d} \leq 2|\mathrm{P}\Gamma \mathrm{L}_n(q).2|q^{\lfloor \frac{n-1}{n} \dim V \rfloor}
%\]
%which holds if and only if $\lceil \frac{d}{n} \rceil \leq \log_q (4|\mathrm{P} \Gamma \mathrm{L}_n(q)|)<n^2+1$. So $\frac{d}{n} \leq  n^2$ and the result follows. 
Now suppose $n=3$ or $4$ with $(n,q)\neq (4,2)$. Then for  $x \in G\setminus F(G)$ of projective prime order, $\dim C_V(x) \leq \lfloor \frac{n-1}{n} d\rfloor$, except that $\dim C_V(x) \leq \lfloor \frac{3}{4} d\rfloor $ when $n=3$ and $x$ is a graph automorphism, or $\dim C_V(x) \leq \lfloor \frac{5}{6}d \rfloor $ when $n=4$ and $x$ is a graph-field automorphism. Therefore, if $V$ is not preserved by graph automorphisms, then we may use the argument above to prove the result.  If $V$ is preserved by graph automorphisms and $G$ has no regular orbit on $V$ then for $n=3$ we have, by Proposition \ref{graph},
\[
q^d \leq |\pgl_3(q)| (q^{\lfloor \frac{2}{3} d \rfloor}+q^{\lceil \frac{d}{3} \rceil} ) + (2q^{5}+2q^4)(q^{\lfloor \frac{3}{4} d \rfloor}+q^{\lceil \frac{d}{4} \rceil} ) + \field.
\]

%\[
%q^{\lceil \frac{d}{4} \rceil} \leq 2|\mathrm{P} \Gamma \mathrm{L}_3(q).2|< q^{10},
%\]
%where $d = \dim V$. So for $d\geq 36$, $G$ has a regular orbit on $V$. The remaining irreducible $G$-modules preserved by graph automorphisms all have dimension at most $n^3=27$, so the result follows for $n=3$.
So $G$ has a regular orbit on $V$.
On the other hand, if $n=4$, $q\geq 3$ and $G$ has no regular orbit on $V$ then by Proposition \ref{graph},

\[
q^d \leq (|\pgl_4(q)|+q^{10})(q^{\lfloor \frac{3}{4} d \rfloor}+q^{\lceil \frac{d}{4} \rceil} ) + (2q^{9})(q^{\lfloor \frac{5}{6} d \rfloor}+q^{\lceil \frac{d}{6} \rceil} ) + \field < 2q^{d-1},
\]
and again this implies that $G$ has a regular orbit on all $V$ of dimension at least $n^3$.  If instead $(n,q)=(4,2)$, then the result follows by \cite{MR3500766}.
%\[
%q^d \leq 2 |\mathrm{P} \Gamma \mathrm{L} _4(q).2| q^{\lfloor \frac{3}{4} d\rfloor}+ (q^{10}+q^9)q^{\lfloor \frac{5}{6} d\rfloor}
%%\]
%and this gives a contradiction for all $d \geq 69$. By \cite[Appendix A.7]{MR1901354}, all $\mathbb{F}_qG$-modules of dimension less than 69 also have dimension at most $n^3=64$, so the result follows.
Finally, if $n=2$, then by Proposition \ref{alphas}, %L2Lemma
provided $q \geq 7$ and $q\neq 9$, $\dim C_V(x) \leq \floor{\frac{1}{2} \dim V}$ for $x$ of odd prime order, and $\dim C_V(x) \leq \floor{\frac{2}{3} \dim V}$ for involutions, except that $\dim C_V(x) \leq \floor{\frac{3}{4} \dim V}$ for involutory field automorphisms. Therefore, by Proposition \ref{invols}, if $G$ has no regular orbit on $V$, then 
\[
q^d\leq 4(q^2+q)(q^{\floor{2d/3}}+q^{\lceil d/3 \rceil}) + 2|\pgl_2(q)|q^{\floor{d/2}} + \frac{|\mathrm{PGL}_2(q)|}{|\mathrm{PGL}_2(q^{1/2})|}q^{\floor{3d/4}}+ \field
\]
%\[
%q^d\leq 4(q^2+q)q^{\floor{2d/3}} + 2|\mathrm{P} \Gamma \mathrm{L}_2(q)|q^{\floor{d/2}} + \frac{|\mathrm{PGL}_2(q)|}{|\mathrm{PGL}(2,q^{1/2})|}q^{\floor{3d/4}}
%\]
This gives a contradiction for $d \geq 9$ and $q\geq 7$, as well as $d=8$ and $q\geq 11$. When $(d,q)=(8,7)$, deleting the final two terms in the inequality (since they account for field automorphisms) gives the result. If $(n,q)=(2,9)$, then the result follows by \cite{MR3500766}.

%If $q\leq 5$, or $q=9$, then $E(G)/Z(E(G))$ is either soluble or an alternating group, so we do not have to consider it.
\end{proof}

By Proposition \ref{n^3}, we only need consider those irreducible $\mathbb{F}_qG$-modules of dimension less than $n^3$. The following result gives a characterisation of all such modules.

\begin{proposition}
\label{complete}
Let $G=\mathrm{SL}_n(q)$ with $q=p^e$ and $p$ prime, and suppose $V = V(\lambda)$ is an absolutely irreducible $d$-dimensional $\mathbb{F}_qG$-module with $d< n^3$. If $\lambda$ is $p$-restricted, then it appears in Table \ref{fulllist}. Otherwise, we can write $V(\lambda)= V(\iota_1)^{(p^a)}\otimes  V(\iota_2)^{(p^b)}$ for $0\leq a<b <e$ and $\iota_1,\iota_2$ $p$-restricted. In this case, $\iota_1$ and $\iota_2$ appear in Table \ref{comptensor}.
\end{proposition}
\begin{proof}
First let $\lambda$ be $p$-restricted.
For $l\geq 19$, Martinez \cite{alvaro} provides a complete list of irreducible $G$-modules of dimension at most $(l+1)^3$, and all of these are included in Table \ref{fulllist}. For $2\leq l\leq 18$, L\"ubeck  [A.6--A.21]\cite{MR1901354} provides a complete list of such modules. If instead $\lambda$ is not $p$-restricted, then for $l\geq 12$, the result follows from inspection of \cite[Table 2]{MR1901354} and for $l\leq 11$, we instead inspect \cite[Appendices A.6--A.15]{MR1901354}.
\end{proof}

\begin{table}[!htbp]
\begin{minipage}{0.45 \textwidth}
\centering
\begin{tabular}{@{}cc@{}}
\toprule
$\lambda$ & $l$ \\ \midrule
$\lambda_1$ & $[1,\infty]$ \\
$\lambda_2$ & $[2,\infty]$ \\
$2\lambda_1$ & $[1,\infty]$ \\
$\lambda_1+\lambda_l$ &  \\
$\lambda_3$ & $[5,\infty]$ \\
$3\lambda_1$ & $[1,\infty]$ \\
$\lambda_1+\lambda_2$ & $[2,\infty]$ \\
$\lambda_1+\lambda_{l-1}$ & $[4,\infty]$ \\
$2 \lambda_1+\lambda_{l}$ & $[2,\infty]$ \\
$\lambda_4$ & $[7,28]$ \\
$2\lambda_2$ & $[3,17]$ \\
$\lambda_5$ & $[9,14]$ \\
$4 \lambda_1$ & $[1,13]$ \\
$\lambda_6$ & $[11,12]$ \\ \bottomrule
\end{tabular}
%\begin{tabular}{ccc}
%\toprule
%$\lambda$ & $\dim V(\lambda)$ & $l$\\ 
%\midrule 
%$\lambda_1$ & $n$ &$[1,\infty]$\\
%$\lambda_2$ & $\binom{n}{2}$ &$[2,\infty]$\\
%$2\lambda_1$ & $\binom{n+1}{2}$ & $[1,\infty]$\\
%$\lambda_1+\lambda_l$& $n^2-1-\ep_{n}$ & \\
%$\lambda_3$ & $\binom{n}{3}$ & $[5,\infty]$ \\ 
%$3\lambda_1$ & $\binom{n+2}{3}$ & $[1,\infty]$ \\ 
%$\lambda_1+\lambda_2$ & $2 \binom{n+1}{3}-\ep_{3} \binom{n}{3}$ & $[2,\infty]$\\
%$\lambda_1+\lambda_{l-1}$ & $3 \binom{n+1}{3}-\binom{n+1}{2}-\ep_{n-1} n$ & $[4,\infty]$ \\ 
%$2 \lambda_1+\lambda_{l}$ & $3 \binom{n+1}{3}+\binom{n}{2}-\ep_{n+1} n$ &$[2,\infty]$  \\ 
%$\lambda_4$ & $\binom{n}{4}$ & $[7,28]$ \\  
%$2\lambda_2$ & $\binom{n}{2} ^2 - (l+1) \binom{n}{3}-\ep_{3} \binom{n}{4}$ 
%& $[3,17]$\\
%$\lambda_5$ & $\binom{n}{5}$ & $[9,14]$\\
%$4 \lambda_1$&& $[1,13]$\\
%$\lambda_6$ && $[11,12]$\\
%\bottomrule 
%\end{tabular} 
\quad
\begin{tabular}{cc}
\toprule
$\lambda$& $l$\\ 
\midrule 
$\lambda_1+\lambda_3$ & $[5,11]$\\
$\lambda_1+\lambda_{l-2}$ & $[6,8]$\\
$\lambda_1+\lambda_4$ & $7$\\
$\lambda_2+\lambda_3$ & $[4,7]$\\
$2\lambda_1+\lambda_2$ &$[2,5]$\\
$\lambda_2+\lambda_4$ & $5$\\
$2\lambda_3$ & $5$\\
$3\lambda_1+\lambda_2$ & $[2,4]$\\
$2\lambda_1+\lambda_3$ & $4$\\
$\lambda_1+\lambda_2+\lambda_3$&$3$\\
$\lambda_1+2\lambda_2$ & $3$\\
$5\lambda_1$ &$[2,3]$\\
$3 \lambda_2$ & $3$\\
%$2\lambda_1+\lambda_2$ & $2$\\
&\\
\bottomrule
\end{tabular}
%Extra module for l=6 if n^3+n is correct bound, also for l\leq 4
\caption{List of $p$-restricted modules $V(\lambda)$ to consider. \label{fulllist}}
\end{minipage}%
%\end{table}
%
%\begin{table}[!htbp]
%\centering
\begin{minipage}{0.55 \textwidth}
\centering
\begin{tabular}{@{}m{2.5cm}m{2cm}l@{}}
\toprule
$\iota_1$ & $\iota_2$ & $l$ \\
\midrule
$\lam 1$ 
%& $(p^a+1)\lam1$ & $[2,\infty)$ \\
% & $(p^a+1)\lam l$ & $[2,\infty)$ \\
 & $\lam 1$ & $[1,\infty)$ \\
 & $\lam l$ & $[2,\infty)$ \\
 & $\lam 2$ & $[2,\infty)$ \\
 & $\lam {l-1}$ & $[2,\infty)$ \\
 & $2\lam 1$ & $[1,\infty)$ \\
 & $2\lam l$ & $[2,\infty)$ \\
 & $\lam 1+\lam l$ & $[2,\infty)$ \\
 & $\lam 3$ & $[5,7]$ \\
 & $\lam {l-2}$ & $[5,7]$ \\
% & $\lam 1 +\lam 2$ & 3, $p=3$ \\
% & $\lam 2 + \lam 3$ & 3, $p=3$ \\
 \midrule
$\lam 2$ & $\lam 2$ & $[3,4]$ \\
 & $\lam {l-1}$ & $[4,5]$ \\
 & $2\lam 1$ & 3 \\
\bottomrule
\end{tabular}
\caption{The list of non-$p$-restricted modules $V(\lambda)= V(\iota_1)^{(p^a)}\otimes  V(\iota_2)^{(p^b)}$ to consider.\label{comptensor}}
\end{minipage}
%\begin{tabular}{ccc}
%\toprule
%$\lambda$ & $\dim V(\lambda)$ & $l$ \\ 
%\midrule 
%$(p^a +p^b+1 )\lambda_1$ & $(l+1)^3$ & $[2,\infty)$ \\
%$(p^a +p^b )\lambda_1+\lambda_l$ & $(l+1)^3$ & $[2,\infty)$ \\
%$(p^a)\lambda_1+(p^b+1)\lambda_l$ & $(l+1)^3$ & $[2,\infty)$ \\
%$(p^b )\lambda_1+(p^a+1)\lambda_l$ & $(l+1)^3$ & $[2,\infty)$ \\
%$(p^a +1 )\lambda_1$ & $(l+1)^2$ & $[2,\infty)$ \\ 
%$p^a\lambda_1+ \lambda_l$ & $(l+1)^2$ & $[2,\infty)$ \\ 
%$\lambda_1+p^a \lambda_2$ & $l(l+1)^2/2$ & $[2, \infty)$ \\ 
%$\lam1+ p^a \lam{l-1}$ & $l(l+1)^2/2$ & $[2, \infty)$ \\ 
%$(p^a+2)\lam1$ & $(l+1) \binom{l+2}{2}$ & $[2,\infty)$ \\ 
%$p^a \lam1+ 2\lam l$ & $(l+1) \binom{l+2}{2}$ & $[2,\infty)$ \\ 
%$(p^a+1)\lam1+ \lam l$ & $(l+1)(l^2+2l-\ep_{l+1})$ & $[2,\infty)$ \\ 
%$\lam1+ p^a\lam 3$ & $(l+1)\binom{l+1}{3}$ & $[5,7]$ \\  
%$\lam1+ p^a\lam {l-2}$ & $(l+1)\binom{l+1}{3}$ & $[5,7]$ \\ 
%$(3^a+1)\lam1+\lam2$ & 64 & $3$ with $p=3$ \\ 
%$\lam1+ 3^a(\lam2+ \lam3)$ & 64 & $3$ with $p=3$ \\ 
%$(p^a+1)\lam2$ & $\binom{l+1}{2}^2$ & $[3,4]$ \\ 
%$p^a \lam 2 + \lam {l-1}$ & $\binom{l+1}{2}^2$ & $[4,5]$ \\ 
%$2\lam 1 + p^a \lam 2$ & 60 & $l=3$ \\ 
%\bottomrule 
%\end{tabular} 

\end{table}

We will aim to prove Theorem \ref{main} for the modules $V=V(\lambda)$ in Tables \ref{fulllist} and \ref{comptensor} using Proposition \ref{tools}. 
In order to do this, we need a method of computing tighter upper bounds on the sizes of fixed point spaces for projective prime order elements than those afforded by Proposition \ref{tools}\ref{alphabound}). We now describe this method, which has been pioneered by Guralnick and Lawther in \cite{bigpaper}.

Let $G \leq \Gamma \mathrm{L}(V)$ be almost quasisimple with $E(G)/Z(E(G)) \cong \mathrm{PSL}_n(q)$ such that $E(G)$ is absolutely irreducible on $V$. Let $\overline{G}$ be a simple algebraic group and $\sigma$ a Frobenius endomorphism of $\overline{G}$ such that $E(G) \leq \overline{G}_{\sigma}$ and $E(G) \leq \overline{G} \leq \mathrm{GL}(V)$. 

Let $\hat{s}\in G$ be a semisimple element of projective prime order $r\neq p$. 
Then there exists $ \nu \in \overline{\mathbb{F}}_q$ such that $s = \nu \hat{s} \in \overline{G}$. Note that $s$ is of projective prime order $r$, and that $s$ and $\hat{s}$ have eigenspaces of the same dimension on $V$. In addition, $s$ lies in a fixed maximal torus $T \leq \overline{G}$, and we define $\Phi$ to be the root system of $\overline{G}$ with respect to $T$. Let $\Phi(s) = \{ \alpha \in \Phi \mid \alpha(s)=1\}$ so that $C_{\overline{G}}(s) = \langle T, U_\alpha \mid \alpha \in \Phi(s) \rangle$, where $U_\alpha$ is the root subgroup in $\overline{G}$ corresponding to $\alpha$. For a closed subsystem $\Psi \subseteq \Phi$, we define the subsystem subgroup  $\overline{G}_\Psi = \langle U_\alpha \mid \alpha \in \Psi \rangle$.

Let $\Psi$ be a \textit{standard subsystem} of $\Phi$. That is, $\Psi = \langle S \rangle$ for some $S \subseteq \Delta$. We define an equivalence relation $\sim$ on the set of weights of $V=V(\lambda)$ of highest weight $\lambda$ by setting $\m1 \sim \m2$ if and only if $\m1 - \m2$ is a linear combination of roots in $\Psi$. We call the corresponding equivalence classes of \textit{$\Psi$-nets}. If $\Psi$ is generated by a single simple root, the resulting $\Psi$-nets are called \textit{weight strings}. 

The \textit{fixed point space} $C_V(g)$ of $g \in G$ acting on $V$ is its 1-eigenspace.
We use $\Psi$-nets to compute lower bounds for $\mathrm{codim} C_V(s)$ (and indeed the codimension of any eigenspace) for semisimple $s \in \overline{G}_\sigma$ of projective prime order $r$ as follows. Assume that $\Psi \cap \Phi(s)$ is empty; then in any $\Psi$-net, any pair of weights that differ by a multiple $M\alpha$ of a root $\alpha$ can only correspond to the same eigenvalue if $M \alpha(s)=1$, so unless $r \mid M$, the two weight spaces must lie in different eigenspaces for $s$. We can use this to compute a lower bound for the contribution of the $\Psi$-net to $\mathrm{codim} C_V(s)$, and we denote this lower bound by $c(s)$.

We may also use these $\Psi$-nets to compute a lower bound $c(u_\Psi)$ for regular unipotent elements $u_\Psi \in \overline{G}_\Psi$. The sum of weight spaces of a given $\Psi$-net forms a (not necessarily irreducible) $\overline{G}_\Psi$-module, and so if we assume $u_\Psi$ is of prime order $p$, we can compute the contribution of the $\Psi$-net to $\mathrm{codim}C_V(u_\Psi)$. We do this using the Jordan canonical form of $u_\Psi$ in its action on the composition factors of the $\Psi$-net module. We denote the lower bound on $\mathrm{codim}C_V(u_\Psi)$ achieved by analysing these $\Psi$-nets by $c(u_\Psi)$.
Once we have computed lower bounds for $\mathrm{codim} C_V(s)$ and $\mathrm{codim} C_V(u_\Psi)$, we then apply Proposition \ref{tools} to attempt to prove that there is a regular orbit of $G$ on $V$. For this to be successful, we need to calculate lower bounds on the numbers of semisimple elements $s$ with $\Phi(s) \cap \Psi = \emptyset$ and unipotent elements $u \in \overline{G}_\sigma$ conjugate to $u_\Psi$ for a given $\Psi$.

%We now give some technical results bounding the number of certain elements of prime order in $\mathrm{Aut}(\mathrm{PSL}_n(q))$.
%Let $l=n-1$.
\begin{proposition}
\label{sscounts}
Let $\Psi$ and $N_s(\Psi)$ be as in Table \ref{ssboundtable}. The number of prime order semisimple elements $g \in \pgl_n(q)$ that are $\mathrm{PGL}_n(\overline{\F}_q)$-conjugate to $s\in T$ such that $\Phi(s)$ intersects every subsystem of type $\Psi\subseteq \Phi$ is at most $N_s(\Psi)$.
\begin{table}[!htbp]
\begin{tabular}{@{}cc@{}}
\toprule
$\Psi$ & $N_s(\Psi)$ \\ \midrule
$A_1^2$ & $q^{2n-1}$ \\\midrule
$A_1^3$ & $2 q^{4n-4}$ \\\midrule
$A_2$ & $\begin{cases} 2q^{ n^2/2 + 3/2} & n \textrm{ odd} \\ 4q^{n^2/2+2} & n \textrm{ even}\\ \end{cases}$ \\\midrule
$A_3$ & $4 q^{\lfloor n^2/2 + 2 \rfloor}+ +\frac{43}{3}q^{2n^2/3+2}+2q^{2n^2/3+1}$ \\ \bottomrule
\end{tabular}
\caption{Upper bounds on the number of prime order semisimple elements $s$ such that $\Phi(s) \cap \Psi \neq \emptyset$.\label{ssboundtable}}
\end{table}
\end{proposition}
\begin{proof}
First let $\Psi$ be of type $A_1^2$. The semisimple elements $s$ with $\Phi(s)$ intersecting every conjugate of $\Psi$ are those with centraliser type $A_{l-1}$. Let $r$ be the order of $s$. By \cite[Table B.3]{bg},  $r\mid q-1$ and each element $s$ has eigenvalues $(\gamma_1^{n-1}, \gamma_2)$ for $\gamma_1$, $\gamma_2$ both $r$th roots of unity. 
The total number of these elements in $\mathrm{PGL}_n(q)$ is less than
\[
(q-1)\frac{|\gl_n(q)|}{|\mathrm{GL}_{n-1}(q)|\gl_1(q)|}< q^{2n-1},
\]
as required.
Now suppose $\Psi$ is of type $A_1^3$.
The closed subsystems of $\Phi$ that intersect every conjugate of $\Psi$ are those of type $A_{l-1}$, $A_1A_{l-2}$, or $A_{l-2}$. We have already dealt with the first case, so
suppose $\Phi(s)$ is of type  $A_1A_{l-2}$. Then by \cite[Table B.3]{bg}, $s$ has eigenvalues $(\gamma_1^2,\gamma_2^{n-2})$ for $\gamma_i$ some $r$th roots of unity in $\mathbb{F}_q$ (so in particular, $q>2$). The total number of elements of this form in $\mathrm{PGL}_n(q)$ is less than
\[
(q-1)\frac{ |\gl_n(q)|}{|\mathrm{GL}_2(q)||\mathrm{GL}_{n-2}(q)|} < 2q^{4n-7}
\]
Now suppose $\Phi(s)$ is of type $A_{l-2}$. Then $s$ has centraliser type either $\mathrm{GL}_1(q)^2  \mathrm{GL}_{n-2}(q)$ or $\mathrm{GL}_1(q^2)\times  \mathrm{GL}_{n-2}(q)$. In the former case, $r\mid q-1$ and $s$ has two eigenvalues of multiplicity one, and a third of multiplicity $n-2$. The total number of elements of this type in $\mathrm{PGL}_n(q)$ is at most
\[
(q-1)^2\frac{ |\gl_n(q)|}{ |\mathrm{GL}_1(q)|^2|\mathrm{GL}_{n-2}(q)|}< q^{4n-4}
\]
Finally, if $s$ has centraliser type $\mathrm{GL}_1(q^2)\times  \mathrm{GL}_{n-2}(q)$, then $r>2$, $r|q+1$ and $s$ has eigenvalues $(\gamma, \gamma^q,1^{n-2})$ for $\gamma$ a primitive $r$th root of unity in $\mathbb{F}_{q^2} \setminus \mathbb{F}_q$.  The total number of these elements in $\mathrm{PGL}_n(q)$ is at most
\[
\frac{q+1}{2}\frac{ |\gl_n(q)|}{|\mathrm{GL}_1(q^2)||\mathrm{GL}_{n-2}(q)|}<\frac{3}{2} q^{4n-5}
\]
The result follows from taking the sum of bounds computed in each case. The proofs for $\Psi$ of type $A_2$ or $A_3$ are similar.
\end{proof}

Suppose $u_1,u_2 \in \pgl_n(q) $ are unipotent of prime order $p$. Define $\pi_i$ to be the partition of $n$ given by the Jordan blocks in the Jordan canonical form of $u_i$ on the natural module for $\gl_n(q)$, and arrange $\pi_i$ so that the parts are in weakly decreasing order. Now, the closure of $u_1^{\overline{G}}$ with respect to the Zariski topology contains $u_2$ if and only if $\pi_1$ dominates $\pi_2$ in the usual partial dominance ordering of partitions \cite[\S 4]{MR672610}. 

\begin{proposition}
\label{unipcounts}
Let $\Psi$ and $N_u(\Psi)$ be as in Table \ref{unipboundtable}, and define $\overline{G} = \mathrm{PGL}_n(\overline{\F}_q)$. The number of unipotent elements $u \in \pgl_n(q) $ of prime order such that the closure of $u^{\overline{G}}$ does not contain a regular unipotent element in $\overline{G}_{\Psi}$ is at most $N_u(\Psi)$.
\begin{table}[!htbp]
\begin{tabular}{@{}cc@{}}
\toprule
$\Psi$ & $N_u(\Psi)$ \\ \midrule
$A_1^2$ & $q^{2n-1}$ \\\midrule
$A_2$ &  $4 \left(\frac{q^{n^2/2+2}-1}{q^2-1} \right)$\\  \midrule
$A_3$ &  $8 \left(\frac{q^{2n^2/3+7/2}-1}{(q^2-1)(q^{3/2}-1)} \right)$ \\\bottomrule
\end{tabular}
\caption{Upper bounds on the number of unipotent elements $u$ of prime order for certain subsystems $\Psi$.\label{unipboundtable}}
\end{table}
\end{proposition}
\begin{proof}
Let $\Psi$ be of type $A_1^2$. The unipotent elements $u$ satisfying the hypothesis are conjugates of a root element. The number of conjugates of a root element is the index of a root parabolic in $\mathrm{PGL}_n(q)$:
\[
\frac{|\gl_n(q)|}{q^{2n-3}|\mathrm{GL}_1(q)||\mathrm{GL}_{n-2}(q)|} = \frac{(q^n-1)(q^{n-1}-1)}{q-1}< q^{2n-1}.
\]
Now suppose $\Psi$ is of type $A_2$. Then $u$ must have associated partition $(2^j, 1^{n-2j})$. Therefore, by \cite[Table B.3]{bg} the total number of such elements is equal to 
\begin{align*}
\sum_{j=1}^{\lfloor n/2 \rfloor} \frac{|\mathrm{GL}_n(q)|}{q^{2(n-2j)j+j^2} |\mathrm{GL}(j,q)||\mathrm{GL}(n-2j,q)|}< \sum_{j=1}^{\lfloor n/2 \rfloor} \frac{4q^{n^2}}{q^{2(n-2j)j+2j^2+(n-2j)^2}} = \sum_{j=1}^{\lfloor n/2 \rfloor}4q^{2j(n-j)}
\end{align*}
Let $f(j) = 2j(n-j)$. Then $f$ is increasing with respect to $j$, achieving a maximum of $\lfloor n^2/2 \rfloor$  at $j = \lfloor n/2 \rfloor$. Moreover, $|f(j+1)-f(j)| = |2(n-2j)-2|\geq 2$, so
\[
\sum_{j=1}^{\lfloor n/2 \rfloor}4q^{2j(2n-j)} < 4(1+q^2+ \dots q^{n^2/2}) = 4\left(\frac{q^{n^2/2+2}-1}{q^2-1} \right).
\]
The proof for $\Psi$ of type $A_3$ is similar. %See book 6
\end{proof}

\subsection*{Summary of table organisation}
%We will use the techniques described in Section \ref{prelim} to prove Theorem \ref{main} for each of the modules in Tables \ref{fulllist} and \ref{comptensor}. 
Let $V(\lambda)$ be a highest weight module with highest weight $\lambda$ in Table \ref{fulllist} or \ref{comptensor}.
In Sections \ref{prest} and \ref{tensor}, we use the technique outlined in the discussion preceding Proposition \ref{sscounts} to find upper bounds for the  eigenspace dimensions of projective prime order elements in $G$ acting on $V(\lambda)$.
We will summarise information about the weights and $\Psi$-nets of $V(\lambda)$, as well as the lower bounds for eigenspace dimensions obtained from these $\Psi$-nets in a collection of tables. We now give a general description of these tables and then give a worked example.

\subsubsection*{Weyl orbit tables}
We use a table to summarise information about the orbits of the Weyl group $W$ on the set of weights of the module. The first column of such a table assigns an index $i$ to each of the Weyl group orbits. The second column lists the unique dominant weight $\mu$ lying in each Weyl orbit, while the third column gives the size of $W.\mu$. Finally, the fourth column gives the multiplicity of each of the weights in the orbit.

\subsubsection*{Weight string tables}

After the Weyl orbit table, we usually proceed by analysing the \textit{weight strings} of the module i.e., the equivalence classes of the set of weights under addition or subtraction by a fixed simple root $\alpha$. Unless explicitly stated, we assume that $\alpha=\alpha_1$.  Let $s \in \overline{G}$ be a semisimple element of prime order, and for a subset $\Psi$ of the root system $\Phi$, let $u_\Psi$ denote a regular unipotent element in $\overline{G}_\Psi$. 

%Recall from Section \ref{prelim} that we denote the lower bounds on the codimension of the largest eigenspace of $s$ and $u_\Psi$ arising from by $c(s)$ and $c(u_{\Psi})$ respectively.

The first column of a weight string table describes the form $\mu_{i_1}\mu_{i_2} \dots \mu_{i_k}$ of weight strings, with each $\mu_i$ a weight lying in the $i$th Weyl orbit (as designated in the Weyl orbit table). The second column gives the number of weight strings of this form in the module. The next (possibly several) columns labelled $c(s)$ give the minimum contribution from each type of weight string to the codimension of the largest eigenspace of a semisimple element $s\in \overline{G}$ of prime order $r$. The final collection of columns labelled $c(u_{\Psi})$ give the minimum contribution from each type of weight string to the codimension of the fixed point space of a regular unipotent element $u_\Psi = u_{\alpha}(1)$. The final row of the table totals the contributions of each type of weight string to give the values of $c(s)$ and $c(u_\Psi)$ arising from the analysis of weight strings.
%The following is a description of the information in the columns of the weight string tables.
%\begin{itemize}
%\item The first column gives the form of the weight string, with each $\mu_i$ a weight lying in the $i$th Weyl orbit (as designated in the Weyl orbit table).
%\item The second column gives the number of weight strings of this form in the module.
%\item The columns labelled $c(s)$ give the minimum contribution from each type of weight string to the codimension of the largest eigenspace of a semisimple element $s\in \overline{G}$ of prime order $r$.
%\item The columns labelled $c(u_{\Psi})$ give the minimum contribution from each type of weight string to the codimension of the fixed point space of a regular unipotent element $u_\Psi = u_{\alpha}(1)$. 
%\end{itemize}
\subsubsection*{$\Psi$-net tables}
If the lower bounds on $\codim C_V(s)$ and $\codim C_V(u)$ computed from the weight strings of $V(\lambda)$ are insufficient to prove that $G$ has a regular orbit on $V$ by Proposition \ref{tools}, we proceed by computing the $\Psi$-nets for larger standard subsystems $\Psi \subseteq\Phi$. We usually do this using GAP \cite{GAP4}.
%Recall that we can define an equivalence relation on the set of weights of $V$ by saying that two weights are related if and only if they differ by a sum of roots in $\Psi$. We call the equivalence classes $\Psi$-nets. 
%We can consider sum of weight spaces in each $\Psi$-net as a $\Psi$-module, and therefore we describe them by their highest weights. Each highest weight is written in terms of the fundamental dominant weights of $\Psi$, which we denote by $\lam i$ for each simple root $\alpha_i \in \Psi$.
As in the case of the weight strings, we summarise this information in a table. 
We will first explain our notation for $\Psi$-nets, which originates from \cite{bigpaper}. We may write a standard subsystem $\Psi = \langle \alpha_i \mid i \in S\rangle$ as a product $\Psi = \Psi_1 \dots \Psi_t$ where each $\Psi_j$ is an irreducible root system. For each $\alpha_i \in \Psi$, there exists $\Psi_j$ with $\alpha_i \in\Psi_j$ and we write $\omega_i$ for the fundamental dominant weight of $\overline{G}_{\Psi_j}$ corresponding to $\alpha_i$. We may then write the highest weight of any $\overline{G}_{\Psi}$-module as a non-negative linear combination of the $\omega_i$. 
%Each $\Psi$-net studied here is comprised of weights which, (ignoring multiplicities) constitute a single Weyl $\overline{G}_\Psi$-module.
 Each $\Psi$-net forms a $\overline{G}_\Psi$-module, and the highest weight $\nu$ of this module is used to denote the $\Psi$-net in the first column of a $\Psi$-net table.
 The next columns, labelled $n_i$ give the number of weights in each $\nu$ from the $i$th orbit under the Weyl group. The next column, labelled ``Mult" gives the number of $\Psi$-nets of this form in the module. The columns labelled $c(s)$ give the minimum contribution from each type of $\Psi$-net to the codimension of the largest eigenspace of a semisimple element $s\in \overline{G}$ of prime order $r$.
The columns labelled $c(u_{\Psi})$ give the minimum contribution from each type of weight string to the codimension of the fixed point space of a regular unipotent element $u_\Psi$ in the subsystem subgroup $\overline{G}_{\Psi}$. 
%\begin{itemize}
%\item The first column gives the weight net $\nu$ as the highest weight of the corresponding $\Psi$-module.
%\item The columns labelled $n_i$ give the number of weights in the $\Psi$-net from the $i$th orbit under the Weyl group.
%\item The column labelled ``Mult" gives the number of $\Psi$-nets of this form in the module.
%\item The columns labelled $c(s)$ give the minimum contribution from each type of weight string to the codimension of the largest eigenspace of a semisimple element $s\in \overline{G}$ of prime order $r$.
%\item The columns labelled $c(u_{\Psi})$ give the minimum contribution from each type of weight string to the codimension of the fixed point space of a regular unipotent element $u_\Psi$ in the subsystem subgroup $\overline{G}_{\Psi}$. 
%\end{itemize}

Throughout, we make extensive use of GAP \cite{GAP4} to determine the $\Psi$-nets of each module and compute the contributions of some $\Psi$-nets to $c(s)$. We also use GAP to give information about the number of prime order elements in some small linear groups and to explicitly construct some small modules to determine whether there is a regular orbit. We also use Mathematica \cite{Mathematica} to confirm that the inequalities arising from Proposition \ref{tools}\ref{eigsp1}--\ref{crude} fail, so as to prove the existence of a regular orbit.
We now demonstrate this technique with a worked example. In this example and throughout this paper, for $G=G(q)$ acting on $V$, we define $\epsilon_k=1$ if the characteristic of $\mathbb{F}_q$ divides $k$, and $\epsilon_k=0$ otherwise.
\begin{wexample}
Let $G$ be an almost quasisimple group with $E(G)/Z(E(G)) \cong \mathrm{PSL}_{l+1}(q)$, $l\in [4,\infty)$ such that $E(G)$ acts absolutely irreducibly on the realisation $V$ of $V(\lam1 + \lam{l-1})$ over $\mathbb{F}_q$. By \cite{MR1901354} ($l\leq 17$) and \cite{alvaro} ($l\geq 18$), we have $d = \dim V = 3\binom{l+2}{3}-\binom{l+2}{2}-\ep_l(l+1)$. 

Using GAP \cite{GAP4}, we compute the weights of the representation, and see that they lie in two orbits under the Weyl group, each with a unique dominant weight $\mu$. This information, along with the size of each Weyl orbit and the multiplicity of weights in each of the Weyl orbits is given in Table \ref{exampleweyl}. 
\begin{table}[!htbp]
	\begin{tabular}{cccc}
		\toprule
		$i$ & $\mu$ & $|W.\mu |$ & Multiplicity \\ 
		\midrule 
		1 & $\lambda_1+\lambda_{l-1}$& $3 \binom{l+1}{3}$  & 1 \\ 
		2 &  $\lambda_{l}$&  $l+1$& $l-1-\epsilon_l$ \\ 
		\bottomrule 
	\end{tabular} 
	\caption{The Weyl orbit table of $V(\lam1 + \lam {l-1})$.\label{exampleweyl}}
\end{table}
Let $\Psi = \langle \alpha_1 \rangle$. After computing the weight strings with respect to $\Psi$ using GAP, we write each of them as a list of $\mu_i$, based on the Weyl orbit of each weight in the string. We write these in the ``String" column of the weight string table (Table \ref{examplews}), and record their multiplicities in the ``Mult" column.
For each weight string, we then determine its minimum contribution to $c(s)$, a lower bound on $\mathrm{codim}C_V(s)$ for a semisimple element $s$ of prime order $r$, using our assumption that $\Psi \cap \Phi(s) = \emptyset$. We do this by trying to find the maximal dimension of a union of weight spaces that could correspond to the same eigenvalue of $s$.
For example, the contribution of $\mu_1$ to $c(s)$ is zero, because there is no restriction on the eigenvalue corresponding to the weight based on our assumption. If we instead consider the weight string $\m1 \m2\m1$, then each consecutive pair of weights differ by $\alpha_1$, so if they correspond to the same eigenvalue, it would imply that $\gamma := \alpha_1(s) = 1$, violating our assumption that $\Psi \cap \Phi(s) = \emptyset$. On the other hand, the first and third weights in the string may lie in the same eigenspace if and only if $2\alpha_1(s)=\gamma^2 = 1$. By our assumption, this can only occur if $r=2$. However, since $\m1$ has multiplicity 1 and $\m2$ has multiplicity $l-1-\ep_l$, the largest contribution to an eigenspace of $s$ from this weight string is $l-1-\ep_l$, so the contribution of each of the $l-1$ weight strings of this form to $c(s)$ is 2.

We subsequently compute the contributions of the weight strings to $c(u_\Psi)$. Define the $A_1$ type subgroup $A = \langle U_{\pm \alpha_1 } \rangle$. We take the sum of the weight spaces in each weight string and consider it as an $A$-module. For example, the string $\mu_2 \mu_2$
corresponds to an $\overline{\mathbb{F}}_pA$-module with $l-1-\ep_l$  composition factors isomorphic to the natural module. Now $u_{\Psi}$ has a 1-dimensional fixed point space on each natural module, so the total contribution to $c(u_\Psi)$ is $l-1-\ep_l$. On the other hand, consider the weight string $\m1 \m2 \m1$.
In characteristic $p=2$, the composition factors of this $A$-module are one copy of the twisted natural module $V(\om1)^{(2)}$ and $l-2-\ep_l$ copies of the trivial module. Now, $u_\Psi$ has a one-dimensional fixed point space on each of these, so the total contribution to $c(u_\Psi)$ by the $l-1$ strings of this type is $l-1$. If instead $p\geq3$, then the weight string has the following composition factors: one copy of the symmetric square $V(2\om1)$, and $l-2-\ep_l$ copies of the trivial module. Again, $u_\Psi$ has a one-dimensional fixed point space on each of these, so the total contribution to $c(u_\Psi)$ by the $l-1$ strings of this type is $2l-2$.
We continue in this way for the remaining weight strings, and the results are summarised in Table \ref{examplews}.
\begin{table}[!htbp]
\begin{tabular}{cccccc}
	\toprule
	&  & \multicolumn{2}{c}{$c(s)$} &  \multicolumn{2}{c}{$c(u_{\Psi})$} \\ 
	
	String & Multiplicity & $r=2$ & $r\geq 3$ & $p=2$ & $p\geq 3$ \\ 
	\midrule 
	$\mu_1$ & $(l-4) \binom{l-1}{2}+ \binom{l}{2}$ & 0 & 0 & 0 & 0 \\ 
	
	$\mu_1 \mu_1$ &$ 3 \binom{l-1}{2}$ & $ 3 \binom{l-1}{2}$ &$ 3 \binom{l-1}{2}$ &$ 3 \binom{l-1}{2}$& $ 3 \binom{l-1}{2}$ \\ 
	
	$\mu_1 \mu_2 \mu_1$ &$ l-1$ & $2l-2$ & $2l-2 $& $l-1$ &$ 2l-2 $\\ 
	
	$\mu_2 \mu_2$ & 1 & $l-1-\ep_l$&$ l-1-\ep_l$ & $l-1-\ep_l$ & $l-1-\ep_l$ \\ 
	
	\midrule 
	Total & &$3 \binom{l}{2}-\ep_l$&$3 \binom{l}{2}-\ep_l$ & $3 \binom{l}{2}-(l-1)-\ep_l$&$3 \binom{l}{2}-\ep_l$\\
	\bottomrule
\end{tabular} 
\caption{Weight string table for $V(\lam1 + \lam{l-1})$.\label{examplews}}
\end{table}
Now let $\Psi = \langle \a_1 , \a_2 \rangle$ be a subsystem of $\Phi$ of type $A_2$.   The $\Psi$-net table is given in Table \ref{l_1+l_l-1_A2}, where $\delta_{l,4}$ is the Kronecker delta function. We compute the $\Psi$-nets of the weights of $V$ using GAP \cite{GAP4}, consider each $\Psi$-net as a $\overline{G}_\Psi$-module and determine its highest weight.
We then perform a similar analysis to that for the weight string table. For instance, take the $\overline{G}_\Psi$ module of highest weight $2\om1$. The weights lying in the first Weyl orbit are $s_1=\{2\om1, 2\om1-2\a_1, 2\om1-2\a_1-2\a_2\}$, while those lying in the second are $s_2=\{2\om1-\a_1, 2\om1-\a_1-\a_2, 2\om1-2\a_1-\a_2\}$. Examining Table \ref{exampleweyl}, we see that each weight in $s_1$ has multiplicity 1, while each weight in $s_2$ has multiplicity $l-1-\ep_l$. The collection of weights that we can take with maximal weight space dimension, which also do not pairwise differ by an element of $\Psi$ is obtained by taking one element each of $s_1$ and $s_2$, for example $\{2\om1,2\om1-2\a_1-\a_2 \}$. Therefore, the contribution to $c(s)$ for this $\Psi$-net is the sum of multiplicities of the remaining  four weights (two from each of $s_1$ and $s_2$), which is $2(l-1-\ep_l+1) = 2(l-\ep_l)$. We now consider the contribution of this $\Psi$-net to $c(u_\Psi)$. For $p\geq 3$, the composition factors of the $\Psi$-net with highest weight $2\om1$ are: one copy of $V(2\om1)$, and $l-2-\ep_l$ copies of $V(\om2)$. A regular unipotent element $u_\Psi$ has a fixed point space of dimension two on the former %by reg ss/u fps lemma of G&L
 and 1 on the latter, giving a total contribution to $c(u_\Psi)$ of $4+2(l-2-\ep_l) = 2l-2\ep_l$. We continue in this manner with the other $\Psi$-nets to complete the $\Psi$-net table.
\begin{table}[!htbp]
	\begin{tabular}{ccccccc}
		\toprule
		&&&&$c(s)$ & \multicolumn{2}{c}{$c(u_\Psi)$}\\
		\midrule 
		$\nu$ & $n_1$ & $n_2$ & Mult & $r\geq 3$ & $p=3$ & $p=5$ \\ 
		\midrule 
		$\om1+\om2$ & 6 & 1 & $l-2$ & $(6-\ep_l \delta_{l,4})(l-2)$ & $4(l-2)$ & $6(l-2)$ \\ 
		
		$2\om1$ & 3 & 3 & 1 & $2(l-\ep_l)$ & $2l-2\ep_l$ & $2l-2\ep_l$ \\ 
		
		$\om2$ & 3 & 0 & $(l-2)(l-3)$ & $2(l-2)(l-3)$ & $2(l-2)(l-3)$ & $2(l-2)(l-3)$ \\ 
		
		$\om1$ & 3 & 0 & $\binom{l-1}{2}$ & $(l-1)(l-2)$ & $(l-1)(l-2)$ & $(l-1)(l-2)$ \\ 
		
		0 & 1 & 0 & $3\binom{l-2}{3}$ & 0 & 0 & 0 \\ 
		\midrule 
		\textbf{Total} &  &  & & $3l^2-5l+2-2\ep_l(1+\ep_2)$ & $3l^2-7l+6-2\ep_l$ & $3l^2-5l+2-2\ep_l$ \\ 
		\bottomrule 
	\end{tabular} 
	\caption{The $A_2$-net table for $V(\lam1+\lam{l-1})$. \label{l_1+l_l-1_A2}}
\end{table}

We are now ready to prove that $G$ has a regular orbit on $V$, where in this example $V=V(\lam1+\lam{l-1})$. We will use Proposition \ref{tools}\ref{crude}), and also Proposition \ref{field} for field automorphisms, and Propositions \ref{sscounts} and \ref{unipcounts} to determine an upper bound on the number of elements of prime order in $G/F(G)$ that we cannot use bounds computed in Table \ref{l_1+l_l-1_A2} for.
If $G$ has no regular orbit on $V$, then
\begin{align*}
q^d \leq & 2 |\pgl_n(q)|q^{d-(3l^2-5l+2-2\ep_l(1+\ep_2)}+(8q^{n^2/2+2}+4\left(\frac{q^{n^2/2+2}-1}{q^2-1}\right))q^{d-3\binom{l}{2}+\ep_l}\\
&+2q^{2n-1}q^{d-3\binom{l}{2}+l-1+\ep_l}+ 2\log(\log_2q+2)q^{n^2/2+d/2}.
\end{align*}
This is a contradiction for $l\geq 4$ and $q\geq 2$, so $G$ has a regular orbit on $V$.
\end{wexample}

We now analyse the modules in Table \ref{fulllist} individually. We will start by fixing some notation. Firstly, $V=V(\lambda)$ will denote an absolutely irreducible $\F_q\mathrm{SL}_n(q)$ module of highest weight $\lambda$.
Moreover, $G \leq \Gamma \mathrm{L}(V)$ will be an almost quasisimple group with $E(G)/Z(E(G)) \cong \mathrm{PSL}_n(q)$ such that the restriction of $V$ to $E(G)$ is absolutely irreducible.
We will also write $d= \dim V$, $q=p^e$ and let $r \neq p$ be a prime. We will also use both $l$ and $n = l+1$ throughout for convenience in notation. 

 \section{Proof of Theorem \ref{main},  II : $p$-restricted modules}
\label{prest}
The main result of this section is as follows.
\begin{theorem}
Let $V=V_d(q)$ be a $d$-dimensional vector space over $\mathbb{F}_q$ with $q=p^e$, and let $G \leq \Gamma \mathrm{L}(V)$ be almost quasisimple with $E(G)/Z(E(G)) \cong \mathrm{PSL}_n(q)$. Further suppose that the restriction of $V=V(\lambda)$ to $E(G)$ is an absolutely irreducible module of $p$-restricted highest weight $\lambda$. Either $G$ has a regular orbit on $V$, or $b(G)>1$ and $n$, $b(G)$ and $\lambda$ (up to quasiequivalence) appear in Table \ref{rem1}.
\label{prestmain}
\end{theorem}
By Propositions \ref{n^3} and \ref{complete}, the proof of Theorem \ref{prestmain} reduces to an analysis of the modules $V(\lambda)$ with highest weight $\lambda$ in Table \ref{fulllist}.
We begin this analysis with some modules where there is no regular orbit under the action of $G$.
\begin{proposition}
Theorem \ref{prestmain} holds for $G$ with $E(G)/Z(E(G)) \cong \mathrm{PSL}_n(q)$ acting on $V=V(\lam1)$ and $V(2\lam1)$  with $l \in [1, \infty)$, and $V=V(\lam2)$ for $l \in [3,\infty)$. Namely:
\begin{enumerate}
	\item if $V=V(\lam1)$ then $b(G)\leq n+1$ if $G$ contains field automorphisms and $b(G)=n$ otherwise,
	\item if $V=V(2\lam1)$ then $2\leq b(G)\leq 3$, and 
	\item if $V=V(\lam2)$ then $b(G)=3$ if $n\geq 7$,  $3\leq b(G)\leq 4$ for  $n \in [5,6]$ and $3\leq b(G)\leq 5$ for $n=4$.
\end{enumerate}
\end{proposition}
\begin{proof}
We have $\log |G|/\log|V|> n-1, 2$ and 1 for $V(\lam1)$, $V(\lam2)$ and $V(2\lam1)$ respectively, except that  $\log |G|/\log|V|<1$ for some $G$ with $l=1$ acting on $V(2\lam1)$. In this case, $q$ is odd and since $L_2(q) \cong \Omega_3(q)$, there is no regular orbit of $G$ on $V$ by \cite[Lemma 2.10.5(iv)]{KL}.
We will prove the upper bounds on base size for these three modules later in Propositions \ref{L1_base_size} and \ref{k=1_leftovers}.
\end{proof}

As before, for $G$ acting on $V$, we define $\epsilon_k=1$ if the characteristic of $\mathbb{F}_q$ divides $k$, and $\epsilon_k=0$ otherwise.

\subsection{$\lambda = 4\lam1$, with $l \in [2,13]$}
\begin{proposition}
Theorem \ref{prestmain} holds for $G$ with $E(G)/Z(E(G)) \cong \mathrm{PSL}_{l+1}(q)$ acting on $V=V(4\lam1)$ for $l \in [1, 13]$. Namely, $G$ has a regular orbit on $V$.
\end{proposition}
\begin{proof}
Here $d = \binom{l+4}{4}$ and $p\geq 5$. The Weyl orbit and weight string tables are given in Tables \ref{4L1_orb} and \ref{4L1_ws} respectively.\\
{
\renewcommand{\arraystretch}{1.3}
\begin{table}[!htbp]
\begin{tabular}{cccc}
\toprule 
$i$ & $\mu$ & $|W.\mu |$ & Mult \\ 
\midrule 
1 & $4\lambda_1$& $l+1$& 1 \\ 
2 &  $2\lam1+\lam2 $&$2\binom{l+1}{2}$& 1 \\ 
3& $2\lam2$ &$\binom{l+1}{2}$&1\\
4& $\lam1+\lam3$&$3\binom{l+1}{3}$&1\\
5& $\lam 4$ & $\binom{l+1}{4}$ & 1\\
 \bottomrule
\end{tabular}
\caption{The Weyl orbit table of $V(4\lam1)$.\label{4L1_orb}}
\end{table}
}
\begin{table}[!htbp]
\begin{tabular}{cccccc}
\toprule
 &  & \multicolumn{3}{c}{$c(s)$} & $c(u_\Psi)$ \\ 
 \midrule
String & Mult & $r=2$ & $r=3$ & $r\geq 5$ & $p\geq 5$ \\ 
\midrule
$\m1$ & $l-1$ & 0 & 0 & 0 & 0 \\ 
$\m1\m2\m3\m2\m1$ & 1 & 2 & 3 & 4 & 3 \\ 
$\m2$ & $2\binom{l-1}{2}$ & 0 & 0 & 0 & 0 \\ 
$\m2\m2$ & $l-1$ & $l-1$ & $l-1$ & $l-1$ & $l-1$ \\ 
$\m2\m4\m4\m2$ & $l-1$ & $2(l-1)$ & $2(l-1)$ & $3(l-1)$ & $3(l-1)$ \\ 
$\m3$ & $\binom{l-1}{2}$ & 0 & 0 & 0 & 0 \\ 
$\m3\m4\m3$ & $l-1$ & $l-1$ & $2(l-1)$ & $2(l-1)$ & $2(l-1)$ \\ 
$\m4$ & $3\binom{l-1}{3}$ & 0 & 0 & 0 & 0 \\ 
$\m4\m4$ & $2\binom{l-1}{2}$ & $2\binom{l-1}{2}$ & $2\binom{l-1}{2}$ & $2\binom{l-1}{2}$ & $2\binom{l-1}{2}$ \\ 
$\m4\m5\m4$ & $\binom{l-1}{2}$ & $\binom{l-1}{2}$ & $2\binom{l-1}{2}$ & $2\binom{l-1}{2}$ & $2\binom{l-1}{2}$ \\ 
$\m5$ & $\binom{l-1}{4}$ & 0 & 0 & 0 & 0 \\ 
$\m5\m5$ & $\binom{l-1}{3}$ & $\binom{l-1}{3}$ & $\binom{l-1}{3}$ & $\binom{l-1}{3}$ & $\binom{l-1}{3}$ \\ 
\midrule
\textbf{Total} &  & $\binom{l}{3}+2\binom{l+1}{2}$ & $\binom{l+2}{3}+\binom{l-1}{2}+2l$ & $\binom{l+3}{3}$ &  $\binom{l+3}{3}$-1 \\ 
\bottomrule 
\end{tabular} 
\caption{The weight string table of $V(4\lam1)$.\label{4L1_ws}}
\end{table}

If $G$ has no regular orbit on $V$, then by Propositions \ref{invols}, \ref{field} and \ref{tools} \ref{crude}),
\begin{align*}
q^{\binom{l+4}{4}} \leq& 2|\mathrm{PGL}_{l+1}(q)| q^{d-\binom{l+2}{3}-\binom{l+1}{2}-2l}+ 4(q^{\binom{l+2}{2}-1} + q^{\binom{l+2}{2}-2})q^{d-\binom{l}{3}-2\binom{l+1}{2}}+ 2\log(\log_2 q+2)q^{n^2/2+ d/2}
\end{align*}
This gives a contradiction for all $l\geq 2$ and $q\geq 5$. When $l=1$, we use more accurate upper bounds on the number of involutions, elements of order 3 and semisimple elements in $\mathrm{PGL}_2(q)$, and this gives the result for $q\geq 180$. For the remaining values of $q$, we use GAP \cite{GAP4} to construct each module $V$ and find regular orbits explicitly.
\end{proof}
\subsection{$\lambda = \lam1+\lam3$, with $l\in[5,11]$}
\begin{proposition}
Theorem \ref{prestmain} holds for $G$ with $E(G)/Z(E(G)) \cong \mathrm{PSL}_{l+1}(q)$ acting on $V=V(\lam1+\lam3)$ for $l \in [5, 11]$. Namely, $G$ has a regular orbit on $V$.
\end{proposition}
\begin{proof}
Here $d = 3 \binom{l+2}{4} - \ep_2\binom{l+1}{4}$. The Weyl orbit and weight string tables are found in Tables \ref{weyll1l3} and \ref{wsl1l3} respectively.\\
\begin{table}[!htbp]
\begin{tabular}{cccc}
\toprule 
$i$ & $\mu$ & $|W.\mu |$ & Mult \\ 
\midrule 
1 & $\lambda_1+\lambda_{3}$& $3 \binom{l+1}{3}$ & 1 \\ 
2 &  $\lam4 $&$\binom{l+1}{4}$& $3-\ep_2$ \\ 
\bottomrule 
\end{tabular}
\caption{Weyl orbit table for $V( \lam1+\lam3)$.\label{weyll1l3}}
\end{table}
\begin{table}[!htbp]
\begin{tabular}{cccccc}
\toprule
 &  & \multicolumn{2}{c}{$c(s)$} &  \multicolumn{2}{c}{$c(u_{\Psi})$}\\ 
 \midrule
String & Mult & $r=2$  & $r\geq 3$ & $p=2$ & $p\geq 3$ \\ 
\midrule
$\m1$ & $3\binom{l-1}{3}+l-1$ & 0 & 0 & 0 & 0 \\ 
$\m1\m1$ & $(l-1)^2$ & $(l-1)^2$ & $(l-1)^2$ & $(l-1)^2$ & $(l-1)^2$ \\ 
$\m1\m2\m1$ & $\binom{l-1}{2}$ & $2 \binom{l-1}{2}$ & $2 \binom{l-1}{2}$ & $ \binom{l-1}{2}$ & $2 \binom{l-1}{2}$ \\ 
$\m2$ & $\binom{l-1}{4}$ & 0 & 0 & 0 & 0 \\ 
$\m2\m2$ & $\binom{l-1}{3}$ & $3\binom{l-1}{3}$ & $\binom{l-1}{3}(3-\ep_2)$ & $2\binom{l-1}{3}$ & $3\binom{l-1}{3}$ \\ 
\midrule
\textbf{Total} &  & $\frac{1}{2} l(l-1)^2$ & $\frac{1}{2} l(l-1)^2 - \ep_2 \binom{l-1}{3}$ & $\binom{l+1}{3} + \binom{l}{3}$ & $\frac{1}{2} l(l-1)^2$ \\ 
\bottomrule 
\end{tabular} 
\caption{Weight string table for $V( \lam1+\lam3)$.\label{wsl1l3}}
\end{table}

If $G$ has no regular orbit on $V$ then by Propositions \ref{invols}, \ref{field} and \ref{tools} \ref{crude}),
\[
q^d \leq 2 |\pgl_n(q)| q^{d-(l(l-1)^2/2-\ep_2 \binom{l-1}{3}}+ \itwo q^{d-\binom{l+1}{3}-\binom{l}{3}}+\field
\]
%Suppose $p\neq 2$. If $G$ has no regular orbit on $V$, then
%\[
%q^{3 \binom{l+2}{4} } \leq 2|\mathrm{PGL}_n(q)| q^{d-\frac{1}{2} l(l-1)^2} + 2\log(\log_2q+2)q^{n^2/2+d/2}
%\]
%This gives a contradiction for $l\geq 5$ and $q\geq 3$, so $G$ has a regular orbit on $V$ here.\\
%Now suppose $p=2$. If $G$ has no regular orbit on $V$, then
%\[
%q^{3 \binom{l+2}{4} -\binom{l+1}{4}} \leq  2|\mathrm{PGL}_n(q)| q^{d-\frac{1}{2} l(l-1)^2+\binom{l-1}{3}}+ q^{l(l+1)+d- \binom{l+1}{3} - \binom{l}{3}} + 2\log(\log_2q+2)q^{n^2/2+d/2}
%\]
This is a contradiction for $l\geq 5$ and $q\geq 2$, so $G$ has a regular orbit on $V$ here.
\end{proof}
\subsection{$\lambda = \lam2 + \lam3$, $l\in [4,7]$}
\begin{proposition}
Theorem \ref{prestmain} holds for $G$ with $E(G)/Z(E(G)) \cong \mathrm{PSL}_{l+1}(q)$ acting on $V=V(\lam2+\lam3)$ for $l \in [4, 7]$. Namely, $G$ has a regular orbit on $V$.
\end{proposition}
\begin{proof}
First suppose $l=4$.
Here $d=75-\ep_2-24\ep_3$. The Weyl orbit and weight string tables are given in Table \ref{l2+l3weyl4}.\\
\begin{table}[!htbp]
\begin{tabular}{cccc}
\toprule 
$i$ & $\mu$ & $|W.\mu |$ & Mult \\ 
\midrule 
1 & $\lambda_2+\lambda_3$& $30$ & 1 \\ 
2 &  $\lam1+\lam4$&20& $2-\epsilon_3$ \\ 
3 &  $0$&1& $5-\ep_2-4\epsilon_3$ \\ 
\bottomrule 
\end{tabular}
\begin{tabular}{ccccccc}
\toprule
&&\multicolumn{2}{c}{$c(s)$}&\multicolumn{3}{c}{$c(u_{\Psi})$}\\
\midrule 
String & Mult & $r=2$ & $r\geq 3$ & $p=2$ & $p=3$ & $p\geq 5$ \\ 
\midrule 
$\m1$ & 6 & 0 & 0 & 0 & 0 & 0 \\ 
$\m1\m1$ & 6 & 6 & 6 & 6 & 6 & 6 \\ 
$\m1\m2\m1$ & 6 & $6(2-\epsilon_3)$ & 12 & 6 & 12 & 12 \\ 
$\m2\m2$ & 6 & $6(2-\epsilon_3)$ & $6(2-\epsilon_3)$ & 12 & 6 & 12 \\ 
$\m2\m3\m2$ & 1 & $4-3\epsilon_3$ & $4-2\epsilon_3$ & 2 & 2 & 4 \\ 
\midrule 
\textbf{Total} &  & $34-15\epsilon_3$ & $34-8\epsilon_3$ & 26 & 26 & 34 \\ 
\bottomrule 
\end{tabular} 
\caption{Weyl orbit and weight string tables for $V(\lam2+\lam3)$, with $l=4$.\label{l2+l3weyl4}}
\end{table}

Here we must also consider graph and graph field automorphisms, since they preserve the set of weights of the module.
If $\tau$ is a graph automorphism, then we determine using Proposition \ref{graphfix} that $\dim C_V(\tau) \leq 46-16\ep_3-\ep_2$.
Therefore, if $G$ has no regular orbit on $V$, then by Propositions \ref{invols}, \ref{field}, \ref{graph} and \ref{tools} \ref{crude})
\begin{align*}
q^{75-\ep_2-24\ep_3} \leq& 2|\mathrm{PGL}_5(q)|q^{d-34+8\ep_3}+4(q^{14}+q^{13})q^{d-34+15\ep_3}+ 2\log(\log_2 q+2)q^{25/2 + d/2}\\
& + 4(q^{12}+q^{14})q^{46-16\ep_3-\ep_2}
\end{align*}
which gives a contradiction for all $q\geq 2$. Therefore, $G$ has a regular orbit on $V$.
Now suppose $l=5$.
Here $d=210-6\ep_2-84\ep_3$. The Weyl orbit and weight string tables are given in Table \ref{l2+l3_5}\\
\begin{table}[!htbp]
\begin{minipage}{0.35 \textwidth}
\begin{tabular}{cccc}
\toprule 
$i$ & $\mu$ & $|W.\mu |$ & Mult \\ 
\midrule 
1 & $\lambda_2+\lambda_3$& $60$ & 1 \\ 
2 &  $\lam1+\lam4$&60& $2-\ep_3$ \\ 
3 &  $\lam5$&6& $5-\ep_2-4\ep_3$ \\ 
\bottomrule 
\end{tabular}
%\caption{Weyl orbit table for $V(\lam2+\lam3)$, with $l=5$.\label{l2+l3weyl5}}
\end{minipage}%
\quad
\begin{minipage}{0.6\textwidth}
\begin{tabular}{ccccccc}
\toprule
&& \multicolumn{2}{c}{$c(s)$} & \multicolumn{3}{c}{$c(u_{\Psi})$}\\
\midrule
String & Mult & $r=2$ & $r\geq 3$ & $p=2$ & $p=3$ & $p\geq 5$ \\ 
\midrule
$\m1$ & 16 & 0 & 0 & 0 & 0 & 0 \\ 
$\m1\m1$ & 10 & 10 & 10 & 10 & 10 & 10 \\ 
$\m1\m2\m1$ & 12 & $24-2\ep_3$ & 24 & 12 & 24 & 24 \\ 
$\m2$ & 4 & 0 & 0 & 0 & 0 & 0 \\ 
$\m2\m2$ & 18 & $36-18\ep_3$ & $36-18\ep_3$ & 36 & 18 & 36 \\ 
$\m2\m3\m2$ & 4 & $16-8\ep_3$ & $16-8\ep_3$ & 8 & 8 & 16 \\ 
$\m3\m3$ & 1 & $5-4\ep_3$ & $5-\ep_2-4\ep_3$ & 4 & 1 & 5 \\ 
\midrule
\textbf{Total} &  & $91-32 \ep_3$ & $91-30\ep_3-\ep_2$ & 70 & 61 & 91 \\ 
\bottomrule 
\end{tabular} 
%\caption{Weight string table for $V(\lam2+\lam3)$, with $l=5$.\label{l2+l3wt5}}
\end{minipage}
\caption{Weyl orbit  and weight string tables for $V(\lam2+\lam3)$ with $l=5$. \label{l2+l3_5}}
\end{table}

If $G$ has no regular orbit on $V$, then by Propositions \ref{field} and \ref{tools} \ref{crude}),
\[
q^d \leq 2|\mathrm{PGL}_6(q)|q^{d-59}+ 2\log(\log_2q+2)q^{36/2+d/2}
\]
This gives a contradiction for $q\geq 2$.
The arguments for $l=6,7$ are similar.
\end{proof}
\subsection{$\lambda = \lam2+\lam4$, with $l=5$}
\begin{proposition}
Theorem \ref{prestmain} holds for $G$ with $E(G)/Z(E(G)) \cong \mathrm{PSL}_6(q)$ acting on $V=V(\lam2+\lam4)$. Namely, $G$ has a regular orbit on $V$.
\end{proposition}
\begin{proof}
Here $d=189-\ep_{5}-35\ep_2$. The Weyl orbit and weight string tables are given in Table \ref{l2+l4_5}.
\begin{table}[!htbp]
\begin{minipage}{0.35\textwidth}
\begin{tabular}{cccc}
\toprule 
$i$ & $\mu$ & $|W.\mu |$ & Mult \\ 
\midrule 
1 & $\lambda_2+\lambda_4$& $90$ & 1 \\ 
2 &  $\lambda_1 +\lambda_5$&30& $3-\ep_2$ \\ 
3 &  $0$&1& $9-\ep_5-5\ep_2$ \\ 
\bottomrule 
\end{tabular}
%\caption{Weyl orbit table for $V(\lam2 + \lam 4)$.\label{weyll2l4}}
\end{minipage}
\quad
\begin{minipage}{0.6 \textwidth}
\begin{tabular}{ccccccc}
\toprule 

&& \multicolumn{2}{c}{$c(s)$}& \multicolumn{3}{c}{$c(u_{\Psi})$}\\ 
\midrule 

String & Mult & $r=2$ & $r\geq 3$ & $p=2$ & $p=5$ & $p \neq 2,5$ \\ 
\midrule 

$\m1$ & 18 & 0 & 0 & 0 & 0 & 0 \\ 
$\m1\m1$ & 24 & 24 & 24 & 24 & 24 & 24 \\ 
$\m1\m2\m1$ & 12 & 24 & 24 & 12 & 24 & 24 \\ 
$\m2\m2$ & 8 & 24 & $24-8\ep_2$ & 16 & 24 & 24 \\ 
$\m2\m3\m2$ & 1 & 6 & $6-2\ep_2$ & 2 & 6 & 6 \\ 
\midrule 
\textbf{Total} &  & 78 & $78-10\ep_2$ & 54 & 78 & 78 \\ 
\bottomrule 
\end{tabular}
%\caption{Weight string table for $V(\lam2 + \lam 4)$ with $l=5$.\label{wsl2l4}}
\end{minipage}
\caption{Weyl orbit  and weight string tables for $V(\lam2+\lam4)$ with $l=5$. \label{l2+l4_5}}
\end{table}

Using Proposition \ref{graphfix}, we also determine that for a graph automorphism $\tau$, we have $\dim C_V(\tau) \leq 114-23\ep_2 - \ep_{5}$.

Therefore, if $G$ has no regular orbit on $V$, then by Propositions \ref{field}, \ref{graph} and \ref{tools} \ref{crude})
\[
q^{189-\ep_{5}-35\ep_2} \leq 2|\pgl_6(q) | q^{d - 54} + 2\log(\log_2q+2)q^{36/2+ d/2} + (2q^{20}+ 2q^{35/2})q^{114-23\ep_2 - \ep_{5}}.
\]
This gives a contradiction for all $q\geq 2$, so $G$ has a regular orbit on $V$.
\end{proof}
\subsection{$\lambda = 3\lambda_1+\lambda_2$, $l\in [2,4]$}
\begin{proposition}
Theorem \ref{prestmain} holds for $G$ with $E(G)/Z(E(G)) \cong \mathrm{PSL}_{l+1}(q)$ acting on $V=V(3\lam1+\lam2)$ for $l \in [2, 4]$. Namely, $G$ has a regular orbit on $V$.
\end{proposition}
\begin{proof}
Note here that $p\geq 5$.
First suppose that $l=2$.
Here $d=24-6\ep_{5}$. The Weyl orbit and weight string tables are given in Tables \ref{3l1+l2weyl} and \ref{3l1+l2A1}.
\begin{table}[!htbp]
\begin{minipage}{0.35 \textwidth}
\centering
\begin{tabular}{cccc}
\toprule 
$i$ & $\mu$ & $|W.\mu |$ & Mult \\ 
\midrule 
1 & $3\lambda_1+\lambda_2$& $6$ & 1 \\ 
2 &  $\lambda_1 +2\lambda_2$&6& 1 \\ 
3 &  $2\lambda_1$&3& $2-\ep_5$ \\ 
4& $\lam2$ &3&$2-\ep_5$\\
\bottomrule 
\end{tabular} 
\caption{The Weyl orbit table of $V(3\lam1+\lam2)$, $l=2$.\label{3l1+l2weyl}}
\end{minipage}
\begin{minipage}{0.6 \textwidth}
\centering
\begin{tabular}{ccccccc}
\toprule 
&&\multicolumn{3}{c}{$c(s)$} &\multicolumn{2}{c}{$c(u_{\Psi})$}\\
String & Mult & $r=2$ & $r=3$ & $r\geq 5$ & $p=5$ & $p\geq 7$ \\ 
\midrule 
$\m1\m1$ & 1 & 1 & 1 & 1 & 1 & 1 \\ 

$\m1\m2\m2\m1$ & 1 & 2 & 2 & 3 & 3 & 3 \\ 

$\m1\m3\m4\m3\m1$ & 1 & $4-2\ep_5$ & $5-2\ep_5$ & $6-2\ep_5$ & 4 & 6 \\ 

$\m2\m3\m2$ & 1 & $2-\ep_5$ & 2 & 2 & 2 & 2 \\ 

$\m2\m4\m4\m2$ & 1 & $3-\ep_5$ & $4-2\ep_5$ & $4-\ep_5$ & 3 & 4 \\ 
\midrule 
\textbf{Total} &  & $12-4\ep_5$ & $14-4\ep_5$ & $16-3\ep_5$ & 13 & 16 \\ 
\bottomrule 
\end{tabular} 
\caption{The weight string table of $V(3\lam1+\lam2)$, $l=2$.\label{3l1+l2A1}}
\end{minipage}
\end{table}
If $G$ has no regular orbit on $V$, then 
\begin{align*}
q^{d}&\leq 2|\mathrm{PGL}_3(q)| q^{d-13}+ 4(q^{5}+q^4)q^{d-(12-4\ep_5)}+ 2(q^6+q^5)q^{d-(14-4\ep_5)}+2\log(\log_2q+2)q^{9/2+d/2}\\
\end{align*}
This gives a contradiction for all $q \geq 5$.
The proofs for $l=3,4$ are similar.
\end{proof}
%\subsection{$\lambda=2\lambda_1+\lambda_3$, $l=4$}
%\begin{proposition}
%Let $G$ be an almost quasisimple group with $E(G)/Z(E(G)) \cong \mathrm{PSL}_5(q)$, acting absolutely irreducibly on the restriction of $V=V(\lambda)$, $\lambda = 2\lam1 + \lam 3$ to $\mathbb{F}_q$. Then $G$ has a regular orbit on $V$.
%\end{proposition}
%\begin{proof}
%Here, $\dim V =126-23 \ep_{5}$. Therefore, by Proposition \ref{n^3}, we only need consider the case where $p=5$.  The Weyl orbit and weight string tables are given below.
%\begin{table}[!htbp]
%\begin{tabular}{cccc}
%\toprule 
%$i$ & $\mu$ & $|W.\mu |$ & Mult \\ 
%\midrule 
%1 & $2\lambda_1+\lambda_3$& $30$ & 1 \\ 
%2 &  $\lambda_2 +\lambda_3$&30& 1 \\ 
%3 &  $\lambda_1+\lambda_4$&20& 2 \\ 
%4& $0$ &1&$3$\\
%\bottomrule 
%\end{tabular}
%\begin{tabular}{ccccc}
%\toprule 
% &  & \multicolumn{2}{c}{$c(s)$} & $c(u_{\Psi})$ \\ 
% \midrule
%String & Mult & $r=2$ & $r\geq 3$ & $p=5$ \\ 
%\midrule
%$\m1$ & 6 & 0 & 0 & 0 \\ 
%$\m1\m1$ & 6 & 6 & 6 & 6 \\ 
%$\m1\m2\m1$ & 3 & 3 & 6 & 6 \\ 
%$\m1\m3\m3\m1$ & 3 & 9 & 12 & 12 \\ 
%$\m2$ & 3 & 0 & 0 & 0 \\ 
%$\m2\m2$ & 6 & 6 & 6 & 6 \\ 
%$\m2\m3\m2$ & 6 & 12 & 12 & 12 \\ 
%$\m3\m3$ & 3 & 6 & 6 & 6 \\ 
%$\m3\m4\m3$ & 1 & 3 & 4 & 4 \\ 
%\midrule
%\textbf{Total} &  & 45 & 52 & 52 \\ 
%\bottomrule 
%\end{tabular} 
%\caption{Weyl orbit and weight string table for $V(2\lam1+\lam3)$ with $l=4$.}
%\end{table}
%
%
%If $G$ has no regular orbit on $V$, then 
%\[
%q^{103}\leq 2|\pgl _5(q)|q^{103-45}+ 2\log (\log_2q+2)q^{25/2+103/2} .
%\]
%This is a contradiction for $q=5^e$, for all $e\geq 1$. Therefore, $G$ has a regular orbit on $V$.
%\end{proof}
\subsection{$\lambda=\lambda_1+\lambda_2+\lambda_3$, $l=3$}
\begin{proposition}
Theorem \ref{prestmain} holds for $G$ with $E(G)/Z(E(G)) \cong \mathrm{PSL}_4(q)$ acting on $V=V(\lam1+\lam2+\lam3)$. Namely, $G$ has a regular orbit on $V$.
\end{proposition}
\begin{proof}
Here $\dim V= 64-20\ep_{3}-6\ep_{5}$ and $V$ is preserved by graph automorphisms.
We determine that a graph automorphism $\tau$ has $\dim C_V(\tau) \leq 44-5\ep_{5}-14\ep_3$.
\begin{table}[!htbp]
\begin{tabular}{cccc}
\toprule 
$i$ & $\mu$ & $|W.\mu |$ & Mult \\ 
\midrule 
1 & $\lambda_1+\lambda_2+\lambda_3$& $24$ & 1 \\ 
2 &  $2\lambda_3$&4& $2-\ep_{3}$ \\ 
3 &  $2\lambda_1$&4& $2-\ep_{3}$ \\ 
4& $ \lambda_2$ &6&$4-\ep_{5}-2\ep_{3}$\\
\bottomrule 
\end{tabular}
\caption{Weyl orbit table for $V(\lambda_1+\lambda_2+\lambda_3)$, $l=3$.}
\end{table}
\begin{table}[!htbp]
\begin{tabular}{ccccccccc}
\toprule
 &  & \multicolumn{3}{c}{$c(s)$} & \multicolumn{4}{c}{$c(u_{\Psi})$} \\ 
 \midrule
String & Mult & $r=2$ & $r=3$ & $r\geq 5$ & $p=2$ & $p=3$ & $p=5$ & $p\geq 7$ \\ 
\midrule
$\m1\m1$ & 6 & 6 & 6 & 6 & 6 & 6 & 6 & 6 \\ 
$\m1\m2\m1$ & 2 & $4-2 \ep_{3}$ & 4 & 4 & 2 & 4 & 4 & 4 \\ 
$\m1\m3\m1$ & 2 & $4-2 \ep_{3}$ & 4 & 4 & 2 & 4 & 4 & 4 \\ 
$\m1\m4\m4\m1$ & 2 & $10-2\ep_{5}-4\ep_{3}$ & $12-2\ep_{5}$ & $12-2\ep_{5}-4\ep_{3}$ & 10 & 6 & 10 & 12 \\ 
$\m2\m4\m2$ & 1 & $4-\ep_{5}-2\ep_{3}$ & 4 & $4-2\ep_{3}$ & 2 & 2 & 4 & 4 \\ 
$\m3\m4\m3$ & 1 & $4-\ep_{5}-2\ep_{3}$ & 4 & $4-2\ep_{3}$ & 2 & 2 & 4 & 4 \\ 
\midrule
\textbf{Total} &  & $32-12\ep_{3}-4\ep_{5}$ & $34-2\ep_{5}$ & $34-8\ep_{3}-2\ep_{5}$ & 24 & 24 & 32 & 34 \\ 
\bottomrule 
\end{tabular} 
\caption{Weight string table for $V(\lambda_1+\lambda_2+\lambda_3)$, $l=3$.}
\end{table}

%Suppose $p\neq 2$.
If $G$ has no regular orbit on $V$, then 
\[
q^{d} \leq 2|\pgl_4(q)| q^{d- 32+12\ep_3+4\ep_{5}} + 2\log(\log_2 q+2)q^{16/2+d/2} + (2q^{15/2}+2q^{9})q^{44-5\ep_{5}-14\ep_3}
\]
%which gives a contradiction for $q\geq 3$.
%Finally, if $p=2$, then if $G$ has no regular orbit on $V$ we have
%\[
%q^{64} \leq 2|\mathrm{PGL}_4(q)| q^{64-20}+2\log(\log_2 q+2)q^{16/2+64/2} + (q^{10}+q^{9})q^{24} 
%\]
and this gives a contradiction for all $q\geq 2$.
\end{proof}

There are a number of highest weight modules $V(\lambda)$ with $\lambda$ in Table \ref{fulllist} for which the proof of Theorem \ref{main} is a straightforward application of Proposition \ref{tools}\ref{crude}) with the lower bounds on codimension computed from the corresponding weight string tables. For conciseness, we treat these cases together in the following proposition.
\begin{proposition}
	\label{ws_mods}
Theorem \ref{prestmain} holds for $G$ with $E(G)/Z(E(G)) \cong \mathrm{PSL}_{l+1}(q)$ acting on $V=V(\lambda)$ for $l$ and $\lambda$ given in Table \ref{condtable}. Namely, $G$ has a regular orbit on $V$ in each case.
\end{proposition}

{\footnotesize
\begin{table}[]
\begin{tabular}{@{}m{1.5cm}llllllllll@{}}
\toprule
&&& \multicolumn{3}{c}{$c(s)$} &&\multicolumn{3}{c}{$c(u_\Psi)$} &\\
\cmidrule{4-6} \cmidrule{8-10}
$\lambda$ & $l$ & $\dim V(\lambda)$ & $r=2$ & $r=3$ & $r\geq 5$ &  & $p=2$ & $p=3$ & $p\geq 5$ & $\dim C_V(\tau)$ \\
\midrule
$\lam6$ & 11 & 924 & 252 & 252 & 252 &  & 252 & 252 & 252 & 430 \\
 & 12 & 1716 & 462 & 462 & 462 &  & 462 & 462 & 462 & -- \\
$\lam1+\lam4$ & 7 & $504-56\ep_5$ & $175-15\ep_5$ & $175-15\ep_5)$ & $175-15\ep_5$ &  & 155 & $175-15\ep_5$ & $175-15\ep_5$ & -- \\
$2\lam1+\lam2$ & 2 & 15 & 7 & 9 & 9 &  & -- & 8 & 9 & -- \\
 & 3 & 45 & 20 & 24 & 24 &  & -- & 22 & 24 &  \\
 & 4 & 105 & 43 & 50 & 50 &  & -- & 47 & 50 & -- \\
 & 5 & 210 & 79 & 90 & 90 &  & -- & 86 & 90 & -- \\
 $\lam1+2\lam2$ & 3& 60& 27&34&34&&-- & 31& 34 & --\\
$2\lam3$ & 5 & $175-34\ep_3$ & $74-23\ep_3$ & $80-10\ep_3$ & $80-10\ep_3$ &  & -- & 70 & 80 & $69-12\ep_3$ \\
$2\lam1+\lam3$ $(p=5)$ & 4 & $103$ & 45 & 52 & 52 &  & -- & -- & 52 & -- \\
$5\lam1$ & 1&6& 3 & 4 & 4&& -- &-- & 3 $(p\geq 7)$ & --\\
& 2 & 21 & 9 & 12 & 14 &  & -- & -- & 12 $(p\geq 7)$ & -- \\
 & 3 & 56 & 22 & 29 & 34 &  & -- & -- & 31 $(p\geq 7)$ & -- \\
$3\lam2$ & 3 & 50 & 22 & 28 & 30 &  & -- & -- & 30 & 15 \\
$\lam1+\lam4$ & 6 & $371-84\ep_{l-1}$&$80-5\ep_{l-1}$ & $80-5\ep_{l-1}$ &$80-5\ep_{l-1}$&&$70-5\ep_{l-1}$&$80-5\ep_{l-1}$&$80-5\ep_{l-1}$&--\\
$\lam1+\lam5$ & 7 & $630-112\ep_{l-1}$ & $140-6\ep_{l-1}$&$140-6\ep_{l-1}$&$140-6\ep_{l-1}$&&$125-6\ep_{l-1}$&$140-6\ep_{l-1}$&$140-6\ep_{l-1}$&--\\
$\lam1+\lam6$ & 8& $990-144\ep_{l-1}$ & $224-7\ep_{l-1}$ & $224-7\ep_{l-1}$ & $224-7\ep_{l-1}$ && $203-7\ep_{l-1}$ &$224-7\ep_{l-1}$&$224-7\ep_{l-1}$ & --\\
%$\lam2+\lam3$ & 4 & $75-\ep_2-24\ep_3$ & $34-15\ep_3$ & $34-8\ep_3$ & $34-8\ep_3$ &  & 34 & 26 & 34 & $58-18\ep_3$ \\
% & 5 & $210-6\ep_2-84\ep_3$ & $91-46\ep_3$ & $91-30\ep_3-\ep_2$ & $91-30\ep_3-\ep_2$ &  & 90 & 61 & 91 & -- \\
% & 6 & 266 & 90 & 120 & 120 &  & -- & 120 & -- &  \\
% & 7 & 504 & 161 & 211 & 211 &  & -- & 211 & -- &  \\
\bottomrule
 \end{tabular}
\caption{Lower bounds for codimensions of eigenspaces in some $p$-restricted modules $V(\lambda)$.}
\label{condtable}
\end{table}
}
\begin{proof}
Using weight string tables for each case (omitted here), we compute lower bounds on $c(s)$ for prime order semisimple elements $s \in G$ and $c(u_\Psi)$ for unipotent elements $u_\Psi \in G$. We also, if applicable, bound the size of fixed point space of an involutory graph automorphism using Proposition \ref{graphfix}. For a given $\lambda$ and $l$, let $m$ be the minimum value of the $c(s)$ and $c(u_\Psi)$ listed in the row corresponding the $\lambda$ and $l$ in Table \ref{condtable}. Also let $m_g$ be equal to the entry listed in the $\dim C_V(\tau)$ column if it exists, otherwise set $m_g=\infty$. If $G$ has no regular orbit on $V$, then applying Propositions  \ref{field} \ref{graph} and \ref{tools}\ref{qsgood}, we see that $q^{\dim V(\lambda)}$ is less than
\[
%q^{\dim V(\lambda)} \leq 
2|\pgl_{l+1}(q)|q^{\dim V(\lambda)-m}+ 2(2q^{(n^2+n)/2-1}+2q^{(n^2-1)/2})q^{\dim V(\lambda)-m_g} +2\log(\log_2q+2)q^{(l+1)^2/2+\dim V(\lambda)/2}.
\]
This gives a contradiction for every choice of $\lambda$ and $l$, so the result follows.
%Applying Propositions \ref{graphfix}, \ref{invols}, \ref{field} \ref{graph} and \ref{tools}\ref{qsgood}, we see that the result follows.
\end{proof}
\subsection{$\lambda = 2\lam2$, with $l \in [3,17]$}
\begin{proposition}
 \label{2l2}
Theorem \ref{prestmain} holds for $G$ with $E(G)/Z(E(G)) \cong \mathrm{PSL}_{l+1}(q)$ acting on $V=V(2\lam2)$ for $l \in [3, 17]$. Namely, $G$ has a regular orbit on $V$ if $l\geq 4$ or $l=3$ and $G$ is quasisimple, and $b(G)\leq 2$ otherwise.
\end{proposition}
\begin{proof}
Here $d = \binom{l+1}{2} ^2 - (l+1)\binom{l+1}{3} - \ep_{3} \binom{l+1}{4}$ and $p\geq 3$. The Weyl orbit and weight string tables are given in Tables \ref{2l2weyl} and \ref{2l2A1} respectively.\\
\begin{table}[!htbp]
\begin{tabular}{cccc}
\toprule 
$i$ & $\mu$ & $|W.\mu |$ & Mult \\ 
\midrule 
1 & $2\lambda_{2}$& $\binom{l+1}{2}$ & 1 \\ 
2 & $\lam1+\lam3$ & $3\binom{l+1}{3}$ & 1 \\
3 & $\lam4$ &$\binom{l+1}{4}$& $2-\ep_3$\\
\bottomrule 
\end{tabular}
\caption{Weyl orbit table of $V(2\lam2)$. \label{2l2weyl}}
\end{table}
\begin{table}[!htbp]
\begin{tabular}{ccccc}
\toprule 
 & & \multicolumn{2}{c}{$c(s)$} & $c(u_\Psi)$ \\ 
 \midrule 
String & Mult & $r=2$ & $r\geq 3$ & $p\geq 3$ \\ 
\midrule 
$\m1$ & $\binom{l-1}{2}+1$ & 0 & 0 & 0 \\ 
$\m1\m2\m1$ & $l-1$ & $l-1$ & $2(l-1)$ & $2(l-1)$ \\ 
$\m2$ & $3 \binom{l-1}{3}$ & 0 & 0 & 0 \\ 
$\m2\m2$ & $(l-1)^2$ & $(l-1)^2$ & $(l-1)^2$ & $(l-1)^2$ \\ 
$\m2\m3\m2$ & $\binom{l-1}{2}$ & $(2-\epsilon_3) \binom{l-1}{2}$ & $2 \binom{l-1}{2}$ & $2 \binom{l-1}{2}$ \\ 
$\m3$ & $\binom{l-1}{4}$ & 0 & 0 & 0 \\ 
$\m3\m3$ & $\binom{l-1}{3}$ & $(2-\ep_3)\binom{l-1}{3}$ & $(2-\ep_3)\binom{l-1}{3}$ & $(2-\ep_3)\binom{l-1}{3}$ \\ 
\midrule 
\textbf{Total} &  & $2 \binom{l+1}{3} - \ep_3 \binom{l}{3}$ & $2 \binom{l+1}{3}+ l-1 -\ep_3\binom{l-1}{3}$  & $2 \binom{l+1}{3}+ l-1 -\ep_3\binom{l-1}{3}$ \\ 
\bottomrule 
\end{tabular} 
\caption{Weight string table of $V(2\lam2)$. \label{2l2A1}}
\end{table}

If $l\geq 4$ and $G$ has no regular orbit on $V$, then 
\[
q^{d} \leq 2|\mathrm{PGL}_{n}(q)|q^{d-2 \binom{l+1}{3}- l+1 +\ep_3\binom{l-1}{3}} +4(q^{\binom{l+2}{2}-1} + q^{\binom{l+2}{2}-2})q^{d-2 \binom{l+1}{3} +\ep_3 \binom{l}{3}} +\field
\]
which gives a contradiction for $l\geq 5$ and all $q\geq 3$. It remains to consider $l=3,4$.

Suppose $l=4$, and let $\Psi = \langle \alpha_1, \alpha_3 \rangle$. The $\Psi$-net table is given in Table \ref{2L2_A1^2}.
\begin{table}[!htbp]
\begin{tabular}{cccccccc}
\toprule 
&&&&\multicolumn{2}{c}{$c(s)$} & \multicolumn{2}{c}{$c(u_\Psi)$} \\ 
\midrule 
$\nu$ & $n_1$ & $n_2$ & $n_3$ & Mult & $r=2$ & $r\geq 3$ & $p\geq 3$ \\ 
\midrule
0 & 1 & 0 & 0 & 2 & 0 & 0 & 0 \\ 
$2\om1$ & 2 & 1 & 0 & 1 & 1 & 2 & 2 \\ 
$2\om3$ & 2 & 1 & 0 & 1 & 1 & 2 & 2 \\ 
$2\om1+2\om3$ & 4 & 4 & 1 & 1 & 4 & 6 & 6 \\ 
$\om1$ & 0 & 2 & 0 & 1 & 1 & 1 & 1 \\ 
$\om3$ & 0 & 2 & 0 & 1 & 1 & 1 & 1 \\ 
$\om1+\om3$ & 0 & 4 & 0 & 3 & 6 & 6 & 6 \\ 
$2\om1+\om3$ & 0 & 4 & 2 & 1 & $4-\ep_3$ & $5-\ep_3$ & $-\ep_3$ \\ 
$\om1+2\om3$ & 0 & 4 & 2 & 1 & $4-\ep_3$ & $5-\ep_3$ & $4-\ep_3$ \\ 
\midrule
Total &  &  &  &  & $22-2\ep_3$ & $28-2\ep_3$ & $26-2\ep_3$ \\ 
\bottomrule 
\end{tabular} 
\caption{The $\Psi$-net table for $V(2\lam2)$ and $\Psi = \langle \alpha_1, \alpha_3\rangle$. \label{2L2_A1^2}}
\end{table}

If $G$ has no regular orbit on $V$, then by Propositions \ref{invols}, \ref{tools}\ref{qsgood}), \ref{sscounts} and \ref{unipcounts},
\[
q^d \leq 2|\pgl _5(q)|q^{d-28+2\ep_3} + 4(q^{14}+q^{13})q^{d-22+2\ep_3} + (2q^{9}+ q^{9})q^{d-20+4\ep_3} +2 \log(\log_2 q+2)q^{25/2 + d/2}
\]
This gives a contradiction for all $q\geq 3$.

Finally, suppose $l=3$. This case is examined in \cite[Proposition 3.1.3]{bigpaper} for $H$ a simple algebraic group of type $A_3$. The authors consider the equivalent action of a group of type $D_3$ acting on the symmetric square of its natural module. They present a vector which they subsequently prove is a regular orbit representative. Applying the same argument to the realisation of this module over $\mathbb{F}_q$, we see that if $G$ is quasisimple, then it has a regular orbit on $V$. Otherwise, $G$ has $b(G)\leq 2$ by Proposition \ref{k=1_leftovers}.
%When $l=3$, $G$ has a regular orbit on $V$ by \cite[Proposition 3.1.3]{bigpaper}.
\end{proof}

\subsection{$\lambda = \lam5$, with $l\in [9,14]$}
\begin{proposition}
Theorem \ref{prestmain} holds for $G$ with $E(G)/Z(E(G)) \cong \mathrm{PSL}_{l+1}(q)$ acting on $V=V(\lam5)$ for $l \in [9, 14]$. Namely, $G$ has a regular orbit on $V$.
\end{proposition}
\begin{proof}
Here $d = \binom{l+1}{5}$. The Weyl orbit and weight string tables are given in Table \ref{l5weyl}.\\

\begin{table}[!htbp]
\begin{tabular}{cccc}
\toprule 
$i$ & $\mu$ & $|W.\mu |$ & Mult \\ 
\midrule 
1 & $\lambda_{5}$& $\binom{l+1}{5}$ & 1 \\ 
\bottomrule 
\end{tabular}
\quad
\begin{tabular}{cccc}
\toprule 
 &  & $c(s)$ & $c(u_\Psi)$ \\ 
 \midrule 
String & Mult & $r\geq 2$ & $p\geq 2$ \\ 
\midrule 
$\m1$ &$ \binom{l+1}{5}-2\binom{l-1}{4}$ & 0 & 0 \\ 
$\m1\m1$ &$\binom{l-1}{4}$&$\binom{l-1}{4}$&$\binom{l-1}{4}$ \\ 
\midrule 
\textbf{Total} &  & $\binom{l-1}{4}$&$\binom{l-1}{4}$ \\ 
\bottomrule 
\end{tabular} 
\caption{The Weyl orbit and weight string table for $V(\lam5)$.\label{l5weyl}}
\end{table}

If $l>9$ and $G$ has no regular orbit on $V$ then
\[
q^{\binom{l+1}{5}}\leq 2|\mathrm{PGL}_{l+1}(q)|q^{\binom{l+1}{5}-\binom{l-1}{4}}+ 2\log(\log_2q+2)q^{n^2/2+\tfrac{1}{2}\binom{l+1}{5}}
\]
This gives a contradiction for $l\geq 10$ and all $q\geq 2$, so it remains to consider $l=9$.
When $l=9$, we have $d=252$, and $V = V(\lam5)$ is preserved by graph automorphisms in $\mathrm{Aut}(\mathrm{PSL}_{10}(q))$.
%Let $\eta$ be a graph-field automorphism. Then 
%\[
%|\eta^G| \leq 2\log _p q \frac{|\mathrm{PGL}(10,q)|}{|\mathrm{PGU}(10,q^{1/2})|}< q^{51}.
%\]
Using Proposition \ref{graphfix}, we determine that a graph automorphism $\tau$ has $\dim C_V(\tau) \leq 142$.
We give the $\Psi$-net table for $\Psi$ of type $A_1^2$ in Tables \ref{l5A12}.\\
\begin{table}[!htbp]

\centering
\begin{tabular}{ccccc}
\toprule 
& & & $c(s)$ & $c(u_\Psi)$ \\ 
\midrule 
$\nu$ & $m_1$ & Mult & $r\geq 2$ & $p\geq 2$ \\ 
\midrule 
0 & 1 & 52 & 0 & 0 \\ 
$\om1$ & 2 & 30 & 30 & 30 \\ 
$\om3$ & 2 & 30 & 30 & 30 \\ 
$\om1+\om3$ & 4 & 20 & 40 & 40 \\ 
\midrule 
\textbf{Total} &  &  & 100 & 100 \\ 
\bottomrule 
\end{tabular} 
\caption{$\Psi = \langle \alpha_1,\alpha_3 \rangle$.\label{l5A12}}

%\begin{minipage}{0.48\textwidth}
%\centering
%\begin{tabular}{ccccc}
%\toprule 
% &  &  & $c(s)$ & $c(u_\Psi)$ \\ 
%\midrule 
%$\nu$ & $n_1$ & Mult & $r\geq 2$ & $p\geq 2$ \\ 
%\midrule 
%0 & 1 & 24 & 0 & 0 \\ 
%$\om1$ & 2 & 14 & 14 & 14 \\ 
%$\om3$ & 2 & 14 & 14 & 14 \\ 
%$\om5$ & 2 & 14 & 14 & 14 \\ 
%$\om1+\om3$ & 4 & 8 & 16 & 16 \\ 
%$\om1+\om5$ & 4 & 8 & 16 & 16 \\ 
%$\om3+\om5$ & 4 & 8 & 16 & 16 \\  
%$\om1+\om3+\om5$ & 8 & 6 & 24 & 24 \\ 
%\midrule 
%\textbf{Total} &  &  & 114 & 114 \\ 
%\bottomrule 
%\end{tabular} 
%\caption{$\Psi = \langle \alpha_1,\alpha_3, \alpha_5 \rangle$.\label{l5A13}}
%\end{minipage}
\end{table}

%Therefore, the only elements of $\mathrm{Aut}(\mathrm{PSL}(10,q))$ which may have a fixed point space of dimension greater than 615 are unipotent root elements, and semisimple elements  with centraliser type $A_{l-1}$. By Proposition\ref{root}, there are less than $q^{19}$ of the former and $ q^{20}$ of the latter.
Therefore, if $G$ has no regular orbit on $V$, then by Propositions \ref{tools}\ref{qsgood}, \ref{sscounts} and \ref{unipcounts}, 
\[
q^{252} \leq 2|\pgl_{10}(q)|q^{152}+4q^{19}q^{d-70}+2\log(\log_2q+2)q^{50+126}+2(2q^{54}+2q^{99/2})q^{142}
\]
This gives a contradiction for all $q\geq 2$, so $G$ has a regular orbit on $V$.
\end{proof}
\subsection{$\lambda = 2 \lambda_1 + \lambda_{l}$}
\begin{proposition}
Theorem \ref{prestmain} holds for $G$ with $E(G)/Z(E(G)) \cong \mathrm{PSL}_{l+1}(q)$ acting on $V=V(2\lam1+\lam l)$ for $l \in [2, \infty)$. Namely, $G$ has a regular orbit on $V$.
\end{proposition}
\begin{proof}
Here $d= 3 \binom{l+2}{3}+\binom{l+1}{2}-\ep_{l+2} (l+1)$
.The Weyl group has three orbits on the weights of the module, which are summarised in Table \ref{2l1+llweyl}.
\begin{table}[!htbp]
\begin{tabular}{cccc}
\toprule
$i$ & $\mu$ & $|W.\mu |$ & Multiplicity \\ 
\midrule 
1 & $2 \lambda_1+\lambda_{n-1}$ & $l^2+l$ & 1 \\ 
2 & $\lambda_2+\lambda_{n-1}$ & $3 \binom{l+1}{3}$ & 1 \\ 
3 & $\lambda_1$ & $l+1$ & $l - \ep_{l+2}$ \\ 
\bottomrule 
\end{tabular} 
\caption{Weyl orbit table of $V(2\lam1 + \lam l)$.\label{2l1+llweyl}}
\end{table}
The weight string table is given in Table \ref{ws2l1ll}. For conciseness, define $\xi = \ep_{l+2}$. 
\begin{table}[!htbp]
\begin{tabular}{cccccc}
\toprule
 &  & \multicolumn{2}{c}{$c(s)$} &  \multicolumn{2}{c}{$c(u_{\Psi})$} \\ 
\cmidrule{3-6}
String & Multiplicity & $r=2$ & $r\geq 3$ & $p=3$ & $p\geq 5$ \\ 
\midrule 
$\mu_1$ & $(l-1)(l-2)$ & 0 & 0 & 0 & 0 \\ 
 
$\mu_1 \mu_1$ &$ l-1$ & $l-1$ &$ l-1$ &$ l-1 $& $l-1$ \\ 

$\mu_1 \mu_2 \mu_1$ &$ l-1$ & $l-1$ & $2l-2 $& $2l-2$ &$ 2l-2 $\\ 

$\mu_1 \mu_3 \mu_3 \mu_1$ & 1 & $l+1-\xi $& $l+2-\xi $&$ l+1-\xi $& $l+2-\xi$ \\ 

$\mu_2$ & $3 \binom{l-1}{3}$ & 0 & 0 & 0 & 0 \\ 
 
$\mu_2 \mu_2$ & $3 \binom{l-1}{2}$ & $3 \binom{l-1}{2} $&$ 3 \binom{l-1}{2}$ & $3 \binom{l-1}{2}$ & $3 \binom{l-1}{2}$ \\ 

$\mu_2 \mu_3 \mu_2$ &$ l-1 $& $2l-2 $& $2l-2$ & $2l-2$ & $2l-2$ \\ 
\midrule 
Total & & $\frac{1}{2}(3l^2+l-2\xi)$ & $\frac{1}{2}(3l^2+3l-2\xi)$ & $\frac{1}{2}(3l^2+l-2\xi-2)$ & $\frac{1}{2}(3l^2+3l-2\xi)$ \\ 
\bottomrule 
\end{tabular} 
\caption{Weight string table for $V(2\lam1 + \lam l)$. \label{ws2l1ll}}
\end{table}

%Now, any non-trivial field automorphism  $\phi$ in $\mathrm{P}\Gamma \mathrm{L}_n(q)$ has $\mathrm{codim}C_v(\phi) \geq d/2$, which is larger than both $c(s)$ and $c(u_{\Psi})$ given in Table \ref{ws2l1ll}.
%%
%
% a fixed point space of size at most $q^{d/2}$, where $d = \dim V = 3 \binom{l+2}{3}+\binom{l+1}{2}-\ep_{l+2} (l+1)$. So $c(\phi) \geq d/2$, which is larger than both $c(s)$ and $c(u_{\Psi})$ for any $s, u_{\Psi} \in G/F(G)$. 
%So we have shown that every element of prime order in $\mathrm{P}\Gamma \mathrm{L}_n(q)$ has a fixed point space of size at most $d- \frac{1}{2}(3l^2+l-2\ep_{l+2})$.
If $G$ has no regular orbit on $V$, then by Propositions \ref{sscounts}, \ref{unipcounts} and \ref{invols},
\begin{align*}
q^d \leq & 2|\pgl_n(q)|q^{d- \frac{1}{2}(3l^2+3l-2\xi)}+ 4(q^{\binom{n+1}{2}-1}+q^{\binom{n+1}{2}-2})q^{d- \frac{1}{2}(3l^2+l-2\xi)}+ q^{n(n-1)}q^{d- \frac{1}{2}(3l^2+l-2\xi-2)}\\
&+2\log(\log_2q+2)q^{n^2/2+d/2}
\end{align*}
which gives a contradiction for $l\geq 3$ and $q\geq 3$, and also $l=2$, $q\geq 5$. Using more accurate counts of unipotent elements gives the result for $q=3$. So $G$ has a regular orbit on $V$.
%Actually l\geq 3 when ep=1, but if l=2, then ep is one iff p=2, which is a contradiction.
\end{proof}
\subsection{$\lambda = \lambda_1+\lambda_l$}
\begin{proposition}
Theorem \ref{prestmain} holds for $G$ with $E(G)/Z(E(G)) \cong \mathrm{PSL}_{l+1}(q)$ acting on $V=V(\lam1 + \lam l)$ with $l \in [2, \infty)$. Namely, $b(G) = 2$.
\label{adjoint}
\end{proposition}
\begin{proof}
%DONE
Here $d = (l+1)^2-1-\ep_{l+1}$.  If $p \mid l+1$, then $|V|<|G|$ and there is no regular orbit. So suppose $p \nmid l+1$. To prove that $G$ has no regular orbit on $V$, it is sufficient to show it for $H = \mathrm{PSL}_{l+1}(q)$.
Now, $V$ is the adjoint module for $H$ and the action of $H$  equivalent to conjugation on $n\times n$ matrices with trace 0. This action is rank preserving. The number of matrices of rank $k$ for  $0\leq k\leq n-1$ is less than $|H|$, so there cannot be a regular orbit here. Moreover, there is no regular orbit of $H$ on the rank $n$ matrices, since $\mathrm{GL}_n(q)$ has no conjugacy classes of size $|H|$ \cite[Theorem 6.4]{MR2888238}.
%Therefore, if there is a regular orbit of $G$ on $V$, it must be on $n\times n$ matrices of rank $n-1$. However, the number of trace zero matrices of rank $n-1$ is less than $2q^{n^2-2}$, which is less than $|H|$. 
So $H$ has no regular orbit on $V$, so neither does $G$.

We now show that $b(G)=2$. There are $q^{n-1}-1$ elements of $\mathbb{F}_{q^n}^\times $ with relative field trace 0 over $\mathbb{F}_q$. Since $n\geq 3$, we can always find $\xi \in \mathbb{F}_{q^n}^\times $ not contained in any proper subfield of $\mathbb{F}_{q^n}^\times $ with relative field trace 0. Therefore, let $A$ be a regular semisimple element of $\mathrm{GL}_n(q)$ with eigenvalues $\{\xi, \xi^q, \dots, \xi^{q^{n-1}}\}$ over $\mathbb{F}_{q^n}$. Then $A$ has matrix trace 0, and $|G_A|= |C_G(A)|< q^{n}$. We claim that $C_G(A)$ has a regular orbit on $V$.
Take $\Psi$ to be of type $A_1$. The relevant Weyl orbit and $\Psi$-net tables are given below.
\begin{table}[!htbp]
\begin{tabular}{cccc}
\toprule
$i$ & $\mu$ & $|W.\mu |$ & Multiplicity \\ 
\midrule 
1 & $\lam1+\lam l$ &  $ l^2+l$& 1 \\
2 & $0$ & 1& $l-\ep_{l+1}$\\
\bottomrule 
\end{tabular} 
\quad 
\begin{tabular}{ccccc}
\toprule
 &  & $c(s)$ & \multicolumn{2}{c}{$c(u_\Psi)$} \\ 
\midrule 
String & Mult & $r\geq 2$ & $p=2$ & $p\geq 3$ \\ 
\midrule 
$\m1$  & $2\binom{l-1}{2}$ & 0 & 0 & 0 \\ 
$\m1\m1$ & $2(l-1)$ & $2(l-1)$ & $2(l-1)$ & $2(l-1)$ \\  
$\m1\m2\m1$ & 1 & 2 & 1 & 2 \\ 
\midrule 
\textbf{Total} &    & $2l$ & $2l-1$ & $2l$ \\ 
\bottomrule 
\end{tabular} 
\caption{The Weyl orbit and weight string tables of $V(\lam1+\lam l)$.}
\end{table}

If $C_G(A)$ has no regular orbit on $V$, then 
\[
q^d \leq |C_G(A)| q^{d-2l} < q^{d-l+1}
\]
which gives a contradiction for all $l>1$. Therefore, $b(G)=2$ on the adjoint module.
\end{proof}

\subsection{$\lambda = \lam3$}

For $s \in \mathrm{GL}(V)$, we denote the $t$-eigenspace of $s$ on $V$ by $V_t(s)$.
\begin{proposition}
\label{l3}
Theorem \ref{prestmain} holds for $G$ with $E(G)/Z(E(G)) \cong \mathrm{PSL}_{l+1}(q)$ acting on $V=V(\lam3)$ with $l \in [5, \infty)$. Namely, one of the following holds.
\begin{enumerate}
	\item $l \geq 9$, or $l=8$ with $G$ quasisimple,  and $G$ has a regular orbit on $V$,
	\item $l=8$, $E(G) \neq G$ and $b(G)\leq 2$,
	\item $l\in [6,7]$ and $b(G)=2$, or 
	\item $l=5$ and we have $2 \leq b(G) \leq 3+\delta$, where $\delta=1$ if $G$ contains a graph automorphism, and $\delta=0$ otherwise. 
\end{enumerate}
% if $l\geq 9$ or $l\in [6,7]$, then $b(G) = 1$ and $2$ respectively. If $l=5$, then $2 \leq b(G) \leq 3+\delta$, where $\delta=1$ if $G$ contains a graph automorphism, and $\delta=0$ otherwise. 
\end{proposition}
\begin{proof}
Here $d = \binom{l+1}{3}$, and we note that $|V|<|G|$ for $5\leq l\leq 7$, so $b(G) \geq 2$ here. We give the Weyl and $\Psi$-net tables in Table \ref{l3table}. If $\Psi$ is of type $A_1$, $A_1^2$, $A_1^3$ or $A_2A_1^2$, then the $\Psi$-net tables are computed in \cite[Proposition 2.6.3]{bigpaper}, so in these cases, we give the  
lower bounds on $c(s)$ and $c(u_{\Psi})$ computed there. We also give the $\Psi$-net tables for $\Psi$ of type $A_2$, $A_3$ and $A_1^4$ as well as the $A_1^5$, $A_2A_2$ and $A_6$ tables for semisimple elements with $n=10$ in Table \ref{L3_n=10_tabs}. For conciseness, we give a condensed version of the $A_1^4$ table and the $A_1^5$ table for $n=10$. In each row of one of these tables, the indices $i,j$ and $k$ are distinct, and $i,j,k\in \{1,3,5,7\}$ or $\{1,3,5,7,9\}$ for $\Psi=A_1^4$ and $A_1^5$ respectively. The entries in the $c(s)$ and $c(u_\Psi)$ columns give the total contribution of the $\Psi$-nets with highest weight of the given type. 

Suppose $s\in G$ is semisimple of projective prime order $r$, and write the eigenvalues of $s$ on the $n$-dimensional natural module $W$ for $E(G)$ over $\overline{\mathbb{F}}_q$ as $t_1, t_2, \dots , t_m$, ordered so that the multiplicities $a_i$ of the $t_i$ are weakly decreasing. Since $V$ is isomorphic to the exterior cube of $W$,  $\dim C_V(s) \leq \frac{1}{6}n\sum a_i^2$ by Proposition \ref{eigsp_from_permsquare}, and by \cite[Table B.3]{bg}, $|s^{\mathrm{PGL}_n(q)}|< 2^m q^{n^2-\sum a_i^2}$, where $m$ is the number of distinct $t_i$. Therefore, 
\begin{align*}
\mathrm{codim}_{\mathbb{F}_q}C_V(s) - \log_q(|s^{\pgl_n(q)}|) \geq & \left(\binom{n}{3}-\frac{1}{6}n\sum a_i^2\right)-\left(n^2-\sum a_i^2+m\right)\\
&> \left(\frac{n}{6}-1\right)(n^2-\sum a_i^2) -\frac{n^2}{2} +\frac{n}{3}-m\numberthis \label{l3eqn}\\
\end{align*}
If $a_1 \leq n-4$, we have $n^2-\sum a_i^2\geq 8n-32$, so the right hand side of 
\eqref{l3eqn}  is at least $\frac{1}{6} (5n^2-84n+192)$. If $a_1 \geq n-3$, we instead note that  $\dim C_V(s) \leq \frac{1}{6}n\sum a_i^2$ and treat them separately in the inequality to come.
The number of semisimple elements of $\mathrm{PGL}_n(q)$ with an $(n-3)$-dimensional eigenspace is at most $4q^{6n-9}$, while the number of elements with an eigenspace of dimension $n-1$ or $n-2$ is counted in the proof of Proposition \ref{sscounts}.
By \cite[Lemma 3.7]{MR2888238}, the number of conjugacy classes in $\mathrm{PGL}_n(q)$ is at most $q^{n-1}+5q^{n-2}$. Therefore, if $n\geq 7$ and $G$ has no regular orbit on $V$, then
\begin{align*}
q^d \leq & 2(q^{n-1}+5q^{n-2})q^{d- \frac{1}{6} (5n^2-84n+192)} +4q^{6n-9}q^{d-(3l^2-21l+50)/2}+ 4q^{4n-4}q^{d-\binom{l-1}{2}}+q^{2n-1}q^{d-\binom{l-1}{2}}\\
&+q^{l(l+1)}q^{d-(3l^2/2-13l/2+9)}+ 8 \left(\frac{q^{2n^2/3+7/2}-1}{(q^2-1)(q^{3/2}-1)} \right)q^{d-(l^2-3l+2)}+4 \left(\frac{q^{n^2/2+2}-1}{q^2-1} \right)q^{d-(l^2-5l+8)}\\
&+ 2\log(\log_2q+2)q^{n^2/2+d/2}
\end{align*}
This is a contradiction for $n\geq 16$ and $q\geq 2$.
Using the $\Psi$-net tables given in Table \ref{l3table}, we determine that every unipotent element $u \in G$ of prime order has $\dim C_V(u) + \log_q(u^{\mathrm{PGL}_n(q)})\leq d -3$ for $10\leq n \leq 14$, with the exception of $u \in \mathrm{PGL}_{10}(q)$ with associated partition $(3^2,1^4)$. Here, we compute that $\dim C_V(u)=44$.  Moreover, if $n\geq 11$, we compute from Table \ref{l3table} that every semisimple $s$ element of $G$ of projective prime order has $\dim V_t(s) + \log_q(s^{\mathrm{PGL}_n(q)})\leq d -n$.
Therefore, if $G$ has no regular orbit on $V$ with $11\leq n \leq 14$,
\begin{align*}
q^d \leq & 2(q^{n-1}+5q^{n-2})q^{d- n} +4q^{6n-9}q^{d-(3l^2-21l+50)/2}+ 4q^{4n-4}q^{d-(l^2-5l+8)}+ 2q^{2n-1}q^{d-\binom{l-1}{2}}\\
&+c_p(n)q^{d-3}+ 2\log(\log_2q+2)q^{n^2/2+d/2}
\end{align*}
where $c_p(n)$ is the number of partitions of $n$ with parts of size at most $p$. This gives a contradiction for $n\geq 11$ and $q\geq 2$.

When $n=10$, we conduct a detailed analysis of conjugacy classes of semisimple elements of prime order in $G/F(G)$ including finding tighter upper bounds for the number of classes with given eigenvalue multiplicities. Using this technique and applying Proposition \ref{tools}\ref{eigsp2}), we see that $G$ has a regular orbit on $V$ for $q\geq 2$.

The case where $n=9$ is treated in \cite[Propositions 3.1.1, 3.1.4]{bigpaper}, where the authors construct a connected simple algebraic group $\hat{G}$ of type $A_8$ acting on $V(\lambda_3)$ as a subgroup of a simply connected group of type $E_8$ acting on its adjoint module $\mathfrak{L}$. They give an explicit vector $v \in \mathfrak{L}$ and show  that it is a representative of a regular orbit under the action of $\hat{G}$. We can apply the same argument to the realisation over $\mathbb{F}_q$, and find if $G$ is quasisimple, then there is a regular orbit of $G$ on $V$. If $l=8$ and $G$ is not quasisimple, then we show  $b(G)\leq 2$ in Proposition \ref{k=1_leftovers}.

Finally, for $6\leq n \leq 8$, we apply Proposition \ref{tools}\ref{crude} with the information in Table \ref{l3table} and Lemma \ref{fieldext} and find that $b(G) = 2$ for $n=7,8$ and $b(G)\leq 3$ for $n=6$ if $G$ contains no graph automorphisms, and $b(G)\leq 4$ if it does.
\end{proof}
%\renewcommand{\arraystretch}{1.7}
%\begin{table}[!htbp]
%\begin{tabular}{m{7cm}m{5cm}m{3cm}}
%\toprule
%Type of element & Bound on number of elts & Upper bound on $|C_V(g)|$ \\
%\midrule  
%Semisimple, $\Phi(s)$ of type $A_{n-1}$ & $q^{2n-1}$& $q^{d-\binom{l-1}{2}}$\\   
%Unipotent, Jordan form $(J_2, J_1^{n-2})$ & $q^{2n-1}$ & $q^{d-\binom{l-1}{2}}$\\  
% Unipotent $l=9$, Jordan blocks $(J_3^2,J_1^4)$ & $\displaystyle\frac{|\mathrm{GL}(10,q)|}{q^{32}|\mathrm{GL}_2(q)||\mathrm{GL}_4(q)|}$ & $q^{d-76}$\\ 
%     Semisimple, $\Phi(s)$ of codimension $4$ & $34 q^{4n^2/5+4}$&$q^{d-(2l^2-10l+16)}$\\   
%             Semisimple, $\Phi(s)$ of codimension $5$ & $67 q^{5n^2/6+5}$ & $q^{d-(5l^2-31l+62)/2}$\\   
%            Semisimple elements not mentioned above & $|\mathrm{PGL}_n(q)|$ & $q^{d-(3l^2-21l+46)}$\\   
%            Unipotent elements not mentioned above & $q^{n(n-1)}$ & $q^{d-(5l^2-31l+62)/2}$\\
%  \bottomrule
%\end{tabular}
%\caption{First of two detailed lists of elements and bounds on $|C_V(g)|$ for $V(\lambda_3)$ \label{l3table1}}
%\end{table}

\renewcommand{\arraystretch}{1}

\begin{table}[!htbp]
\begin{minipage}{0.4\textwidth}
\centering
\begin{tabular}{cccc}
\toprule
$i$ & $\mu$ & $|W.\mu |$ & Multiplicity \\ 
\midrule 
1 & $\lam3$ &  $ \binom{l+1}{3}$& 1 \\
\bottomrule 
\end{tabular} 
\caption{Weyl orbit table. \label{L3_weyl}}
\end{minipage}
\quad
\begin{minipage}{0.5 \textwidth}
\centering
\begin{tabular}{ccc}
\toprule 
 & $c(s)$ & $c(u_\Psi)$ \\ 
\midrule 
$\Psi$ & $r\geq 2$ & $p\geq 2$ \\ 
\midrule 
$A_1$ & $\binom{l-1}{2}$ & $\binom{l-1}{2}$ \\ 

$A_1^2$ & $l^2-5l+8$ & $l^2-5l+8$ \\ 
$A_1^3$ & $\frac{1}{2}(3l^2-21l+50)$  &$\frac{1}{2}(3l^2-21l+50)$\\
$A_2A_1^2$ & $2l^2-14l+34$ & $2l^2-14l+34$.\\
$A_3A_2$ & $\frac{1}{2}(5l^2-35l+78)$ &$\frac{1}{2}(5l^2-35l+78)$\\
\bottomrule 
\end{tabular}
\caption{Bounds for $c(s)$ and $c(u_\Psi)$.}

\end{minipage}

\begin{minipage}{0.4\textwidth}
\centering
\begin{tabular}{ccccc}
\toprule 
 &  & & $c(s)$ & $c(u_\Psi)$ \\ 
\midrule 
$\nu$ & $n_1$ & Mult & $r\geq 2$ & $p\geq 2$ \\ 
\midrule 
0 & 1 & $\binom{l-2}{3}+1$ & 0 & 0 \\ 
$\om2$ & 3 & $l-2$ & $2(l-2)$ & $2(l-2)$ \\ 
$\om1$ & 3 & $\binom{l-2}{2}$ & $2\binom{l-2}{2}$ & $2\binom{l-2}{2}$ \\ 
\midrule 
\textbf{Total} &  &  & $2\binom{l-1}{2}$ & $2\binom{l-1}{2}$ \\ 
\bottomrule 
\end{tabular} 
\caption{$\Psi$-net table for $\Psi = \langle \alpha_1, \alpha_2 \rangle$.}
\end{minipage}
\qquad
\begin{minipage}{0.5\textwidth}
\centering
\begin{tabular}{ccccc}
\toprule
 &  & & $c(s)$ & $c(u_\Psi)$ \\ 
\midrule 
$\nu$ & $n_1$ & Mult & $r\geq 2$ & $p\geq 2$ \\ 
\midrule 
0 & 1 & $\binom{l-3}{3}$ & 0 & 0 \\ 
$\om3$ & 4 & $1$ & 3 & 3 \\ 
$\om1$ & 4& $\binom{l-3}{2}$ & $3\binom{l-3}{2}$ & $3\binom{l-3}{2}$ \\ 
$\om2$ & 6& $l-3$& $4(l-3)$&$4(l-3)$\\
\midrule 
\textbf{Total} &  &  & $\frac{3l^2}{2}-\frac{13l}{2}+9$ & $\frac{3l^2}{2}-\frac{13l}{2}+9$ \\ 
\bottomrule 
\end{tabular} 
\caption{$\Psi$-net table for $\Psi = \langle \alpha_1, \alpha_2,\alpha_3 \rangle$.}
\end{minipage}
%\begin{minipage}{0.5\textwidth}
%\centering
%\begin{tabular}{ccccc}
%\toprule 
% &  & & $c(s)$ & $c(u_\Psi)$ \\ 
%\midrule 
%$\nu$ & $n_1$ & Mult & $r\geq 2$ & $p\geq 2$ \\ 
%\midrule 
%0 & 1 & $\binom{l-4}{3}$ & 0 & 0 \\ 
%$\om1$ & 5 & $\binom{l-4}{2}$ & $4\binom{l-4}{2}$ & $4\binom{l-4}{2}$ \\ 
%$\om{l-1}$ & 10& 1 & 8&8 \\ 
%$\om2$ & 10& $l-4$& $8(l-4)$&$8(l-4)$\\
%\midrule 
%\textbf{Total} &  &  & $4\binom{l-2}{2}+4$& $4\binom{l-2}{2}+4$\\ 
%\bottomrule 
%\end{tabular} 
%\caption{$\Psi$-net table for $\Psi = \langle \alpha_1, \alpha_2,\alpha_3, \alpha_4\rangle$.}
%\end{minipage}
\quad
\begin{minipage}{0.55\textwidth}
\centering
\begin{tabular}{ccccc}
\toprule 
 &  & & $c(s)$ & $c(u_\Psi)$ \\ 
\midrule 
$\nu$ & $n_1$ & Mult & $r\geq 2$ & $p\geq 2$ \\ 
\midrule 
0 & 1 & $\binom{l-7}{3}+4(l-7)$ & 0 & 0 \\ 
$\om i$ & 2 & $\binom{l-7}{2}+3$ & $4(\binom{l-7}{2}+3)$ & $4(\binom{l-7}{2}+3)$ \\ 
$\om i +\om j$ & 4& $l-7$ & $12(l-7)$&$12(l-7)$ \\ 
$\om i + \om j +\om k$ & 8& $1$& $16$&$16$\\
\midrule 
\textbf{Total} &  &  & $4\binom{l-4}{2} +16$& $4\binom{l-4}{2} +16$\\ 
\bottomrule 
\end{tabular} 
\caption{$\Psi$-net table for $\Psi = \langle \alpha_1, \alpha_3,\alpha_5, \alpha_7\rangle$.}
\end{minipage}
%\begin{minipage}{0.6\textwidth}
%\centering
%\begin{tabular}{ccccc}
%\toprule 
% &  & & $c(s)$ & $c(u_\Psi)$ \\ 
%\midrule 
%$\nu$ & $n_1$ & Mult & $r\geq 2$ & $p\geq 2$ \\ 
%\midrule 
%0 & 1 & $\binom{l-5}{3}$ & 0 & 0 \\ 
%$\lam1$ & 6 & $\binom{l-5}{2}$ & $5\binom{l-5}{2}$ & $5\binom{l-5}{2}$ \\ 
%$\lam2$ & 15& $l-5$ & $12(l-5)$&$12(l-5)$ \\ 
%$\lam3$ & 20& 1& 16&16\\
%\midrule 
%\textbf{Total} &  &  & $\frac{5l^2}{2}-\frac{31l}{2}+31$&$\frac{5l^2}{2}-\frac{31l}{2}+31$\\ 
%\bottomrule 
%\end{tabular} 
%\caption{$\Psi$-net table for $\Psi = \langle \alpha_1, \alpha_2,\alpha_3, \alpha_4,\alpha_5\rangle$.}
%\end{minipage}
%\begin{minipage}{0.35 \textwidth}
%\centering
%\begin{tabular}{cccc}
%\toprule
% &  & & $c(s)$  \\ 
%\midrule 
%$\nu$ & $n_1$ & Mult & $r\geq 2$\\ 
%\midrule 
%0 & 1 & $\binom{l-6}{3}$ & 0  \\ 
%$\lam1$ & 7 & $\binom{l-6}{2}$ & $6\binom{l-6}{2}$  \\ 
%$\lam2$ & 21& $l-6$ & $18(l-6)$\\ 
%$\lam3$ & 35& 1& 28\\
%\midrule 
%\textbf{Total} &  &  & $3l^2-21l+46$\\ 
%\bottomrule 
%\end{tabular} 
%\caption{$\Psi$-net table for $\Psi = \langle \alpha_1, \alpha_2,\alpha_3, \alpha_4,\alpha_5, \alpha_6\rangle$.}
%\end{minipage}
\caption{The Weyl and $\Psi$-net tables for $V(\lambda_3)$.\label{l3table}}
\end{table}

\begin{table}
	\begin{minipage}{0.32 \textwidth}
	\begin{tabular}{cccc}
		\toprule
		&  &  & $c(s)$ \\ \midrule
		$\nu$ & $n_1$ & Mult & $r\geq 2$ \\
		\midrule
		$\om i$ & 2 & 4 & 20 \\
		$\om i + \om j +\om k$ & 8 & 1 & 40 \\
		\midrule
		\textbf{Total} &  &  & 60 \\ \bottomrule
	\end{tabular}
\caption{$\Psi$-net table for $\Psi = \langle \a_1, \a_3, \a_5, \a_7, \a_9 \rangle$.}
\end{minipage}
\begin{minipage}{0.36\textwidth}
	\centering
\begin{tabular}{cccc}
	\toprule
	&  &  & $c(s)$ \\ \midrule
	$\nu$ & $n_1$ & Mult & $r\geq 2$ \\
	\midrule
	0 & 1 & 6 & 0 \\
	$\om1$ & 3 & 4 & 8 \\
	$\om2$ & 3 & 6 & 12 \\
	$\om4$ & 3 & 4 & 8 \\
	$\om5$ & 3 & 6 & 12 \\
	$\om2+\om4$ & 9 & 1 & 6 \\
	$\om1+\om5$ & 9 & 1 & 6 \\
	$\om2+\om5$ & 9 & 4 & 24 \\
	\midrule
	\textbf{Total} &  &  & 76 \\ \bottomrule
\end{tabular}
\caption{$\Psi$-net table for $\Psi = \langle \a_1, \a_2, \a_4, \a_5 \rangle$.}
\end{minipage}
\begin{minipage}{0.28 \textwidth}
\centering
\begin{tabular}{cccc}
\toprule
 &  & & $c(s)$  \\ 
\midrule 
$\nu$ & $n_1$ & Mult & $r\geq 2$\\ 
\midrule 
0 & 1 & $1$ & 0  \\ 
$\om1$ & 7 & $3$ & $18$  \\ 
$\om2$ & 21& $3$ & $54$\\ 
$\om3$ & 35& 1& 28\\
\midrule 
\textbf{Total} &  &  & $100$\\ 
\bottomrule 
\end{tabular} 
\caption{$\Psi$-net table for $\Psi = \langle \alpha_1, \alpha_2,\alpha_3, \alpha_4,\alpha_5, \alpha_6\rangle$.}
\end{minipage}
	\caption{Some $\Psi$-net tables for $V(\lam3)$ with $n=10$.\label{L3_n=10_tabs}}
\end{table}

\subsection{$\lambda = 3\lambda_1$}
\begin{proposition}
\label{3l1}
Theorem \ref{prestmain} holds for $G$ with $E(G)/Z(E(G)) \cong \mathrm{PSL}_{l+1}(q)$ acting on $V=V(3\lam1)$ for $l \in [1, \infty)$. Namely, if $l\geq 3$ or $l=2$ and $G$ is quasisimple, then $G$ has a regular orbit on $V$. If $l=2$ and $G$ is not quasisimple, then $b(G)\leq 2$. If $l=1$ then, letting $K$ denote  the kernel of the action of $\gl_2(q)$, $G$ has a regular orbit on $V$ for $G$ a subgroup of $\mathrm{GL}_2(q)/K$ of index at least 3,  and has $b(G)\leq 2$ otherwise, with equality if $G$ is a subgroup of index at most 2 in  $\mathrm{GL}_2(q)/K$.
\end{proposition}
\begin{proof}
Here $d = \binom{l+3}{3}$ and $p\geq 5$. The Weyl orbit table is given in Table \ref{3L1_tab_1}. The $\Psi$-nets for $\Psi$ of types $A_1$ and $A_1^2$ are given in \cite[Proposition 2.6.2]{bigpaper}, and so in Table \ref{3L1_tab_1} we give the bounds on $c(s)$ and $c(u_\Psi)$ from these tables. 
%Notice that for $\Psi = A_1^2$, we have $c(u_\Psi) > |\Phi| \geq \log_q(i_p(G))$, so we now focus on finding stronger lower bounds on $c(s)$ for semisimple elements of prime order.
\begin{table}[!htbp]
\begin{tabular}{cccc}
\toprule
$i$ & $\mu$ & $|W.\mu |$ & Multiplicity \\ 
\midrule 
1 & $3\lam1$ &  $l+1$& 1 \\
2& $\lam1+\lam2$ &$l(l+1)$& 1\\
3 & $\lam3$&$\binom{l+1}{3}$&1\\
\bottomrule 
\end{tabular} 
\quad 
\begin{tabular}{ccccc}
\toprule
&\multicolumn{3}{c}{$c(s)$} & $c(u_\Psi)$\\
\midrule 
$\Psi$ & $r=2$ & $r=3$ & $r\geq5$ & $p\geq 5$ \\ 
\midrule 
$A_1$ & $\frac{1}{2}(l^2+l+2)$ & $\frac{1}{2} l(l+3)$ & $\frac{1}{2} (l^2+3l+2)$ & $\frac{1}{2}(l^2+3l+2)$ \\ 
$A_1^2$ & $l^2-l+4$ & $l(l+1)$ & $l^2+l+2$ & $l^2+l+2$ \\ 
\bottomrule 
\end{tabular} 
\caption{The Weyl orbit table and some $\Psi$-net bounds for $V(3\lam 1)$\label{3L1_tab_1}}.
\end{table}

Let $\Psi = \langle \alpha_1, \alpha_2, \alpha_3 \rangle$, a subsystem of type $A_3$. Then the $\Psi$-net table for semisimple elements is given in Table \ref{3l1_A3}.

\begin{table}[!htbp]
\begin{tabular}{cccccccc}
\toprule
&&&&& \multicolumn{3}{c}{$c(s)$}\\
\midrule 
$\nu$ & $n_1$ & $n_2$ & $n_3$ & Mult & $r=2$ & $r=3$ & $r\geq 5$ \\ 
\midrule 
0 & 1 & 0 & 0 & $l-3$ & 0 & 0 & 0 \\ 
$3\om1$ & 4 & 12 & 4 & 1 & 15 & 15 & 16 \\  
0 & 0 & 1 & 0 & $2\binom{l-3}{2}$ & 0 & 0 & 0 \\  
$\om1$ & 0 & 4 & 0 & $l-3$ & $3(l-3)$ & $3(l-3)$ & $3(l-3)$ \\ 
$2\om1$ & 0 & 4 & 6 & $l-3$ & $6(l-3)$ & $8(l-3)$ & $8(l-3)$ \\ 
0 & 0 & 0 & 1 & $\binom{l-3}{3}$ & 0 & 0 & 0 \\ 
$\om1$ & 0 & 0 & 4 & $\binom{l-3}{2}$ & $3 \binom{l-3}{2}$ & $3 \binom{l-3}{2}$ & $3 \binom{l-3}{2}$ \\ 
\midrule 
\textbf{Total} &  &  & &  & $\frac{3}{2}(l^2-l+4)$ & $\frac{1}{2}(3l^2+l)$ & $\frac{1}{2}(3l^2+l+2)$ \\ 
\bottomrule 
\end{tabular} 
\caption{The $A_3$-net table of $V(3\lam1)$. \label{3l1_A3}}
\end{table}
So if $G$ has no regular orbit on $V$, then 
\begin{align*}
q^d \leq & 2(|\mathrm{PGL}_n(q)|- i_2(G)-i_3(G))q^{d-\frac{1}{2}(3l^2+l+2)}+2q^{2n-1}q^{d-\frac{1}{2}(l^2+l+2)}+ 2(q^{2n}+q^{2n-1})q^{d-\frac{1}{2}(l^2+3l)}\\
& + 4(q^{\binom{n+1}{2}-1}+ q^{\binom{n+1}{2}-2})q^{d-(l^2-l+4)} + 4(q^{\frac{1}{3}n(2n+1)-1}+ q^{\frac{1}{3}n(2n+1)-2})q^{d-l(l+1)}\\
& +2 \log (\log_2 q +2) q^{n^2/2+d/2}+ 2( \athree)q^{d-(l^2+l+2)} + q^{n(n-1)}q^{d-(l^2+l+2)}.\\
\end{align*}
This gives a contradiction for $l\geq 4$ and $q\geq 5$.
Now suppose $l=3$. The number of non-regular semisimple elements in $\pgl_4(q)$ is at most
\begin{align*}
(q-1)& \frac{|\mathrm{GL}_4(q)|}{|\mathrm{GL}(1,q)||\mathrm{GL}_3(q)|} + \binom{q-1}{2} \frac{|\mathrm{GL}_4(q)|}{|\mathrm{GL}_2(q)||\mathrm{GL}(1,q)|^2} +(q-1)\frac{|\mathrm{GL}_4(q)|}{|\mathrm{GL}_2(q)|^2}+ \frac{q+1}{2}  \frac{|\mathrm{GL}_4(q)|}{|\mathrm{GL}(2,q^2)|} \\
& +  \frac{q+1}{2}  \frac{|\mathrm{GL}_4(q)|}{|\mathrm{GL}(1,q^2)||\mathrm{GL}_2(q)|}< \frac{1}{2}(q^{12}+3q^9)
\end{align*}
If $G$ has no regular orbit on $V$, then
\begin{align*}
q^{20}\leq& 2|\pgl_4(q)|q^{20-16} + 2(\frac{1}{2}q^{12}+\frac{3}{2}q^{9})q^{20-14} + q^{12}q^{20-14} + 2(q^7+q^7)q^{20-9} +\frac{|\mathrm{GL}_4(q)|}{|\mathrm{GL}(1,q)||\mathrm{GL}_3(q)|} q^{20-7}\\
& + 2\log(\log_2q+2)q^{16/2+20/2}+ 4(q^9+q^8)q^{20-10}+ 4(q^{11}+q^{10})q^{20-12}.
\end{align*}
%#reg ss=
%\left(  \frac{q^2}{4} \frac{|\mathrm{GL}_4(q)|}{|\mathrm{GL}(1,q^4)|}+ \frac{q^2+q}{3}\frac{|\mathrm{GL}_4(q)|}{|\mathrm{GL}(1,q^3)||\mathrm{GL}(1,q)|}+   \binom{\frac{q+1}{2}}{2} \frac{|\mathrm{GL}_4(q)|}{|\mathrm{GL}(1,q^2)|^2}+ \binom{q-1}{3} \frac{|\mathrm{GL}_4(q)|}{|\mathrm{GL}(1,q)|^4} \right)
This gives a contradiction for $q\geq 13$. We use GAP \cite{GAP4} to compute the number of elements of $\mathrm{PGL}_4(q)$ of each prime order for $q\leq 11$, and substituting these into the inequality we have a contradiction in each case, so $G$ has a regular orbit on $V$.

Suppose now that $l=2$. This case is treated in \cite[Proposition 3.1.2]{bigpaper}, where the authors construct a group of type $A_2$ acting on $V(3\lam1)$ as a subgroup of a simply connected group of type $D_4$ acting on its adjoint module. They prove the existence of a regular orbit by providing an explicit vector with trivial stabiliser. Applying the same argument to the realisation of the module over $\mathbb{F}_q$, we see that if $G$ is quasisimple, then $G$ has a regular orbit on $V$. Otherwise, we show in Proposition \ref{k=1_leftovers} that $b(G)\leq 2$.
%If $l=2$, then $G$ has a regular orbit on $V$ by \cite[Proposition 3.1.2]{bigpaper}.

Finally, we consider $l=1$. Let $K$ be the kernel of the action of $\mathrm{GL}_2(q)$ on $V$. 
Note that $|K| = (3,q-1)$, and let $H= \F_q^{\times}\circ (\mathrm{GL}_2(q)/K)$. Let $v_j = e_1^2e_2+je_2^3$ for $j\in \F_q$ and $z$ be a generator of $\F_q^{\times}$.
 Define  $S_1=\{v_z, v_{z^3},v_{z^5}\}$, and $S_2=\{v_1, v_{z^2}, v_{z^4}\}$.
 Assume that  $q \equiv 1 \mod 4$. Then $e_1^2 e_2$ is a regular orbit representative for $\mathrm{SL}_2(q)$. Moreover, if 
$G \leq H$ does not contain the image of $\mathrm{Diag}(-1,1)$ in $\mathrm{GL}(V)$, then $S_1$  contains representatives of three distinct regular orbits of $G$. That is, each vector in $S_1$ has a stabiliser of order 2 in $H$.  Moreover, each element of $S_2$ lies in a distinct orbit of $H$ on $V$, and has a stabiliser of order 6 in $H$. Now suppose $q\equiv 3 \mod 4$. We again find that $e_1^2 e_2$ is a regular orbit representative for $\mathrm{SL}_2(q)$. Moreover, each vector in $S_2$ lies in a distinct orbit of $H$ of size $|H|/2$, while each element of $S_1$ lies in a distinct orbit of $H$ of size $|H|/6$.
%We find that each element of $S_1$ 

%. If $G \leq \F_q^{\times}\circ (\mathrm{GL}_2(q)/K)$ does not contain the image of $\mathrm{Diag}(-1,1)$ in $\mathrm{GL}(V)$, then $v_z, v_{z^3}$ and $v_{z^5}$ are regular orbit representatives of $G$ if $q \equiv 1(4)$ and each have a stabiliser of order 6 if instead $q\equiv 3 (4)$. Similarly, $v_1, v_{z^2}$ and $v_{z^4}$ are regular orbit representatives of $G$ if $q\equiv 3 (4)$ and each have a stabiliser of order 6 in $G$ if $q\equiv1 (4)$. 
We now show that there is no regular orbit of $\gl_2(q)/K$ (and therefore $H$) on $V$.
If $K$ is trivial, then the number of vectors of $V$ lying in orbits not already mentioned is less than $|\gl_2(q)|$, so there can be no regular orbit.
%
%If $K$ is trivial, we observe that there is no regular orbit of $\mathrm{GL}_2(q)$, since the number of vectors in orbits not already mentioned is less than $|\mathrm{GL}_2(q)|$.
% If $|K|=3$,  then the aforementioned regular orbits for $G$ fuse on $\F_q\circ G$. 

If instead $|K|=3$, we also determine that $e_1^3+ze_2^3$ and $e_1^3+z^2e_2^3$ lie in distinct $\gl_2(q)/K$ orbits and both have a stabiliser of size 3 in $\gl_2(q)/K$. The number of vectors in orbits not mentioned is then less than $|\gl_2(q)|/3$ for $q\geq 7$, so $\gl_2(q)/K$ has no regular orbit on $V$. We determine that the same is true for $q=5$ using an explicit construct of the module in GAP \cite{GAP4}.
In the cases where we have not shown that there is a regular orbit, we prove that $b(G)\leq 2$ in Proposition \ref{k=1_leftovers}.

%Finally, suppose $l=1$. Let $K$ be the kernel of the action of $\mathrm{GL}_n(q)$ on $V$, and let $Z$ be the set of non-cubes in $\mathbb{F}_q^\times$. 
%Note that $K$ has order 3 if $3 \mid q-1$ and $K$ is trivial otherwise, and that $|Z| = (q-1)/|K|$.
%We find that $v=e_1^2 e_2$ is a regular orbit representative for ${SL}_2(q)$. Now let $v_i = e_1^2e_2+ie_2^3$. If $G \leq Z.(\mathrm{GL}_2(q)/K)$ does not contain the image of $\mathrm{Diag}(-1,1)$ in $\mathrm{GL}(V)$, then $v_z$ is a regular orbit representative of $G$ if $q \equiv 1(4)$, and $v_1$ is a regular orbit representative otherwise. There is no regular orbit of $\mathrm{GL}_2(q)/K$, nor for any $G$ containing a field automorphism.
%In the cases where there is no regular orbit, we show that $b(G)=2$ in Proposition \ref{k=1_leftovers}.
\end{proof}
\subsection{$\lambda = \lambda_4$}
\begin{proposition}
\label{l4}
Theorem \ref{prestmain} holds for $G$ with $E(G)/Z(E(G)) \cong \mathrm{PSL}_{l+1}(q)$ acting on $V=V(\lam4)$ for $l \in [7, 28]$. Namely, if either $l\geq 8$ or $l=7$ and $G$ is quasisimple, then $G$ has a regular orbit on $V$. If $l=7$ and $G$ is not quasisimple, then $b(G)\leq 2$.
\end{proposition}
\begin{proof}
The Weyl group has one orbit on the weights of the module, and the Weyl orbit table is given in Table \ref{L4_weyl}. The analysis of $\Psi$-nets for $\Psi$ of type $A_1$, $A_1^2$ and $A_2$ is conducted in \cite[Proposition 2.6.4]{bigpaper}. Set $c=2(\binom{l-2}{3}+l-3)$. The values of $c(s)$ and $c(u_\Psi)$ computed in  \cite[Proposition 2.6.4]{bigpaper} are given in Table \ref{l4tab}.

%The Weyl orbit and weight string tables are given below.\\
\begin{table}[!htbp]
\begin{minipage}{0.3 \textwidth}
\centering
\begin{tabular}{cccc}
\toprule
$i$ & $\mu$ & $|W.\mu |$ & Multiplicity \\ 
\midrule 
1 & $ \lambda_4$ & $\binom{l+1}{4}$ & 1 \\
\bottomrule 
\end{tabular} 
\caption{The Weyl orbit table of $V(\lam4)$. \label{L4_weyl}}
\end{minipage}%
%\begin{tabular}{cccc}
%\toprule
%String & Multiplicity & $c(s)$ & $c(u_\Psi)$ \\ 
%\midrule 
%$\mu_1$ & $\binom{l+1}{4}-2\binom{l-1}{3}$ & 0 & 0 \\ 
%
%$\mu_1 \mu_1$ & $\binom{l-1}{3}$ & $\binom{l-1}{3}$ & $\binom{l-1}{3}$ \\ 
%\midrule 
%\textbf{Total} &  & $\binom{l-1}{3}$ & $\binom{l-1}{3}$ \\ 
%\bottomrule 
%\end{tabular} 
\begin{minipage}{0.65 \textwidth}
\centering
\begin{tabular}{c|cc|cc}
\toprule
$\Psi$ & $c(s)$ & $r\geq$ & $c(u_\Psi)$ & $p\geq$\\
\midrule 
$A_1$ & $\binom{l-1}{3}$ &2& $\binom{l-1}{3}$&2 \\ 
$A_1^2$ & $2(\binom{l-2}{3}+l-3)$ & 2&$2(\binom{l-2}{3}+l-3)$ & 2 \\ 
$A_2$ &  $2\binom{l-1}{3}$&3 &$2\binom{l-1}{3}$&3\\
\bottomrule 
\end{tabular} 
\caption{Some values of $c(s)$ and $c(u_\Psi)$ for $V(\lam4)$. \label{l4tab}}
\end{minipage}
\end{table}

If $G$ has no regular orbit on $V$, then
\begin{align*}
q^{\binom{l+1}{4}} \leq & 2|\pgl _n(q)| q^{d-2 \binom{l-1}{3}} + 2(\atwoeven)q^{d-c} +2(q^{2n-1}+q^{2n-1})q^{d-\binom{n-2}{3}}+  \itwo q^{d-c} \\
&+ \field.
\end{align*}
This gives a contradiction for $l=9$ and $q\geq 2$.
Now let $l=8$, and $\Psi = \langle \alpha_1,\alpha_2,\alpha_3 \rangle$. The $\Psi$-net table is given in Table \ref{l4a3}.

\begin{table}[!htbp]
\begin{tabular}{ccccc}
\toprule
 &  &  & $c(s)$ & $c(u_\Psi)$ \\ 
\midrule 
$\nu$ & $n_1$ & Mult & $r\geq 2$ & $p\geq 2$ \\ 
\midrule 
$\om1$ & 4 & 10 & 30 & 30 \\ 
$\om2$ & 6 & 10 & 40 & 40 \\ 
$\om3$ & 4 & 5 & 15 & 15 \\ 
0 & 1 & 6 & 0 & 0 \\ 
\midrule 
\textbf{Total} &  &  & 85 & 85 \\ 
\bottomrule 
\end{tabular} 
\caption{$A_3$-net table for $V(\lambda_4)$ and $l=8$.\label{l4a3}}
\end{table}

So if $l=8$ and $G$ has no regular orbit on $V$ then
\begin{align*}
q^{126} \leq & 2|\pgl _9(q)| q^{126-85} +2( \athree) q^{d-70} + 2(\atwoodd )q^{d-50} \\
&+ 2(q^{2n-1}+q^{2n-1})q^{d-35} + \itwo q^{d-50}+ \field
\end{align*}
This gives a contradiction for $l=8$ and all $q \geq 2$.

The case where $l=7$ is treated in \cite[Propositions 3.1.1, 3.1.4]{bigpaper}, where the authors construct a connected simple algebraic group $\hat{G}$ of type $A_7$ acting on $V(\lambda_4)$ as a subgroup of a simply connected group of type $E_7$ acting on its adjoint module $\mathfrak{L}$. They give an explicit vector $v+Z(\mathfrak{L}) \in \mathfrak{L}/Z(\mathfrak{L})$ and show that it is a representative of a regular orbit under the action of $\hat{G}$. We can apply the same argument to the realisation over $\mathbb{F}_q$, and find that if $G$ is quasisimple, then there is a regular orbit of $G$ on $V$. If instead $G$ is not quasisimple, then we show $b(G)\leq 2$ in Proposition \ref{k=1_leftovers}.

%It remains to show that there is a regular orbit when $l=7$; this follows from the proof of \cite[Propositions 3.1.1, 3.1.4]{bigpaper}.
\end{proof}

\subsection{$\lambda = \lambda_1 + \lambda_2$}

\begin{proposition}
\label{l1+l2}
Theorem \ref{prestmain} holds for $G$ with $E(G)/Z(E(G)) \cong \mathrm{PSL}_{l+1}(q)$ acting on $V=V(\lam1+\lam2)$ for $l \in [2, \infty)$. Namely, $G$ has a regular orbit on $V$ unless $l=3$, $q=3^e$ , where $b(G)=2$.
\end{proposition}
\begin{proof}
This module is of dimension $d = 2 \binom{l+2}{3}-\ep_{3} \binom{l+1}{3}$. The Weyl orbit table can be found in Table \ref{L1+L2_weyl}.

\begin{table}[!htbp]
\begin{tabular}{cccc}
\toprule
$i$ & $\mu$ & $|W.\mu |$ & Multiplicity \\ 
\midrule 
1 & $\lambda_1+\lambda_2$& $l(l+1)$ & 1 \\ 
2 &  $\lambda_3$&$\binom{l+1}{3}$& $2-\ep_{3}$ \\ 
\bottomrule 
\end{tabular} 
\caption{The Weyl orbit table of $V(\lam1+\lam2$. \label{L1+L2_weyl}}
\end{table}
First suppose $p\neq 3$. The $\Psi$-net tables for $\Psi$ of type $A_1$ and $A_2$ may be found in the proof of \cite[Proposition 2.6.6]{bigpaper}, and we summarise the bounds obtained there in Table \ref{l1l2tab}.

\begin{table}[!htbp]
\begin{tabular}{ccccc}
\toprule 
& \multicolumn{2}{c}{$c(s)$} & \multicolumn{2}{c}{$c(u_\Psi)$} \\ 
\midrule 
$\Psi$ & $r=2$ & $r\geq 3$ & $p=2$& $p=5$ \\ 
\midrule 
$A_1$ & $l^2$& $l^2$ & $l^2-l+1$ & $l^2$ \\ 
$A_2$ &--& $2l^2-2l+1$ & -- & $2l^2-2l+2$ \\ 
\bottomrule 
\end{tabular} 
\caption{Values of $c(s)$ and $c(u_\Psi)$ on $V(\lam1+\lam2)$ with $p \neq 3$ for different choices of $\Psi$.\label{l1l2tab}}
\end{table}

%For conciseness, define $\ep = \ep_{3}$. The weight string table is given in Table \ref{wsl1l2}.
%\begin{table}[!htbp]
%\begin{tabular}{ccccccc}
%\toprule
% & & \multicolumn{2}{c}{$c(s)$} & \multicolumn{3}{c}{$c(u_{\Psi})$} \\ 
%
%String & Mult & $r =2$ & $r \geq 3$ & $p=2$ & $p=3$ & $p\geq 5$ \\ 
%\midrule 
%$\mu_1$ & $(l-1)(l-2)$ & 0 & 0 & 0 & 0 & 0 \\ 
%
%$\mu_1 \mu_1$ & $l$ & $l$ & $l$ & $l$ & $l$ & $l$ \\ 
%
%$\mu_1 \mu_2 \mu_1$ & $l-1$ & $(2-\ep)(l-1)$ & $2(l-1)$ & $l-1$ & $2(l-1)$ & $2(l-1)$ \\ 
%
%$\mu_2$ & $\binom{l-1}{3}$ & 0 & 0 & 0 & 0 & 0 \\ 
%
%$\mu_2 \mu_2$ & $\binom{l-1}{2}$ & $(2-\ep)\binom{l-1}{2}$ & $(2-\ep)\binom{l-1}{2}$ & $2 \binom{l-1}{2}$ & $ \binom{l-1}{2}$ & $2 \binom{l-1}{2}$ \\ 
%\midrule 
%\textbf{Total} &  & $\frac{1}{2} l (\ep+(2-\ep)l)$ &$ l^2 - \ep \binom{l-1}{2}$& $l^2-l+1$ & $\frac{1}{2}(l^2+3l-2)$ &  $l^2$\\ 
%\bottomrule 
%\end{tabular} 
%\caption{Weight string table for $V(\lambda_1 + \lambda_2)$.\label{wsl1l2}}
%\end{table}

Now let $\Psi = \langle \alpha_1, \alpha_3 \rangle$ of type $A_1^2$. The corresponding $\Psi$-net table is given in Table \ref{a12l1l2}.
%\begin{table}[!htbp]
%\begin{tabular}{ccccm{2.5cm}m{2.5cm}m{2.5cm}m{2.5cm}}
%\toprule
% &  &  & & \multicolumn{2}{c}{$c(s)$} &  \multicolumn{2}{c}{$c(u_\Psi)$}  \\ 
% \midrule
%$\nu$ & $n_1$ & $n_2$ & Mult & $r=2$ & $r\geq 3$ & $p=2$ & $p\geq 3$ \\ 
%\midrule
%0 & 1 & 0 & $2\binom{l-3}{2}$ & 0 & 0 & 0 & 0 \\ 
%$\lam1$ & 2 & 0 & $l-2$ & $l-2$ & $l-2$ & $l-2$ & $l-2$ \\ 
%$\lam3$ & 2 & 0 & $l-2$ & $l-2$ & $l-2$ & $l-2$ & $l-2$ \\ 
%$2\lam1+\lam3$ & 4 & 2 & 1 & $4-\ep$ & $4-\ep$ & 4 & $4-\ep$ \\ 
%$2\lam1$ & 2 & 1 & $l-3$ & $2(l-3)$ & $2(l-3)$ & $l-3$ & $2(l-3)$ \\ 
%$2\lam3$ & 2 & 1 & $l-3$ & $2(l-3)$ & $2(l-3)$ & $l-3$ & $2(l-3)$ \\ 
%$\lam1+2\lam3$ & 4 & 2 & 1 & $4-\ep$ & $4-\ep$ & 4 & $4-\ep$ \\ 
%0 & 0 & 1 & $\binom{l-3}{3}$ & 0 & 0 & 0 & 0 \\ 
%$\lam1$ & 0 & 2 & $\binom{l-3}{2}$ & $\binom{l-3}{2}(2-\ep)$ & $\binom{l-3}{2}(2-\ep)$ & $2\binom{l-3}{2}$ & $\binom{l-3}{2}(2-\ep)$ \\ 
%$\lam3$ & 0 & 2 & $\binom{l-3}{2}$ & $\binom{l-3}{2}(2-\ep)$ & $\binom{l-3}{2}(2-\ep)$ & $2\binom{l-3}{2}$ & $\binom{l-3}{2}(2-\ep)$ \\ 
%$\lam1+\lam3$ & 0 & 4 & $l-3$ & $2(l-3)(2-\ep)$ & $2(l-3)(2-\ep)$ & $4(l-3)$ & $2(l-3)(2-\ep)$ \\ 
%\midrule
%Total &  &  &  & $2(l^2-3l+5)-\ep(l^2-5l+8)$ & $2(l^2-2l+3)-\ep(l^2-5l+8)$ & $2(l^2-3l+5)$ & $2(l^2-2l+2)-\ep(l^2-5l+8)$ \\ 
%\bottomrule 
%\end{tabular} 
%\caption{$A_1^2$-net table for $V(\lambda_1 + \lambda_2)$. \label{a12l1l2}}
%\end{table}
\begin{table}[!htbp]
\begin{tabular}{ccccm{2.5cm}m{2.5cm}m{2.5cm}}
\toprule
 &  &  & & $c(s)$ &  \multicolumn{2}{c}{$c(u_\Psi)$}  \\ 
 \midrule
$\nu$ & $n_1$ & $n_2$ & Mult &  $r\geq 2$ & $p=2$ & $p\geq 5$ \\ 
\midrule
0 & 1 & 0 & $2\binom{l-3}{2}$  & 0 & 0 & 0 \\ 
$\om1$ & 2 & 0 &  $l-2$ & $l-2$ & $l-2$ & $l-2$ \\ 
$\om3$ & 2 & 0 &  $l-2$ & $l-2$ & $l-2$ & $l-2$ \\ 
$2\om1+\om3$ & 4 & 2 & 1  & $4$ & 4 & $4$ \\ 
$2\om1$ & 2 & 1 & $l-3$ &  $2(l-3)$ & $l-3$ & $2(l-3)$ \\ 
$2\om3$ & 2 & 1 & $l-3$ &  $2(l-3)$ & $l-3$ & $2(l-3)$ \\ 
$\om1+2\om3$ & 4 & 2 & 1 & $4$ & 4 & $4$ \\ 
0 & 0 & 1 & $\binom{l-3}{3}$ & 0 & 0 & 0 \\ 
$\om1$ & 0 & 2 & $\binom{l-3}{2}$ & $2\binom{l-3}{2}$ & $2\binom{l-3}{2}$ & $2\binom{l-3}{2}$ \\ 
$\om3$ & 0 & 2 & $\binom{l-3}{2}$ &  $2\binom{l-3}{2}$ & $2\binom{l-3}{2}$ & $2\binom{l-3}{2}$ \\ 
$\om1+\om3$ & 0 & 4 & $l-3$ &  $4(l-3)$ & $4(l-3)$ & $4(l-3)$ \\ 
\midrule
Total &  &  &  & $2(l^2-2l+2)$ & $2(l^2-3l+5)$ & $2(l^2-2l+2)$ \\ 
\bottomrule 
\end{tabular} 
\caption{$A_1^2$-net table for $V(\lambda_1 + \lambda_2)$ with $p\neq 3$. \label{a12l1l2}}
\end{table}

%Suppose $p\neq 3$. Then if $G$ has no regular orbit on $V$, then 
%\begin{align*}
%q^d \leq & 2|\mathrm{PGL}_n(q)| q^{d-2(l^2-2l+3)} + 4(q^{\binom{n+1}{2}-1}+q^{\binom{n+1}{2}-2})q^{d-2(l^2-3l+5)}+ 2(q^{2n-1}+ q^{2n-1})q^{d-l^2+l-1}\\
%&+ 2\log(\log_2q+2)q^{n^2/2+d/2}
%\end{align*}
%This gives a contradiction for $l\geq 5$, so $G$ has a regular orbit on $V$ here.
%If $p\neq 3$ and $l=3,4$ or $p=3$, then these bounds are not sufficient, and we must take a larger subsystem $\Psi$.
%Let $\Psi = \langle \alpha_1, \alpha_2 \rangle$. The $\Psi$-net table can be found in Table \ref{A2l1l2}.
%\begin{table}[!htbp]
%\begin{tabular}{cccccc}
%\toprule 
% & &  &  & $c(s)$ & $c(u_\Psi)$ \\ 
%\midrule 
%$\nu$ & $n_1$ & $n_2$ & Mult & $r\geq 3$ & $p\geq 5$ \\ 
%\midrule 
%$\lam1+\lam2$ & 6 & 1 & 1 & 5 & 6 \\ 
%$2\lam1$ & 3 & 3 & l-2 & $6(l-2)$ & $6(l-2)$ \\ 
%$\lam1$ & 3 & 0 & $l-2$ & $2(l-2)$ & $2(l-2)$ \\ 
%$\lam1$ & 0 & 3 & $\binom{l-2}{2}$ & $2(l-2)(l-3)$ & $2(l-2)(l-3)$ \\ 
%0 & 1 & 0 & $2\binom{n-2}{2}$ & 0 & 0 \\  
%0 & 0 & 1 & $\binom{n-2}{3}$ & 0 & 0 \\ 
%\midrule 
%\textbf{Total} &  & &  & $2l^2-2l+1$ & $2l^2-2l+2$ \\ 
%\bottomrule 
%\end{tabular} 
%\caption{$A_2$-net table for $V(\lambda_1 + \lambda_2)$.\label{A2l1l2}}
%\end{table}
If $G$ has no regular orbit on $V$ then 
\begin{align*}
q^d \leq &2| \pgl _n(q)| q^{d-(2l^2-2l+1)} + 2(\atwoeven)q^{d-2(l^2-2l+2)}+4q^{2n-1}q^{d-l^2}+  4\left(\frac{q^{n^2/2+2}-1}{q^2-1}\right)q^{d-2(l^2-2l+2)}\\
&\itwo q^{d-l^2}+ \field.
\end{align*}
This gives a contradiction for $l\geq 5$ and $q\geq 2$, as well as $l=4$ with $q\geq 4$. When $(l,q)=(4,2)$, we replace $|\pgl_5(2)|$ with the number of elements of prime order in $\pgl_5(2)$ in the inequality above, and $\itwo$ with twice the number of involutions in 
$\pgl_5(2)$. This gives a contradiction and the result follows.
Now suppose that $l=3$ and $p\neq 3$. 
We compute using \cite[Table B.3]{bg} that the number of semisimple elements of prime order in $\pgl_4(q)$ with centralisers of type $A_2$, $A_1^2$ and $A_1$ is less than $q^7$, $\tfrac{3}{2}q^9$ and $\tfrac{1}{3}q^{12}$ respectively. Using GAP, we also compute that a regular semisimple element has all eigenspaces of dimension at most 4.
Therefore, if $G$ has no regular orbit on $V$, then by Propositions \ref{tools}\ref{qsgood}, \ref{invols} and \ref{unipcounts},
\begin{align*}
q^{20} \leq & 2|\pgl_4(q)|q^{20-16} + \frac{4}{3}q^{12} q^{20-13}+ 3 q^{9}q^{20-10}+ 2q^7q^{20-9} + 4(q^9+q^8)q^{20-10} + 4 \left(\frac{q^{10}-1}{q^2-1}\right) q^{20-10}\\
&+ q^{12}q^{20-14} + 2\log(\log_2 q+2)q^{8+10}+ \frac{(q^4-1)(q^3-1)}{q-1} q^{20-7}.
\end{align*}
This gives a contradiction for $q\geq 13$, so $G$ has a regular orbit on $V$. For the remaining $q$, an explicit construction of $G\leq \Gamma \mathrm{L}(V)$ in GAP \cite{GAP4} shows that $G$ has a regular orbit on $V$ here as well.

Now, for the remainder of the proof, suppose that $p=3$.
Here $d=2\binom{l+2}{3} - \binom{l+1}{3}$. The $\Psi$-net tables for $\Psi$ of type $A_1$, $A_1^2$ and $A_3$ are given in the proof of \cite[Proposition 2.6.6]{bigpaper} and we summarise the resulting lower bounds on $c(s)$ and $c(u_\Psi)$ in Table \ref{l1+l2p=3}.
\begin{table}[!htbp]
\begin{tabular}{cccc}
\toprule 
& \multicolumn{2}{c}{$c(s)$} & $c(u_\Psi)$ \\ 
\midrule 
$\Psi$ & $r=2$ & $r\geq 5$ & $p=3$ \\ 
\midrule 
$A_1$ & $\frac{1}{2} l(l+1)$ & $\frac{1}{2}(l^2+3l-2)$ & $\frac{1}{2}(l^2+3l-2)$ \\ 
$A_1^2$ & $l^2-l+2$ & $l^2+l-2$ & $l^2+l-2$ \\ 
$A_3$ & - & $\frac{1}{2}(3l^2+l-6)$ & - \\ 
\bottomrule 
\end{tabular} 
\caption{Values of $c(s)$ and $c(u_\Psi)$ on $V(\lam1+\lam2)$ for different choices of $\Psi$.\label{l1+l2p=3}}
\end{table}
If $G$ has no regular orbit on $V$,
\begin{align*}
q^d \leq & (2|\mathrm{PGL}_n(q)|)q^{d-\frac{1}{2}(3l^2+l-6)}+ 2(q^{2n-1}+q^{2n-1})q^{d-\frac{1}{2}l(l+1)}+ 4(q^{\binom{n+1}{2}-1}+q^{\binom{n+1}{2}-2})q^{d-(l^2-l+2)} \\
&+ 4(q^{\frac{1}{3} n(2n+1)-1}+q^{\frac{1}{3} n(2n+1)-2})q^{d-(l^2+l-2)} +2(\athree)q^{d-(l^2+l-2)}\\
&+ 2\log(\log_2q+2)q^{n^2/2+d/2}
\end{align*}
This gives a contradiction for $l\geq 6$ and $q\geq 3$, and also for $l=5$ when $q>9$, so $G$ has a regular orbit here. When $l=5$ and $q=3,9$, we explicitly compute the number of prime order elements of $\pgl _6(q)$, as well as the number of unipotent elements of prime order. These, along with the co-dimension bounds for $\Psi = A_1$, $A_1^2$ and $A_3$ show that $G$ has a regular orbit on $V$ by an application of Proposition \ref{tools}.
When $l=4$, we use \cite[Tables B.2, B.3]{bg} to more precisely count elements of prime order in $\pgl_5(q)$, including those enumerated in Propositions \ref{sscounts} and \ref{unipcounts}. Substituting this information into the inequality gives the result for $l=4$ and $p=3$.
Finally, if $l=3$, then there is no regular orbit of $G$ on $V$ by \cite{MR1409976}, so $b(G) \geq 2$.
%Now suppose $l=3$. 
%Cohen and Wales \cite{MR1409976} investigated the action of $\overline{G} = \mathrm{GL}(4, \overline{\mathbb{F}}_p)$ on $V^+$, constructed as a 16 dimensional vector space of monomials with coefficients from $\overline{\mathbb{F}}_p$. As part of their analysis, they compute the orbit lengths of $\mathrm{GL}_4(q)$ on the restriction of $V^+$ to $\mathbb{F}_q$, which is isomorphic to $V$ \cite[Table III]{MR1409976}. In particular, they show that there is no regular orbit. It remains to consider $G =\mathrm{SL}_4(q)$. If $G$ has a regular orbit on $V$, then from \cite[Table III]{MR1409976}, we see that the regular orbit must be a subset of the orbit in $\overline{G}$ of the polynomial denoted $f_1$. This polynomial has a finite stabiliser isomorphic to $S_5$ in $\overline{G}$, and the intersection of this stabiliser with $\mathrm{SL}(4,\overline{\mathbb{F}}_p)$ is isomorphic to $A_5$. An application of the Lang-Steinberg theorem gives that $G$ has no regular orbit on $V$. 
%Therefore, $b(G)\geq 2$.
Now, if $G$ has no base of size 2 in its action on $V$, then by Lemma \ref{fieldext},
\begin{align*}
q^{2\times 16} \leq & 2|\pgl _4(q)| q^{2(d-(l^2+l-2))} + 2(q^{2n-1}+q^{2n-1})q^{2(d-\frac{1}{2}l(l+1))}+ \itwo q^{2(d-(l^2-l+2))}\\
&+ 2\log(\log_2q+2)q^{n^2/2+d/2}
\end{align*}
This gives a contradiction for $q\geq 3$, so $b(G)=2$.
\end{proof}

 This concludes our analysis of highest weight modules $V=V(\lambda)$ over $\mathbb{F}_q$ with $p$-restricted highest weight $\lambda$.
\section{Proof of Theorem \ref{main}, III: Tensor product modules}
\label{tensor}
In this section, we aim to prove the following result which will complete the proof of Theorem \ref{main} for $G,V$ having the same underlying field $\mathbb{F}_q$.
\begin{theorem}
\label{nonprestmain}
Let $V=V(\iota_1)\otimes V(\iota_2)^{(p^a)}$ be a highest weight module defined over $\mathbb{F}_q$, $q=p^e$, with $\iota_1$, $\iota_2$ in Table \ref{comptensor} and $1\leq a < e$. Let $G \leq \Gamma \mathrm{L}(V)$ be almost quasisimple with $E(G)/Z(E(G)) \cong \mathrm{PSL}_{l+1}(q)$ such that the restriction of $V$ to $E(G)$ is absolutely irreducible. Then either $G$ has a regular orbit on $V$, or $\iota_1=\lam1$ and $\iota_2= \lam1$ or $\lam l$ up to quasiequivalence. In the latter case, $b(G)=2$ unless $l=1$, $(a,e)=1$, $p=2,3$ and $G=\mathrm{SL}_2(q)$, in which case there is a regular orbit.
\end{theorem}

%Before we prove our first result to reduce the list of modules given in Table \ref{comptensor}, we first give the Weyl orbit and $\Psi$-net tables for $\Psi$ of types $A_1$ and $A_1^2$.
%%TODO finish - can use a bigger subsystem to make next proposition better?
%\begin{proposition}
% Suppose $G$ is an almost quasisimple group with $E(G)/Z(E(G)) \cong \mathrm{PSL}_n(q)$ that acts faithfully and irreducibly on $V = V_1\otimes V_2$. Then of $G$ has no regular orbit on $V$, at least one of $V_1$ and $V_2$ is isomorphic to $V(\lambda_1)$, $V(\lambda_2)$ or their duals.
%\end{proposition}
%\begin{proof}
%First let $n\geq 19$. Suppose there exists $V = V_1\otimes V_2$ with neither $V_1$ nor $V_2$ isomorphic to $V(\lambda_1)$, $V(\lambda_2)$ or their duals, and $G$ has no regular orbit on $V$. Then by \cite[Theorem 1.1]{alvaro}, $\dim V_i > \binom{l+1}{2}$ for $i=1,2$. But then $\dim V  \geq  \binom{l+1}{2}^2 >n^3$ for $n\geq 5$. By Proposition \ref{n^3}, if $\dim V > n^3$, then $G$ has a regular orbit on $V$. So we have a contradiction, and the result is proved. For $n\leq 18$, we may find all tensor product modules $V$ with $\dim V\leq n^3$ by inspection of \cite[A.6.--A.21.]{MR1901354} and all have at least one of $V(\lam1)$, $V(\lam2)$ or their duals as a factor.
%\end{proof}

We begin by examining the modules with $\iota_1=\lambda_2$.
\subsection{Tensor product modules with $V(\lambda_2)$ as a factor}
\label{tensorl2}
In this subsection, we investigate the modules in Table \ref{comptensor} with $\iota_1=\lambda_2$. We assume that $\iota_2 \neq \lam1$ or $\lam l$ here; we treat tensor product modules with a factor quasiequivalent to $\lam1$ or $\lam l$ in Section \ref{tensorl1}. We also suppose in this section that $l>2$, because if $l=2$, then $V(\lambda_2) = V(\lam1)^*$ and $V(\lambda)$ is quasiequivalent to a module with $\iota_1=\lam1$. We begin by giving the Weyl orbit table and some $\Psi$-net tables for $V(\lambda_2)$ in Table \ref{L2_tabs}.
%To start, we give the Weyl orbit table for $V(\lambda_2)$ below, along with the $\Psi$-net tables for $\Psi$ of type $A_1$, $A_1^2$ and $A_2$.

\begin{table}[H]
\begin{minipage}{0.45\textwidth}
\centering
\begin{tabular}{cccc}
\toprule 
$i$ & $\mu$ & $|W.\mu |$ & Mult \\ 
\midrule 
1 & $\lambda_2$& $\binom{l+1}{2}$ & 1 \\ 
\bottomrule 
\end{tabular}
\caption{Weyl orbit table.}
\end{minipage}
\quad
\begin{minipage}{0.5\textwidth}
\centering
\begin{tabular}{ccccc}
\toprule 
&  & & $c(s)$ & $c(u_\Psi)$ \\ 
\midrule 
$\nu$ & $n_1$ & Mult & $r\geq 2$ & $p\geq 2$ \\ 
\midrule 
0 & 1 & $\binom{l+1}{2}-2(l-1)$ & 0 & 0 \\ 
$\lam1$ & 2 & $l-1$ & $l-1$ & $l-1$ \\ 
\midrule 
\textbf{Total} &  &  & $l-1$ & $l-1$ \\ 
\bottomrule 
\end{tabular} 
\caption{\textbf{$\Psi = \langle \alpha_1 \rangle$ }}
\end{minipage}\\
\vspace{1cm}
\begin{minipage}{0.45\textwidth}
\centering
\begin{tabular}{ccccc}
\toprule 
&  & & $c(s)$ & $c(u_\Psi)$ \\ 
\midrule 
$\nu$ & $n_1$ & Mult & $r\geq 2$ & $p\geq 2$ \\ 
\midrule 
0 & 1 & $\binom{l-2}{2}$ & 0 & 0 \\ 
$\om1$ & 3 & $l-2$ & $2(l-2)$ & $2(l-2)$ \\ 
$\om2$ & 3 & 1 & $2$ & $2$ \\ 
\midrule 
\textbf{Total} &  &  & $2l-2$ & $2l-2$ \\ 
\bottomrule 
\end{tabular} 
\caption{$\Psi = \langle \alpha_1,\alpha_2 \rangle$.}
\end{minipage}
\begin{minipage}{0.53\textwidth}
\centering
\begin{tabular}{ccccc}
\toprule 
&  & & $c(s)$ & $c(u_\Psi)$ \\ 
\midrule 
$\nu$ & $n_1$ & Mult & $r\geq 2$ & $p\geq 2$ \\ 
\midrule 
0 & 1 & $\binom{l+1}{2}-4(l-2)$ & 0 & 0 \\ 
$\om1$ & 2 & $l-3$ & $l-3$ & $l-3$ \\ 
$\om3$ & 2 & $l-3$ & $l-3$ & $l-3$ \\ 
$\om1+\om3$ &4&1&2&2 \\
\midrule 
\textbf{Total} &  &  & $2l-4$ & $2l-4$ \\ 
\bottomrule 
\end{tabular} 
\caption{$\Psi = \langle \alpha_1,\alpha_3 \rangle$.}
\end{minipage}
\caption{The Weyl orbit table and some $\Psi$-net tables for $V(\lam2)$.\label{L2_tabs}}
\end{table}

We are now ready to prove the following result.
\begin{proposition}
		\label{L2_tens_l=3}
Let $G$ be as in Theorem \ref{nonprestmain}, with $l \geq 3$ and $V = V(\lambda_2) \otimes V(\iota_2)^{(p^a)}$ with $\iota_2$ appearing in Table \ref{comptensor}.Then $G$ has a regular orbit on $V$. %except possibly when $l=3$ and $\iota_2=\lambda_2$.
%$V = V_1\otimes V(\lambda_2)^{p^a}$ and  $V = V_1\otimes V(\lambda_{l-1})^{p^a}$ appearing in Table \ref{comptensor} with $l\geq 4$. Namely, $G$ has a regular orbit on $V$ in all cases.
\end{proposition}
\begin{proof}
%By inspection of \cite[A.8.--A.21]{MR1901354} and \cite{alvaro}, we see that if $V_1$ is not the natural module (or dual or field automorphism image), then $\dim V_1\geq \binom{l+1}{2}$.
Let $d_2 = \dim V(\iota_2)$. From Table \ref{comptensor}, we see that $d_2 \geq \binom{l+1}{2}$.
If $G$ has no regular orbit on $V$, then from the $\Psi$-net tables for $V(\lam2)$ given above, as well as Propositions \ref{tensorcodim} and \ref{tools}\ref{alphabound}
\begin{align*}
q^{\binom{l+1}{2} d_2} \leq  &2|\pgl_n(q)|q^{d_2(\binom{l+1}{2}-(2l-4))}+\field + (q^{2n-1}+2 q^{2n-1})q^{d_2 ( \binom{l+1}{2}  - (l-1))}\\
& + (q^{(n^2+1)/2}+q^{(n^2+n)/2})q^{\lfloor \frac{l}{l+1} \binom{l+1}{2} d_2\rfloor}.
\end{align*}
This gives a contradiction for $d_2\geq \binom{l+1}{2}$, when $l\geq 4$ and $q\geq 2$ as required. The inequality is also false for $l=3$ and $d\geq 8$, leaving $\iota_2=\lambda_2$ to consider. So suppose $\iota_2=\lambda_2$.
Here we must take graph automorphisms $\tau$ into consideration, since they preserve the module. We compute that $\dim C_V(\tau) \leq 25$.
Therefore, if $G$ has no regular orbit on $V$, then by Propositions \ref{tensorcodim}, \ref{sscounts}, \ref{unipcounts} and \ref{graph}
\[
q^{36} \leq 2|\mathrm{PGL}_4(q)|q^{36-4\times 6}+ 8q^{10}q^{36-12}+ q^{12}q^{36-24}+4\left(\frac{q^{10}-1}{q^2-1}\right)q^{36-12} + 2\log(\log_2q+2)q^{26}+ (2q^{15/2}+2q^9)q^{25}.
\]
This gives a contradiction for all non-prime $q$, so $G$ has a regular orbit on $V$.
\end{proof}

\subsection{Tensor product modules with $V(\lambda_1)$ as a factor}
\label{tensorl1}
We begin with some discussion about the Weyl orbits and $\Psi$-nets of  $V(\lambda_1)$.

Let $V = V(\lambda_1)$. The weights of $V$ are of the form $\{\lam1 - \alpha_i \mid \alpha_i \in \Delta \} \cup \{\lam1 \}$, where $\Delta$ is a fixed base for $\Phi$. All of these weights lie in one orbit under the Weyl group.

Suppose $\Psi = \Psi_1 \Psi_2\dots \Psi_k$ is a standard subsystem of $\Phi$, where each of the $\Psi_i$ are irreducible. There is one $\Psi$-net for each $\Psi_i$ whose weight space sum is a copy of the natural module for $\overline{G}_{\Psi_i}$. Each of the remaining weights individually forms a $\Psi$-net.

%Then
%the set of $\Psi$-nets is composed of one copy of each of the natural modules of $\overline{G}_{\Psi_i}$, along with a copy of the trivial module for each simple root $\alpha_i \notin \Psi$.

Therefore, if $|\Delta_i|$ is the size of a base of $\Psi_i$, we find that for a semisimple element $s\in G$ of prime order with $\Psi \cap \Phi(s)=\emptyset$, we have 
\[
c(s) \geq \sum_{i=1}^{k} |\Delta_i|,
\]
and the same bound holds for $c(u_\Psi)$.

\begin{proposition}

Suppose $G\leq \Gamma \mathrm{L}(V)$ is almost quasisimple with $E(G)/Z(E(G))\cong \mathrm{PSL}_n(q)$, and that $E(G)$ acts absolutely irreducibly on $V = V(\lambda_1)\otimes V(\iota_2)^{(p^a)}$, where $1\leq a < e$ and $\iota_2$ appears in Table \ref{comptensor} . If either  $n\geq 4$ and $\dim V(\iota_2) \geq \frac{n^2+1}{2}$, or $n=3$ and $\dim V(\iota_2) \geq 6$, then $G$ has a regular orbit on $V$. 
\end{proposition}
\begin{proof}
Let $d_2=\dim V(\iota_2)$.
Note that for $n\geq 3$, $V$ is never fixed by a graph automorphism, and that $q$ is not prime.
If $G$ has no regular orbit on $V$ then by Propositions \ref{tensorcodim} \ref{sscounts} and \ref{unipcounts}, as well as the discussion preceding this proposition,
\[
q^{n d_2} \leq  2|\pgl_n(q)|q^{d_2 (n-2)} + 2\log(\log_2q+2)q^{n^2/2+n d_2/2}+(2q^{2n-1} + q^{2n-1}) q^{(n-1)d_2}\\
\]
This gives a contradiction for $d\geq (n^2+1)/2$ and  $n \geq 4$ and $q\geq 4$, and also $n=3$ with $d\geq 6$ and $q\geq 4$. 

%This gives a contradiction for $d\geq (n^2+1)/2$ and  $n \geq 4$ and $q\geq 2$, except for $(n,q)=(4,2)$, where $\mathrm{soc}(G/F(G)) \cong Alt_8$. It also gives a contradiction for $n=3$, and $q\geq 2$ for $d \geq 7$, and $q\geq 4$ for $d=6$. If $(d,q) = (6,3)$, then substituting in the number of prime order elements in $\pgl_3(3)$ and using the precise number of root elements and elements with $A_{l-1}$-type centralisers gives the result. When $(d,q)=(6,2)$, Proposition \ref{n^3} gives the result, since $L_3(2) \cong L_2(7)$.
\end{proof}

The remaining possibilities for $V(\iota_2)$ are given in the table below.
\begin{table}[!htbp]
\centering
\begin{tabular}{cccccc}
\toprule 
$\iota_2$ & $\lambda_1$ & $\lam l$ & $\lambda_2$&$\lam{l-1}$& $2\lambda_1$\\ 
\midrule 
$l$ & $[1, \infty)$ &$[1, \infty)$ & $[3, \infty)$&$[3, \infty)$&$1$  \\ 
\bottomrule 
\end{tabular} 
\caption{Remaining possibilities for $V(\iota_2)$.}
\end{table}
%\begin{proposition}
%Theorem \ref{nonprestmain} holds for $V = V((p^a+p^b+1)\lam1)$ and $l=1$. Namely, $G$ has a regular orbit on $V$.
%\end{proposition}
%\begin{proof}
%If $G$ has no regular orbit on $V$ then 
%\[
%q^8 \leq 2(q^2+q)(q^{\floor{2d/3}}+q^{\floor{d/3}})+ \frac{3}{2}|\mathrm{PGL}_2(q)|q^{\floor{d/2}}+2\log(\log_2q+2)q^{2+d/2}
%\]
%This gives a contradiction for all $q \geq 7$, and the result follows.
%\end{proof}

%\begin{proposition}
%Theorem \ref{nonprestmain} holds for $V = V((p^a+1)\lam1+ \lam2)$ and $l=2$. Namely, $G$ has a regular orbit on $V$.
%\end{proposition}
%\begin{proof}
%From the proof of Proposition \ref{l1+l2}, if $G$ has no regular orbit on $V$ then 
%\[
%q^{24-3\ep_3} \leq 2|\mathrm{PGL}_3(q)|q^{12}+ 2\log(\log_2q+2)q^{9/2+ \frac{1}{2} (24-3\ep_3)}+ 4(q^6+q^5)q^{15-3\ep_3}
%\]
%This gives a contradiction for all $q\geq 2$.
%\end{proof}
\begin{proposition}
Theorem \ref{nonprestmain} holds for $V = V(\lam1) \otimes V(2\lam1)^{(p^a)}$ and $l=1$. Namely, $G$ has a regular orbit on $V$.
\end{proposition}
\begin{proof}
Here $\dim V =6$. By the proof of \cite[Proposition 2.7.3]{bigpaper}, the codimension of the largest eigenspace of an element $x$ of prime order on $V(2\lam 1)$ is at least 2 if $x$ has odd order, and 1 otherwise.
%The Weyl orbit table for $V(2\lam1)$ and $\Psi$-net table for $\Psi = \langle \alpha_1\rangle$ and $\Psi = \langle \alpha_1, \alpha_2 \rangle$ are given below.
%\begin{table}[H]
%\begin{minipage}{0.5\textwidth}
%\centering
%\begin{tabular}{cccc}
%\toprule 
%$i$ & $\mu$ & $|W.\mu |$ & Mult \\ 
%\midrule 
%1 & $2\lam1$&$l+1$  & 1 \\ 
%2& $\lam2$ &$\binom{l+1}{2}$& 1\\
%\bottomrule 
%\end{tabular}
%\caption{Weyl orbit table.}
%\end{minipage}
%\quad
%\begin{minipage}{0.5\textwidth}
%\centering
%\begin{tabular}{ccccccc}
%\toprule 
% &  &  &  & \multicolumn{2}{c}{$c(s)$} & $c(u_\Psi)$ \\ 
%\midrule 
%$\nu$ & $n_1$ & $n_2$ & Mult & $r=2$ & $r\geq 3$ & $p\geq 3$ \\ 
%\midrule 
%0 & 1 & 0 & $l-1$ & 0 & 0 & 0 \\ 
%
%$2\lam1$ & 2 & 1 & 1 & 1 & 2 & 2 \\ 
%
%0 & 0 & 1 & $\binom{l-1}{2}$ & 0 & 0 & 0 \\ 
%
%$\lam1$ & 0 & 2 & $l-1$ & $l-1$ & $l-1$ & $l-1$ \\ 
%\midrule 
%\textbf{Total} &  &  &  & $l$ & $l+1$ & $l+1$ \\ 
%\bottomrule 
%\end{tabular} 
%\caption{\textbf{$\Psi = \langle \alpha_1 \rangle$ }}
%\end{minipage}
%\begin{minipage}{\textwidth}
%\centering
%\begin{tabular}{ccccccc}
%\toprule 
% &  &  &  & \multicolumn{2}{c}{$c(s)$} & $c(u_\Psi)$ \\ 
%\midrule 
%$\nu$ & $n_1$ & $n_2$ & Mult & $r=2$ & $r\geq 3$ & $p\geq 3$ \\ 
%\midrule 
%0 & 1 & 0 & $l-2$ & 0 & 0 & 0 \\ 
%
%$2\lam1$ & 3 & 3 & 1 & 3 & 4 & 5 \\ 
%
%0 & 0 & 1 & $\binom{l-2}{2}$ & 0 & 0 & 0 \\ 
%
%$\lam1$ & 0 & 3 & $l-2$ & $2(l-2)$ & $2(l-2)$ & $2(l-2)$ \\ 
%\midrule 
%\textbf{Total} &  &  &  & $2l-1$ & $2l$ & $2l+1$ \\ 
%\bottomrule 
%\end{tabular} 
%\caption{\textbf{$\Psi = \langle \alpha_1, \alpha_2\rangle$ }}
%\end{minipage}
%\end{table}
Therefore, if $G$ has no regular orbit on $V$, then by Propositions \ref{tools}\ref{eigsp2}), \ref{tensorcodim} and \ref{invols},
\[
q^6 \leq \frac{3}{2}|\pgl_2(q)|q^2+2(q^2+q)(2q^3)+2\log(\log_2q+2)q^{2+6/2}
\]
and this is a contradiction for $q\geq 9$.
\end{proof}
\begin{proposition}
Theorem \ref{nonprestmain} holds for $V = V(\lam1)\otimes V(\lam2)^{(p^a)}$ and $V = V(\lam1)\otimes V( \lam{l-1})^{(p^a)}$ for $l\geq 3$. Namely, $G$ has a regular orbit on $V$.
\end{proposition}
\begin{proof}
We will make use of the $\Psi$-net tables for $V(\lambda_2)$ given at the beginning of Section \ref{tensorl2}.
Suppose $l\geq4$, and set $d = \dim V = (l+1)\binom{l+1}{2}$. If $G$ has no regular orbit on $V$, then by Propositions \ref{sscounts}, \ref{unipcounts} and \ref{tensorcodim},
\begin{align*}
q^d \leq &2|\mathrm{PGL}_n(q)| q^{d-(2l-2)(l+1)}+ 8q^{n^2/2+2} q^{d-(2l-4)(l+1)}+ 4(\frac{q^{n^2/2+2}-1}{q^2-1}) q^{d-(2l-4)(l+1)}\\ 
&+ 4q^{2n-1}q^{d-(l-1)(l+1)}+ 2\log(\log_2q+2)q^{n^2/2+d/2}
\end{align*}
This gives a contradiction for all $q\geq 4$, so $G$ has a regular orbit on $V$. Now suppose that $l=3$. 
%Mixture of bound results for \lam1 and \lam2. Also use regular elt argument
If $G$ has no regular orbit on $V$ then by Proposition \ref{tensorcodim},
\begin{align*}
q^{24} \leq& 2|\mathrm{PGL}_4(q)|q^{24-16}+ (8q^{10}+ 4\left(\frac{q^{10}-1}{q^2-1}\right) )q^{24-12}+ 4q^7q^{24-8} +2\log(\log_2q+2)q^{20}\
\end{align*}
This gives a contradiction for all $q\geq 8$, so $G$ has a regular orbit on $V$ here. For $q \leq 8$, we replace $|\pgl_4(q)|$ with the number of elements of prime order in $\pgl_4(q)$, and also use more precise versions of Proposition \ref{sscounts} and \ref{unipcounts} in order to give a contradiction in the above inequality.
\end{proof}
\begin{proposition}
\label{tensornatural}
Theorem \ref{nonprestmain} holds for $V = V(\lam1)\otimes V(\lam1)^{(p^a)}$ and $V = V(\lam1)\otimes V(\lam{l})^{(p^a)}$ for $l\geq 3$. Namely, $b(G)=2$, unless $(a, \log_pq)=1$, $p=2$ or $3$ and $G= \mathrm{SL}_2(q)/K$, where $K$ is the kernel of the action of $\mathrm{SL}_2(q)$ on $V$. In these cases, $b(G)=1$. %Really want that [1,0][0,-1] not in there?
\end{proposition}
\begin{proof}
Let $q=p^e$. To show that there is no regular orbit of $G$ on $V$, it is sufficient to consider
 $G=\mathrm{SL}_n(q)/Z$, where $Z$ is the kernel of the action of $\mathrm{SL}_n(q)$. Let $b = \gcd(a,e)$, then $|Z| = \gcd(p^b+1,n)$ and $\gcd(p^b-1,n)$ for $\lambda = (p^a+1)\lam1$ and $p^a\lam1+\lam l $ respectively.
The action of $g \in G$ on $V(\lam1)\otimes V(\lam1)^{(p^a)}$ or $ V(\lam1)\otimes V(\lam{l})^{(p^a)}$  is equivalent to the action of $G$ on $n\times n$ matrices as 
\[
A \mapsto g^\top A g^{(p^a)} \quad \textrm{or} \quad A \mapsto g^{-1} A g^{(p^a)}
\]
respectively. These maps are rank and determinant preserving.
The number of matrices of rank $k$ ($1\leq k \leq n$) is (see \cite[p. 338]{MR1871828} for example):
\[
q^{\binom{k}{2}} \frac{\prod_{i=1}^k (q^{n-k+i}-1)^2}{\prod_{j=1}^k(q^j-1)}.
\]
This is less than $|G|$ for $1\leq k\leq n-2$. So if $G$ has a regular orbit on $V$, then it is composed of matrices of rank $n$ or rank $n-1$. We show that $G$ has no regular orbit on matrices of rank $n$ using Shintani descent (see \cite[\S 2.6]{MR3032138} for an overview).

Suppose $\lambda = \lam1+p^a\lam l$. Let $\sigma: x \mapsto x^{(p^{b})}$ be a field automorphism of $\mathrm{GL}_n(q)$, let $\varphi = \sigma^{c}: x \mapsto x^{(p^{a})}$, where $c=a/b$, and define the semidirect product $K=\mathrm{GL}_n(q) \rtimes \langle \sigma \rangle$. The stabiliser in $H=\mathrm{SL}_n(q) \leq K$ of $A \in \mathrm{GL}_n(q)$ under the action $A \mapsto g^{-1} A g^{(p^a)}$ is equal to $C_H(A\varphi^{-1})$. 
Let $\hat{c}$ be an integer such that $\hat{c} \equiv -c \mod e/b$ and such that $\gcd(\hat{c}, |K|)=1$. Then the map $f : K \to K$ defined by $x \mapsto x^{\hat{c}}$ is a bijection. Therefore, $A\varphi^{-1}$ has a preimage of the form $y\sigma \in K$. Moreover, we see that 
$C_H(A\varphi^{-1}) = C_H(y\sigma)$. Applying Shintani descent (see \cite[Lemma 2.13]{MR3032138}, for instance), $C_{H}(y\sigma)=C_{\mathrm{SL}_n(p^b)}(h)$ for some $h \in \mathrm{GL}_n(p^b)$. Since every element of $\mathrm{GL}_n(p^b)$ is centralised by a non-scalar in $\mathrm{SL}_n(p^b)$, $C_{H}(y\sigma)$ is not contained in the centre of $H$, and there is no regular orbit of $G$ on $V$. 
The argument for $\lambda = (p^a+1)\lam1$ is similar; there we define $\sigma$ to be a graph-field automorphism $x\mapsto (x^{(p^a)})^{-T}$.

It remains to consider whether there exists a regular orbit on the matrices of rank $n-1$. Suppose $n\geq 3$.
We first deal with $V(\lam1)\otimes V(\lam l)^{(p^a)}$. Let 
\[
S = \left\{ \left[\begin{array}{c|c}N& 0\\ \midrule 0&0\\\end{array} \right] \mid N \in \mathrm{GL}_{n-1}(q)\right\}.
\]
Note that every element of $S$ has non-trivial stabiliser in $\mathrm{SL}_n(q)/Z$. %Doesn't rely on rk n result because can just take diagonal in kernel and tack a 1 on the end in most cases. Also, stab of Diag(1,0) is mats of the form Diag(1,x).
 The setwise stabiliser of $S$ in $\mathrm{SL}_n(q)$ is 
\[
K = \left\{ \left[\begin{array}{c|c}A& 0\\ \midrule 0&b\\ \end{array} \right] \mid A \in \mathrm{GL}_{n-1}(q), b=\det A^{-1}\right\}.
\]

Let $S^G=\{ g^{-1}Ag^{(p^a)} \mid g \in \mathrm{SL}_n(q), A \in S\}$, and note that $\mathrm{SL}_n(q)$ has no regular orbit that includes an element of $S^G$. We now count the number of elements in $S^G$. Let $g_1, g_2, \dots, g_t$ be a set of right coset representatives of $\mathrm{SL}_n(q)/Z$; then the set of $g_i^{-1}Sg_i^{p^a}$ for $1\leq i \leq t$ partition $S^G$. We have $|g_i^{-1}Sg_i^{p^a}|=|S|$, so $|S^G| = |\mathrm{SL}_n(q)| |S|/|Z| =  |\mathrm{SL}_n(q)|$. The number of remaining matrices of rank $n-1$ is less than $2q^{n^2-2}$, which is less than $|\mathrm{SL}_n(q)|/|Z|$, so there can be no regular orbit. For $V=V(\lam1)\otimes V(\lam1)^{(p^a)}$, we repeat the same argument, this time setting 
\[
S = \left\{\right.
 \left(\begin{array}{ccc}
 N& 0 & v_1\\ 
 v_2 &0&c\\
 0 & 0 &0\\ \end{array}
  \right) \mid  N \in \mathrm{GL}_{n-2}(q), v_1, v_2^{\top} \in \mathbb{F}_q^{n-2}, c \in \mathbb{F}_q \left. \right\}.
\]
and find that there is no regular orbit under $\mathrm{SL}_n(q)/Z$ here either. 

Finally, if $n=2$, then the stabiliser of an element of $S$ contains $\mathrm{Diag}(\nu, \nu^{-1})$ for $\nu \in \F_{p^b}^\times$. All such elements lie in the kernel of the action if and only if $b=1$ and $p=2$ or 3. In both cases, $\mathrm{Diag}(1,0)$ is a regular orbit representative of $G$ on $V$. When $p=3$, $\mathrm{Diag}(j,0)$, with $j\in \F_q^\times$ a non-square is a representative of a second regular orbit of $G$ on $V$. The two regular orbits are not interchanged by a graph or diagonal automorphism. No group properly containing $G$ has a regular orbit on $V$, since there is no regular orbit of $G$ containing rank 2 matrices, and the number of remaining rank 1 matrices is $q^2-1$, which is less than the order of $G$.
Where we show that is no regular orbit, we  have $b(G)=2$ by Proposition \ref{k=1_leftovers}.
\end{proof}
This completes the proof of Theorem \ref{nonprestmain}, and therefore Theorem \ref{main} for modules with $k=1$. 
\section{Absolutely irreducible representations over subfields} 
\label{subfields}
%In this section we consider modules with $k<1$ as in Theorem \ref{main}.
%Define $c=1/k$, and note that $c$ is an integer.
In this section, we consider the following embedding of $\mathrm{SL}_{n}(q^c) \leq \mathrm{SL}_{n^c}(q)$ for $c\geq 1$. Let $W$ be the $n$-dimensional natural module for $\mathrm{SL}_{n}(q^c)$ over $\F_{q^c}$  with standard basis $\{e_1, \dots ,e_n\}$, and let 
\[
V = W\otimes W^{(q)}\otimes W^{(q^2)}\otimes \dots \otimes W^{(q^{c-1})},
\]
so that $\dim V= n^c$ and $B = \{e_{i_1} \otimes e_{i_2} \otimes \dots \otimes e_{i_c} \mid 1 \leq i_j \leq n\}$ is a basis of $V$.  

The semilinear map on $V$ sending $\nu e_{i_1} \otimes e_{i_2} \otimes \dots \otimes e_{i_c}$ to $\nu^q e_{i_c} \otimes e_{i_1} \otimes \dots \otimes e_{i_{c-1}}$ for $\nu \in \mathbb{F}_{q^c}$ induces an automorphism of $\mathrm{SL}_{n}(q^c)$ defined by $\varphi: (a_{ij}) \mapsto (a_{ij}^q) $ . As discussed in the proof of Proposition \ref{field},  $\langle \varphi\rangle$ fixes $e_i \otimes e_i\otimes \dots \otimes e_i$ for $1\leq i \leq c$ and has orbits of length $c$ on the remaining elements of $B$. Let $\theta$ be a set of $\langle \varphi\rangle$-orbit representatives  and let $\{\nu_1, \nu_2, \dots \nu_c\}$ be a basis of $\mathbb{F}_{q^c}$ over $\mathbb{F}_q$. We define a new basis of $V$ as follows:
\[
B_2 = \{\sum_{i=0}^{c-1}\nu_j^{q^i}\varphi^i(v)\mid v \in \theta, 1\leq j \leq c\}
\]
Notice that every element of $B_2$ is preserved under the action of $\varphi$.
An element $g \in \mathrm{SL}_{n}(q^c)$ preserves the $\mathbb{F}_q$ span of $B_2$ since
 \[
 (\sum_{i=0}^{c-1}\nu_j^{q^i}\varphi^i(v))g =  \sum_{i=0}^{c-1}\nu_j^{q^i}\varphi^i(vg) =\sum_{i=0}^{c-1} \varphi^i(\nu_j vg)
\] 
which is fixed by the action of $\varphi$.
We can therefore realise $V$ as a $\mathbb{F}_q\mathrm{SL}_n(q^c)$-module, and \cite[Proposition 5.4.6]{KL} shows that all absolutely irreducible $\mathbb{F}_q\mathrm{SL}_n(q^c)$-modules arise in this way, replacing $W$ with some arbitrary irreducible $\mathrm{SL}_n(q^c)$-module.  

In the notation of Theorem \ref{main}, the absolutely irreducible $\mathbb{F}_q\mathrm{SL}_n(q^c)$-modules  $V$ constructed in this way correspond to cases where $k<1$ (indeed, $k=1/c$), and we further assume in this section that $V$ cannot be realised over any proper subfield of $\F_q$. The remaining cases with $k<1$ will be dealt with in Section \ref{ext_section}.

%The following proposition demonstrates that the structure of $V$ is severely restricted.
%
%\begin{proposition}[{\cite[Proposition 5.4.6.]{KL}}]
%\label{KLfieldext}
%Let $H=\mathrm{SL}_n(q)$ with $q=p^e$, and suppose that $V$ is an absolutely irreducible $\mathbb{F}_{p^f}H$-module which is realised over no proper subfield of $\mathbb{F}_{p^f}$. Also let $k$ be an algebraically closed field of characteristic $p$. Then $f\mid e$ and there is an irreducible $kH$-module $M$ such that 
%\[
%V \otimes k \cong M \otimes M^{(f)}\otimes \dots \otimes M^{(e-f)}.
%\]
%In particular, $\dim V = (\dim M)^{e/f}$.
%\end{proposition}
%Note that this proposition provides a necessary but not sufficient condition.
The main theorem of this section is as follows.
\begin{theorem}
	\label{submain}
Let $V=V_{m^c}(q)$ be a $m^c$-dimensional vector space over $\mathbb{F}_q$, with $c>1$. Let $G \leq \Gamma \mathrm{L}(V)$ be almost quasisimple
	with $E(G)/Z(E(G)) \cong \mathrm{PSL}_{n}(q^c)$, such that
	$V=V(\lambda)$ is an absolutely irreducible $\mathbb{F}_qE(G)$-module and $V$ cannot be realised over a proper subfield of $\F_q$.
Then one of the following holds.
	\begin{enumerate}
		\item $G$ has a regular orbit on $V$;
		\item $\lambda$, $k=1/c$ and $n$ appear in Table \ref{rem2}.
	\end{enumerate}
\end{theorem}

We begin by reducing the proof of Theorem \ref{submain} to a finite list of modules that require further consideration.
\begin{proposition}
	\label{all_subnoro}
Suppose $G \leq \Gamma \mathrm{L}(V)$ is as in Theorem \ref{submain}, with $V=V_{m^c}(q)$, whose restriction to $E(G)$ is absolutely irreducible. Then one of the following holds.
\begin{enumerate}
\item $G$ has a regular orbit on $V$.
\item\label{subnoro} $(n,m,c) = (n,n,2)$ or $(2,2,3)$ and $G$ has no regular orbit on $V$.
\item The tuple $(n,m,c)$ appear in Table \ref{subfieldbad}.
\end{enumerate}
\end{proposition}
\begin{proof}
Notice that if $(n,m,c)$ appears in part (\ref{subnoro}), then $|V|<|G|$ and there cannot be a regular orbit of $G$ on $V$.
Suppose $n\geq 5$. If $G$ has no regular orbit on $V$, then
$q^{m^c} \leq 2|\mathrm{Aut}(\mathrm{PSL}_n(q^c))| q^{\floor{\frac{n-1}{n}m^c}}$. Since $m\geq n$, this gives a contradiction for $c\geq 4$, for $c=3$ if $m>n$, and for $c=2$ if $m>\binom{n}{2}$ or $m=\binom{n}{2}$ and $n\geq 7$. By \cite{MR1901354}, a full list of the remaining modules not satisfying the inequality (excluding $(n,m,c)=(n,n,2)$) is given in Table \ref{subfieldbad}. Similarly, if $n=4$ and $G$ has no regular orbit on $V$, then by Propositions \ref{alphas}, \ref{field} and \ref{graph}, 
\[
q^{m^c} \leq 2(|\pgl_4(q^c)|+ 2q^{15c/2)}) q^{\floor{\frac{3}{4}m^c}}+ 4\log(\log_2q^c+2)q^{8c+(m^c+m^{c/2})/2}+ 2q^{9c}q^{6m^c/7}.
\]
 This gives a contradiction for $c\geq 5$, for $c=3,4$ if $m>4$ and for $c=2$ if $m \geq 14$ and the remaining modules are given in Table \ref{subfieldbad}. %Checked 16/3/20
If instead $n=3$ and $G$ has no regular orbit on $V$ we have
  \[
 q^{m^c} \leq 2|\pgl_3(q^c)| (q^{\floor{\frac{2}{3}m^c}}+q^{\ceil{\frac{1}{3}m^c}}) 4\log(\log_2q^c+2)q^{9c/2+(m^c+ m^{c/2})/2}+ q^{5c}q^{3m^c/4}.
 \]
 This gives a contradiction for $c\geq 5$ for $m \geq 3$, for $c=3,4$ if $m>3$ and for $c=2$ if $m\geq 7$.
Finally let $n=2$ and suppose $q\geq 7$ and $q\neq 9$.
If $G$ has no regular orbit on $V$, then by Propositions \ref{alphas} and \ref{invols}, 
\[
q^{m^c} \leq 2|\mathrm{PGL}_2(q^c)| q^{\floor{m^c/2}} + 2(q^2+q)(q^{\floor{2m^c/3}}+ q^{\floor{m^c/3}})+4\log(\log_2q^c+2)q^{2c+(m^c+m^{c/2})/2}.
\] This gives a contradiction for $c\geq 5$ and $m\geq 2$, for $r=3,4$ if $m>2$ and for $r=2$ if $m\geq 4$. 
Lastly, if $q=9$, then examining \cite[p.4]{modatlas}, we see that the only absolutely irreducible $\mathrm{SL}_2(9)$ modules realised over $\F_3$ have $(m,c)=(2,2)$ or $(3,2)$, and so the result follows
%If $(n,m,r) = (2,3,2)$ or $(2,4,2)$, then $E(G)/Z(E(G)) \cong \Omega^-_4(q)$ or $\Omega^-_4(q^2)$ respectively, so we omit it from the table.
\end{proof}
Suppose $(n,m,c) = (2,2,3)$ and let $G$ and $V$ be as in Theorem \ref{submain}. In Proposition \ref{sub_ext_prop} we show that $b(G)=1$ on $V \otimes \F_{q^2}$, so by Proposition \ref{fieldext}, $b(G)=2$ on $V$.
Similarly, if $(n,m,c)=(n,n,2)$, we show that $G$ has a regular orbit on $V\otimes \F_{q^3}$ in  Proposition \ref{exttensor}, so for the action of $G$ on $V$, we have $2 \leq b(G) \leq3$ by Proposition \ref{sub_ext_prop}. We will use similar arguments in this section for other $G$ and $V$ where we cannot show there is a regular orbit.
\begin{table}[!htbp]
\begin{tabular}{@{}ccc@{}}
\toprule
$n$      & $m$ & $c$ \\ \midrule
$\geq 3$ & $n$ & 3   \\
\midrule
$4\leq n \leq 6$        & $\binom{n}{2}$  & 2   \\ %WRITTEN for n=5,6
\midrule
4        & 10  & 2   \\ %WRITTEN
       & 4   & 4   \\ %WRITTEN
\midrule
3        & 6   & 2   \\%WRITTEN
        & 3   & 4   \\ %WRITTEN
\midrule
2        & 3   & 2   \\
 & 2   & 4   \\
 \bottomrule
\end{tabular}
\caption{Representations of $SL_n(q^c) < \mathrm{GL}_{m^c}(q)$ left to consider.\label{subfieldbad}} 
\end{table}
%\begin{proposition}
%Let $G$ be an almost quasisimple group with $E(G)/Z(E(G)) \cong \mathrm{PSL}_n(p^e)$, $p$ a prime. Let $V$ be an irreducible $\mathbb{F}_{p^f}G$-module for $f >e$, whose restriction to $E(G)$ is absolutely irreducible. Then $G$ has a regular orbit on $V$.
%\end{proposition}
%\begin{proof}
%By Proposition \ref{KLfieldext}, $V$ must be realised over a proper subfield $\mathbb{F}_{p^k}$ of $\mathbb{F}_{p^f}$, with $k \mid e$. 
%%TODO finish
%\end{proof}

%\begin{proposition}[{\cite[Corollary 1]{MR2764924}}]
%\label{jordblocksize}
%Let $a,b \in \mathbb{N}$ be such that $1\leq a,b \leq p$ for some prime $p$, and let $q=p^e$. Let $V_a$, $V_b$ denote $a-$ and $b-$dimensional indecomposable $\mathbb{F}_qC_p$ modules.  
%\begin{enumerate}
%\item If $a+b\geq p$ , then 
%\[
%V_a \otimes V_b = (a+b-p)V_{p} \oplus \left(\bigoplus_{i=1}^{\min\{p-a,p-b\}} V_{2p-(a+b)-i} \right).
%\]
%
%\item If instead $a+b< p$, 
%\[
%V_a \otimes V_b =  \bigoplus_{i=1}^{\min\{a,b\}} V_{(m+n)-i}.
%\]
%\end{enumerate}
%\end{proposition}
\begin{proposition}
	\label{subcube}
Let $V=V_{n^3}(q)$ be an $n^3$-dimensional vector space over $\F_q$, where $n\geq 3$. Suppose that $G \leq \Gamma \mathrm{L}(V)$ is almost quasisimple with $E(G)/Z(E(G)) \cong \mathrm{PSL}_n(q^3)$ and that the restriction of $V$ to $E(G)$ is absolutely irreducible. If $n\geq 4$, then $G$ has a regular orbit on $V$ and if $n=3$, then $b(G)\leq 2$.
\end{proposition}
\begin{proof}
We have $V\otimes \overline{\mathbb{F}}_p \cong V(\lam1) \otimes V(\lam1)^{(q)} \otimes  V(\lam1)^{(q^2)}$, where $V(\lam1)$ is the natural module for $E(G)$.
Suppose $s \in  G$ is semisimple of projective prime order, and suppose the eigenvalues of $s$ on $V(\lam1)\otimes \overline{\mathbb{F}}_p$ are $t_1, t_2, \dots t_m$, ordered so that the multiplicities $a_i$ of the $t_i$ are weakly decreasing. By Propositions \ref{permsquare} and \ref{tensorcodim}, the codimension of $C_V(s)$ is at least $n(n^2 -\sum a_i^2)$, while $\log_q (|sF(G)^{\pgl_n(q)}|) < 3(n^2-\sum a_i^2)+m\log_q2$. Therefore, noting that $m\leq n$, we have
\begin{equation}
\label{n^3ss}
\mathrm{codim}_{\mathbb{F}_q} C_V(s) -\log_q (|sF(G)^{\pgl_n(q)}|) \geq (n-3)(n^2-\sum a_i^2)-n\log_q2
\end{equation}
Provided that $a_1<n-1$, the right hand side of \eqref{n^3ss} is at least $(n-3)(4n-8)$. If instead $a_1=n-1$, then $\mathrm{dim} C_V(s) \leq n^3-2n^2+2n$ and we will treat such classes separately.
Now suppose $uF(G) \in G/F(G)$ is of prime order with Jordan blocks $(J_p^{a_p}, \dots , J_2^{a_2}, J_1^{a_1})$ on the natural module for $G$. By \cite[Lemma 1.3.3.]{bigpaper} and Proposition \ref{tensorcodim}, the codimension of $C_V(u)$ is at least $n(n^2-\sum_{i<j} ia_ia_j -\sum i a_i^2)$, and $\log_q (|uF(G)^{\pgl_n(q)}|)<3(n^2-2\sum_{i<j} i a_ia_j -\sum i a_i^2)+m$, where $m$ is the number of non-zero $a_i$. Therefore, since $m \leq 1+ \sum_{i<j} i a_ia_j $, we have
 \begin{equation}
\label{n^3u}
\mathrm{codim}_{\mathbb{F}_q} C_V(u) -\log_q (|uF(G)^{\pgl_n(q)}|) \geq (n-3)(n^2-2\sum_{i<j} ia_ia_j -\sum a_i^2)-1.
\end{equation}
Provided that $u$ does not have Jordan canonical form $(J_2,J_1^{n-2})$ on $V(\lam 1)$, the right hand side of \eqref{n^3u} is at least $(n-3)(4n-8)-1$. If $u$ does have Jordan canonical form $(J_2,J_1^{n-2})$, then $\dim C_V(s) \leq n^3-3n^2+4n$. By \cite[Lemma 3.7]{MR2888238}, the number of conjugacy classes in $\mathrm{PGL}_n(q^3)$ is at most $q^{3(n-1)}+5q^{3(n-2)}$.
Therefore, if $G$ has no regular orbit on $V$,
\begin{align*}
q^{n^3} \leq &2(q^{3(n-1)}+5q^{3(n-2)})(q^{n^3-(n-3)(4n-8)-n\log_q2}+q^{n^3-(n-3)(4n-8)+1})+ 2q^{3(2n-1)}q^{n^3-2n^2+2n}\\
&+q^{3(2n-1)}q^{n^3-3n^2+4n}+2\log(\log_2 q^3+2)q^{3n^2/2+(n^3+n^{3/2})/2}
\end{align*}
This gives a contradiction for $n\geq 5$ and $q\geq 2$. Replacing $(q^{3(n-1)}+5q^{3(n-2)})$ with a more accurate count of the number of conjugacy classes of prime order elements in $\mathrm{PGL}_4(q^3)$ gives the result for $n=4$.
%
%We will proceed with the proof by precisely computing the eigenspace dimensions and class sizes of some prime order elements in $\mathrm{Aut}(\mathrm{PSL}_n(q^3))$.
%\begin{table}[]
%\begin{tabular}{@{}ccc@{}c@{}}
%\toprule
%$x$ & $|x^{\mathrm{PGL}_n(q^3)}|$ & $\dim C_{V}(x)\leq$ &  \\ \midrule
%$(J_2^2,J_1^{n-4})$ & $4q^{12n-12}$ & $n^3-6n^2+24n-40+8\ep_2$ &  \\
%$(J_2,J_1^{n-2})$ & $q^{6n-3}$ & $n^3-3n^2+6n-5+\ep_2$ &  \\
%$(J_{3}, J_1^{n-3})$ & $8q^{12n-9}$ & $n^3-6n^2+18n-20+2\ep_3$ &  \\
%$(J_{4}, J_1^{n-4})$ & $4q^{18n-27}$ & $n^3-9n^2+36n-52+\ep_{5}$ &  \\
%$(J_{3},J_2, J_1^{n-5})$ & $8q^{18n-33}$ & $n^3-9n^2+42n-73+5\ep_3$ &  \\ \bottomrule
%\end{tabular}
%\caption{Some upper bounds on $\mathrm{PGL}_n(q^3)$-conjugacy class size and fixed point space dimension for unipotent elements of prime order.}
%\end{table}
%%TODO finish
%{\color{red} -----------$n=3$}
Finally, if $n=3$, then $b(G) \leq 2$ by Propositions \ref{fieldext} and \ref{sub_ext_prop}.
\end{proof}
\begin{proposition}
	\label{sub_symm_square}
Let $m=\binom{n+1}{2}$ with $n \in [2,4]$, and define $V=V_{m^2}(q)$. Suppose that $G \leq \Gamma \mathrm{L}(V)$ is an almost quasisimple group with $E(G)/Z(E(G)) \cong \mathrm{PSL}_n(q^2)$ such that the restriction of $V$ to $E(G)$ is absolutely irreducible. If $n\in [3,4]$, then $G$ has a regular orbit on $V$. If instead $n=2$, then $b(G) \leq 2$.
\end{proposition}
\begin{proof}
We have $V\otimes \overline{\mathbb{F}}_p \cong V(2\lam1) \otimes V(2\lam1)^{(q)}$, and so $p\geq 3$. 
The Weyl orbit and weight string tables for $V(2\lam1)$ are given below.
\begin{table}[H]
%\begin{minipage}{0.5\textwidth}
\begin{tabular}{cccc}
\toprule 
$i$ & $\mu$ & $|W.\mu |$ & Mult \\ 
\midrule 
1 & $2\lam1$&$l+1$  & 1 \\ 
2& $\lam2$ &$\binom{l+1}{2}$& 1\\
\bottomrule
\end{tabular}
%\caption{Weyl orbit table.}
%\end{minipage}
\quad
%\begin{minipage}{0.5\textwidth}
\begin{tabular}{ccccc}
\toprule 
 &  & \multicolumn{2}{c}{$c(s)$} & $c(u_\Psi)$ \\ 
\midrule 
String & Mult & $r=2$ & $r\geq 3$ & $p\geq 3$ \\ 
\midrule 
$\m1$ & $l-1$ & 0 & 0 & 0 \\ 

$\m1\m2\m1$  & 1 & 1 & 2 & 2 \\ 

$\m2$ & $\binom{l-1}{2}$ & 0 & 0 & 0 \\ 

$\m2\m2$  & $l-1$ & $l-1$ & $l-1$ & $l-1$ \\ 
\midrule 
\textbf{Total} &   & $l$ & $l+1$ & $l+1$ \\ 
\bottomrule 
\end{tabular} 
\caption{Weyl orbit and weight string tables for $V(2\lam1)$.}
%\end{minipage}
\end{table}

By Propositions \ref{tensorcodim} and \ref{invols}, if $G$ has no regular orbit on $V$, then 
\[
q^{m^2} \leq 2 |\mathrm{PGL}_n(q^2)|q^{m(m-n)} + 4(q^{n^2+n-2}+q^{n^2+n-3})q^{m(m-n+1)}+2\log(\log_2q^2+2)q^{n^2+(m^2+m)/2}.
\]
This gives a contradiction for $n=4$ and $q\geq 3$, as well as $n=3$, $q\geq 4$. If $E(G)/Z(E(G)) \cong \mathrm{PSL}_3(9)$, then substituting the precise number of involutions and prime order elements (computed using GAP) into the equation, we also obtain a contradiction. 
Finally, if $n=2$, then $b(G)\leq 2$ by Propositions \ref{fieldext} and \ref{n^3}.
%Finally, suppose $n=2$. We summarise some information about projective prime order elements in $G$ in Table \ref{2L1_n=2_tens}.
%
%\begin{table}[h]
%	\begin{tabular}{ccc}
%		\toprule
%		Type & Number in $\mathrm{PGL}_2(q^2)$ & $\dim C_V(x)\leq$\\
%		\midrule
%		Involutions & $q^4$& 5\\
%		Order 3 & $2(q^{14/3}+q^{11/3})$ & 3\\
%		Unipotent & $q^4-1$ & 3\\
%		Semisimple $r\mid q-1$ & $2q^3(q^2+1)$ & $3$\\
%		Semisimple $r\mid q+1$ & $\frac{1}{2}q^2(q^2-1)(q^2+1)$ & $2$\\
%		\bottomrule
%		\end{tabular}
%	\caption{Information about elements of projective prime order $r$ in $G$ with $E(G)/Z(E(G))=\mathrm{PSL}_2(q^2)$. \label{2L1_n=2_tens}}
%	\end{table}
% We also observe that an involution $x\in G$ is conjugate to $-x\in G$. Therefore, if $G$ has no regular orbit on $V$ then by Proposition \ref{eigsp2} and the proof of Proposition \ref{field}, 
%\begin{align*}
%q^9 \leq &\frac(q^{14/3}+q^{11/3})(3q^3)+\frac{2q}{4}q^2(q^2+1)(3q^3)+\frac{5(q^2+1)}{8}q^2(q^2-1)q^2\\
%&+(q^4-1)q^3+\frac{1}{2}q^4(q^5+q^4)+\frac{2}{q+1}\log(\log_2 q+2)q^{4+9/2}
%\end{align*}
%This gives a contradiction for $q\geq 11$. When $q\leq 9$ then we use GAP to construct the respective modules and find that there is a regular orbit in the remaining cases.
\end{proof}
\begin{proposition}
	\label{sub_n^4}
Let $V=V_{n^4}(q)$, with $n\in [2,4]$, and let $G \leq \Gamma \mathrm{L}(V)$ be an almost quasisimple group with $E(G)/Z(E(G)) \cong \mathrm{PSL}_n(q^4)$. Further suppose that $E(G)$ acts absolutely irreducibly on $V$. If $n\in [3,4]$, then $G$ has a regular orbit on $V$, and if $n=2$, then $b(G)\leq 2$.
\end{proposition}
\begin{proof}
We have  $V\otimes \overline{\mathbb{F}}_p \cong V(\lam1) \otimes V(\lam1)^{(q)} \otimes  V(\lam1)^{(q^2)}\otimes  V(\lam1)^{(q^3)}$, where $V(\lam1)$ is the natural module for $E(G)$.
Now, any projective prime order element has an eigenspace of dimension at most $n-1$ on the natural module. Therefore, by Propositions \ref{tensorcodim} and \ref{permsquare}, for projective prime order $x\in G$, $\dim{C_V(g)} \leq n^2((n-1)^2+1^2)$ and so if $G$ has no regular orbit on $V$, then 
\[
q^{n^4} \leq 2|\mathrm{PGL}_n(q^4)|q^{n^2((n-1)^2+1^2)} + 2\log(\log_2q^4+2)q^{2n^2+(n^4+n^2)/2}.
\]
This gives a contradiction for $q\geq 2$ and $n=3,4$, so $G$ has a regular orbit on $V$.

If $n=2$, then $b(G)\leq 2$ by Propositions \ref{fieldext} and \ref{sub_ext_prop}.
\end{proof}

\begin{proposition}
	\label{sub_ext_square}
	Suppose $V=V_{m^2}(q)$, where $m=\binom{n}{2}$ and $n\in [4,6]$.
Let $G$ be an almost quasisimple group with $E(G)/Z(E(G)) \cong \mathrm{PSL}_n(q^2)$ such that $E(G)$ acts absolutely irreducibly on $V$. If $n\in [5,6]$, then $G$ has a regular orbit on $V$, and if $n=4$, then $b(G)\leq 2$.
\end{proposition} 
\begin{proof}
We have $V\otimes \overline{\mathbb{F}}_p \cong V(\lam2) \otimes V(\lam2)^{(q)}$. From the $\Psi$-net tables for $V(\lam2)$ given in Section \ref{tensorl2}, and also Proposition \ref{tensorcodim}, we deduce that if $G$ has no regular orbit on $V$ then
\begin{align*}
q^{\binom{n}{2}^2}\leq & 2(q^{2(2n-1)}+q^{2(2n-1)})q^{\binom{n}{2}(\binom{n}{2}-n+2)}+ \left[ 2(2q^{2(n^2/2+2)})+4 \left( \frac{q^{2(n^2/2+2)}-1}{q^4-1}\right)\right]q^{\binom{n}{2}(\binom{n}{2}-2n+6)}\\
& +2|\mathrm{PGL}_n(q^2)| q^{\binom{n}{2}(\binom{n}{2}-2n+4)}+ 2\log(\log_2q^2+2)q^{n^2+\frac{1}{2}( \binom{n}{2}^2+\binom{n}{2})}.
\end{align*}
This inequality is false for $n=5,6$ and $q\geq 2$.
Finally, if $n=4$, then $b(G) \leq 2$ by Propositions \ref{fieldext} and \ref{L2_tens_l=3}.
\end{proof}
We have now completed the proof of Theorem \ref{submain}.

\section{Absolutely irreducible representations over extension fields}
\label{ext_section}
In this section we consider absolutely irreducible $\mathrm{SL}_n(q)$-modules $V$ defined over a field $\mathbb{F}$ that can be realised over a proper subfield of $\mathbb{F}$. We have already analysed such modules $V$ over $\F_q$ in Sections \ref{prest} and \ref{tensor}, so we also assume that $k\neq 1$ (with $k$ as in Theorem \ref{main}) in this section.
The main result of this section is as follows.

\begin{theorem}
\label{extmain}
Let $V=V_d(q^k)$ be a $d$-dimensional vector space over $\mathbb{F}_{q^k}$ with $k \neq 1$, and let $G \leq \Gamma \mathrm{L}(V)$ be almost quasisimple with $E(G)/Z(E(G)) \cong \mathrm{PSL}_n(q)$ and $n\geq 2$. Further suppose that the restriction of $V=V(\lambda)$ to $E(G)$ is an absolutely irreducible $\F_{q^k}E(G)$-module of highest weight $\lambda$, and that  $V$ can be realised over a proper subfield of $\F_{q^{k}}$. Then either $G$ has a regular orbit on $V$, or $n$, $b(G)$ and $\lambda$ (up to quasiequivalence) appear in Table \ref{rem2}.
\end{theorem}

Maintaining the notation in Theorem \ref{extmain}, let $\F_{q_0}$ be field of definition of $V$ i.e., the smallest subfield of $\F_{q^k}$ that $V$ is realised over. By our previous work, $q_0=q^{1/c}$ for some integer $c\geq 1$, and so we set $q=q_0^c$. Let $V_0$ be the realisation of $V$ over $\F_{q_0}$. We can then write $V=V_0 \otimes \F_{q^k}$, where $q^k=q_0^t$ i.e., $t=ck$. 

 We have already investigated these modules $V_0$ in Sections \ref{prest}, \ref{tensor} and \ref{subfields}. By Proposition \ref{extfield}, if one of the inequalities given in Proposition \ref{tools}\ref{eigsp1}--\ref{crude} fails for $G$ acting on $V_0$, then $\mathbb{F}_{q_0^t}^{\times}\circ G$ has a regular orbit on $V$. Therefore, considering the results contained in Sections \ref{prest}, \ref{tensor} and \ref{subfields}, we see that it remains to consider field extensions of the modules $V_0=V(\lambda)$ listed in Table \ref{extrem}, where there is either no regular orbit of $G$ on $V_0$, or a different method of proof was used.

%Such modules $V$ as in Theorem \ref{extmain} are of the form $V_0\otimes \mathbb{F}_{q_0^t}$ for some absolutely irreducible $\mathbb{F}_{q_0}\mathrm{SL}_n(q_0^c)$-module $V_0$ with field of definition $\mathbb{F}_{q_0}$, and $t>1$. We have already investigated these modules $V_0$ in Sections \ref{prest}, \ref{tensor} and \ref{subfields}. By Proposition \ref{extfield}, if one of the inequalities given in Proposition \ref{tools}\ref{eigsp1})--\ref{tools}) fails for $G$ acting on $V_0$, then $\mathbb{F}_{q_0^t}\circ G$ has a regular orbit on $V$. Therefore, considering the results contained in Sections \ref{prest}, \ref{tensor} and \ref{subfields}, we see that it remains to consider field extensions of the modules $V_0=V(\lambda)$ listed in Table \ref{extrem}, where there is either no regular orbit of $G$ on $V_0$, or a different method of proof was used.

\begin{table}
	
	\begin{tabular}{@{}lccc@{}}
		\toprule
		$\lambda$ & Dimension & $c$ & $n$ \\ \midrule
		$\lambda_1$ & $n$ & 1 & $[2,\infty)$ \\
		$\lambda_2$ & $\binom{n}{2}$ & 1 & $[3,\infty)$ \\
		$2\lambda_1$ & $\binom{n+1}{2}$ & 1 & $[2,\infty)$ \\
		$3\lam1$ & 4 & 1 & $2$ \\
		& 6 & 1 & $3$ \\
		$\lambda_3$ & 20 & 1 & $6$ \\
		& 84 & 1 & $9$ \\
		$\lam4$ & 70 & 1 & $8$ \\
		$2\lam2$ & $20-\ep_3$ & 1 & 4 \\
		$4\lam1$ & 5 & 1 & 2 \\
		$\lam1+\lam {n-1}$ & $n^2-1-\ep_n$ & 1 & $[4,\infty)$ \\
		$\lambda_1+ p^a\lambda_{n-1}$ & $n^2$ & 1 & $[3,\infty)$ \\ 
		$(p^a +1 )\lambda_1$ & $n^2$ & $2$ & $[2,\infty)$ \\
		$2(p^{e/2}+1)\lam1$ & 9 & $2$ & $2$ \\
		$(p^{e/2}+1)\lam2$ & 36 & $2$ & $4$ \\
		$(p^{e/3}+p^{2e/3}+1)\lam1$ & $n^3$ & $3$ & $[2,3]$ \\
		$(p^{e/4}+p^{2e/4}+p^{3e/4}+1)\lam1$ & 16 & $4$ & 2 \\
		 \bottomrule
	\end{tabular}
%\begin{tabular}{ccc}
%\toprule 
%$\lambda$ & Dimension  & $n$\\ 
%\midrule
%$\lambda_1$ &$n$ & $[2,\infty)$\\
%$\lambda_2$ &$\binom{n}{2}$ &$[3,\infty)$\\
%$2\lambda_1$ &$\binom{n+1}{2}$  & $[2,\infty)$\\
%$3\lam1$ & 4 & $2$\\
%				& 6 & $3$\\
%$\lambda_3$  &20& $6$ \\ 
%						 &84& $9$ \\ 
%$\lam4$ & 70 & $8$\\ %G+L ref for q proof
%$2\lam2$ & $20-\ep_3$& 4\\ %G+L ref for q proof
%$4\lam1$ & 5& 2\\ %Computational proof for $q \leq 180
% $\lam1+\lam {n-1}$ &$n^2-1-\ep_n$& $[4,\infty)$\\
%$(p^a +1 )\lambda_1$ &$n^2$ &$[2,\infty)$  \\ 
%$p^a\lambda_1+ \lambda_{n-1}$ &$n^2$ & $[3,\infty)$\\ 
%\bottomrule
%\end{tabular}
\caption{Remaining absolutely irreducible $\mathrm{SL}_n(q_0^c)$ modules to consider, with field of definition $\F_{q_0}$.\label{extrem}}
\end{table}

We now consider each of the modules in Table \ref{extrem} individually. 
\begin{proposition}
\label{L1_base_size}
Let $G$ be almost quasisimple with $E(G) \cong \mathrm{SL}_n(q)$. Denote by $V_0$ the natural module for $G$ over $\mathbb{F}_q$ and let $V = V_0 \otimes \mathbb{F}_{q^k}$ with $k\geq 1$.
If $G\leq \Gamma\mathrm{L}(V)$ contains no field automorphisms, then $b(G)=\lceil n/k \rceil +c$, where $c=1$ if $i = (k,n)>1$ and $G$ contains scalars in $\mathbb{F}_{q^i}^{\times}\setminus \F_q^\times$, and $c=0$ otherwise. If $G$ contains field automorphisms, then $b(G)\leq \lceil n/k \rceil +1$, with equality if $G$ contains scalars in $\mathbb{F}_{q^i}^{\times}\setminus \F_q^\times$ or $\log_q|G|>kn \lceil n/k \rceil$.

%
%The base size of $G$ on $V = V_0 \otimes \mathbb{F}_{q^k}$ is at most $\lceil n/k \rceil +c$, and define $c$ such that $c=1$ if $G$ contains field automorphisms, or $i = (k,n)>1$ and $G$ contains scalars in $\mathbb{F}_{q^i}^{\times}\setminus \F_q^\times$, and $c=0$ otherwise. Moreover:
%\begin{enumerate}
%	\item If $G\leq \gl_n(q)$, then $b(G)=
%\end{enumerate}
% with equality if either: (1) $G$ contains scalars lying in $\mathbb{F}_{q^i}^{\times}\setminus \F_q^\times$, (2) $|G|> q^{n^2}$ and $k\mid n$, or (3) $G$ contains no field automorphisms.
\end{proposition}
\begin{proof}
First note that $b(G) \geq \log |G|/\log |V|$, so $b(G) \geq \lceil n/k \rceil $ if $k\geq 2$, since $|G|\geq |\mathrm{SL}_n(q)| > q^{n^2-2}$. Moreover, if $|G|> kn \lceil n/k \rceil$, then $b(G) \geq \lceil n/k \rceil +1$. 
 Let $e_1, \dots e_n$ denote the standard basis of $V_0$ and let $\omega_k$ be a generator of $\mathbb{F}_{q^k}^\times$. If $k=1$, then $\{e_1, \dots e_n \}$ is a base for $\gl_n(q)$, and adding 
$\omega_1e_1$ to the standard basis gives a base for $\Gamma \mathrm{L}_n(q)$. So from now on, suppose $k\geq 2$. Define
\[
\mathcal{B} = \left\{\sum_{i=1}^{\min \{k,n-jk\}} \omega_k^{i-1}e_{i+jk} \, \Big\vert\,  0 \leq j \leq \lceil n/k \rceil-1\right\}.
\]
Note that $\{\omega_k^i  \mid  0 \leq i \leq k-1\}$ is linearly independent over $\mathbb{F}_q$. Suppose $g \in G_\mathcal{B}$ has projective prime order $r$. The elements of $\mathcal{B}$ lie in the $\F_{q^k}$-span of vectors in the fixed point space  $C_V(g)$ of $g$. If $g$ is unipotent, or is semisimple with $r\mid q-1$, then this holds if and only if $g$ fixes $e_1, \dots e_n$, i.e., $g$ is the identity, contradicting our assumption that $g$ is non-trivial. So either $g$ is semisimple and $r$ divides $q^{k}-1$ but not $q-1$, or $g$ is a field automorphism.

%Now, $g \in G\setminus 1$ fixes $\mathcal{B}$ only if either: $g$ fixes the basis $e_1, \dots e_n$ (that is, $g$ is the identity), or $g$ has projective prime order $r$ dividing $q^{k}-1$ and not $q-1$. 
If $g$ is semisimple, then $g$ has the form $ah$, for some $a \in \F_{q^{k}}^\times \setminus \F_q$ and $h \in \mathrm{SL}_n(q)$.
Let $i$ be the least natural number such that $r \mid q^{i}-1$, and note that $i\mid k$. If $g$ fixes $\mathcal{B}$, then each vector contained in $\mathcal{B}$, excluding that where $j = \lceil n/k \rceil -1$, can be written uniquely as a $\F_{q^k}$ linear combination of $k/i$ linearly independent vectors in $V_0 \otimes \F_{q^i}$, while the vector with $j = \lceil n/k \rceil -1$ can be written as the $\F_{q^k}$ linear combination of $\min\{b,k/i\}$ such vectors, where $b  = n \mod k$. Therefore, the fixed point space $C_V(g)$ of $g$ has dimension at least $\frac{k}{i} \floor{\frac{n}{k}}+ \min\{b,k/i\}$.
Since $h$ is defined over $\F_q$, we must have $b=0$, so $i\mid n$, the fixed point space of $g$ has dimension $n/i$, and $G$ must contain scalars $aI$ of order $r$. Moreover, a representative of $gF(G) \in G/F(G)$ must have eigenvalues $\{\alpha^{q^j} \mid 0 \leq j \leq k_1-1\}$ for some $\alpha\in  \mathbb{F}_{q^{i}}\setminus \F_q$ of order $r$, each of multiplicity $n/i$. 
We see that $e_i\notin C_V(g)$ for all $1\leq i \leq n$, since otherwise $e_ih \notin V_0$. 
Now suppose that $g$ is a field automorphism, and write $q=p^e$. Then we can write $g = h\phi$, where $\phi$ is a standard field automorphism sending the coefficients of an element of $\mathcal{B}$ (with respect to the basis $e_1, e_2, \dots e_n$) to their $p^{a}$th powers, and $h\in \mathrm{GL}_n(q)$ is the unique matrix acting on each element of $\mathcal{B}$ by:
\[
\sum_{i=1}^{\min \{k,n-jk\}} \omega_k^{i-1}e_{i+jk}  \mapsto \sum_{i=1}^{\min \{k,n-jk\}} \omega_k^{p^{e-a}(i-1)}e_{i+jk}.
\]
Now, there exists $e_i$ not fixed by $h\phi$, since otherwise $h$ is the identity, and $g=\phi$ does not fix $\mathcal{B}$ pointwise. So $e_i$ is not fixed by any element in $G_\mathcal{B}$, and therefore, $B\cup \{e_i\}$ forms a base for $G$ on $V$.

%Choose a non-zero vector $v \in V \setminus \langle C_V(g) \rangle$, then $\mathcal{B} \cup \{v\}$ is a base for $G$, since $\mathcal{B} \subseteq \langle C_V(g) \rangle$. 
Therefore, the base size of $G$ on $V$ is less than or equal to $\lceil n /k \rceil+c$, where $c=1$ if either $(k,n)>1$ and $G$ contains scalars outside of $\F_q$, or $G$ contains field automorphisms, and $c=0$ otherwise. 
%
%
%
%For semisimple $g \in G$ of projective prime order $r$, let $i$ be the least natural number such that $r \mid q^{i}-1$. Suppose $g \in G_{\mathcal{B}} \setminus 1$. All eigenvectors for $g$ lie in $V_0\otimes \mathbb{F}_{q^i}$ and cannot be scaled to give a vector with all entries over a subfield of $\mathbb{F}_{q^i}$. Therefore, if $i<k$, then $g$ does not fix at least one element of $\mathcal{B}$, since $\mathcal{B}$ contains at least one vector that is not a scalar multiple of an element of $V_0\otimes \mathbb{F}_{q^i}$. Moreover, if $i>k$, then no element of $\mathcal{B}$ can be an eigenvector for $g$. Therefore, set $i=k$. Since $|\mathcal{B}| = \lceil n/k \rceil$ and $g \in G$, which is defined over $\mathbb{F}_q$, we must have $k \mid n$, $\omega_kI \in G$, and a representative of $gF(G) \in G/F(G)$ must have eigenvalues $\{\alpha^{q^j} \mid 0 \leq j \leq i-1\}$ for some $\alpha\in  \mathbb{F}_{q^i}$ of order $r$, each of multiplicity $n/k$. 
%%If $i=k=1$ then $g$ is a scalar matrix and so must be the identity, so $\mathcal{B}$ is a base for $G$. 
%Choose $v \in V \setminus \langle \mathcal{B} \rangle$, then $\mathcal{B} \cup \{v\}$ is a base for $G$, since $\mathcal{B}$ is a basis for the fixed point space of $g$. Therefore, if $1\neq k\mid n$ and $\omega_kI \in G$, we have shown that $b(G) \leq \lceil n/k \rceil+1$. 
It remains to show that if $(k,n)>1$ and $G$ contains scalars outside of $\F_q$, then $b(G) = \lceil n/k \rceil+1$, i.e., there is no base of size $\lceil n/k \rceil$. If there exists a base $\hat{\mathcal{B}}$ of size $\lceil n/k \rceil$, then we can write each element of the base as a linear combination of a set of vectors $B \subseteq V$ (over $\mathbb{F}_q$) with coefficients $\{\omega_k^i \mid 0 \leq i \leq k-1\}$, since this is a basis of $\mathbb{F}_{q^k}$ over $\mathbb{F}_q$. By the proof of Proposition \ref{fieldext}, $B$ forms a base for $G$ acting on $V$. But a base for $G$ acting on $V$ must be a spanning set, and since $|B| = n$, it follows that $B$ is a basis of $V$ and therefore $\hat{\mathcal{B}}$ is conjugate to $\mathcal{B}$ and cannot form a base of $G$.
\end{proof}
\begin{proposition}
\label{exttensor}
Let $G$ be almost quasisimple with $E(G)/Z(E(G)) \cong \mathrm{PSL}_n(q)$, and let $V=V((p^a+1)\lam1)$ or $V(p^a \lam1+\lam l)$ over $\mathbb{F}_{q^k}$ for $k>1$. Then $G$ has a regular orbit on $V$. 
\end{proposition}
\begin{proof}
Let $V=V((p^a+1)\lam1)$ over $\mathbb{F}_{q^k}$.
We will complete the proof by considering the action of $G$ on $\overline{V} = V\otimes \overline{\mathbb{F}}_q$. Note that here $k$ is an integer, except when $a=e/2$. In this case, since $V$ can be realised over $\mathbb{F}_{q^{1/2}}$, we instead must also consider $k=b/2$ for $b\geq 3$. Let $\overline{W}$ be the natural module for $G$ over $\overline{\mathbb{F}}_q$, and suppose $s\in G$ is semisimple of prime order $r$. Denote the eigenvalues of $s$ on $\overline{W}$ by $t_1, t_2, \dots , t_m$, ordered so that the multiplicities $a_i$ of the $t_i$ are weakly decreasing. Suppose that $\{v_i \mid 1\leq i \leq n\}$ is a basis of eigenvectors of $s$ on $\overline{W}$ so that $\{v_i \otimes v_j \mid 1\leq i,j \leq n\}$ forms a basis of $\overline{V}$ consisting of eigenvectors for $s$. Fix $t \in \overline{\mathbb{F}}_q$. The action of $s$ on $\overline{V}$ has eigenvalues of the form $t_it_j^{p^a}$, so for fixed $i$, there is at most one choice of $j$ so that $t=t_it_j^{p^a}$. Therefore, by Proposition \ref{permsquare}, the codimension of $C_{\overline{V}}(s)$ is at least $n^2 - \sum a_j^2$. Now, by \cite[Table B.3]{bg}, $|sF(G)^{\mathrm{PGL}_n(q)}| < 2^mq^{n^2-\sum a_j^2}$, and this is at most $q^{n^2+m/2-\sum a_j^2}$, since $q \geq 4$. Therefore, 
{\small
\begin{align}
k \, \mathrm{codim}C_V(s)-\log_q(|sF(G)&^{\mathrm{PGL}_n(q)}|)\nonumber \geq\\ 
 & k(n^2-\sum a_j^2)-(n^2+m/2-\sum a_j^2)= (k-1)(n^2-\sum a_j^2) -m/2 \label{pa+1sseqn}
\end{align}
}
% last line since $n^2-\sum a_j^2\leq 2n-2
If $k\geq 2$, then the right hand side of \eqref{pa+1sseqn} is at least $ (k-1)(2n-2)-n/2$, while if $k=3/2$, and $s$ does not have eigenvalues of multiplicity $(n-1,1)$ then the right hand side of \eqref{pa+1sseqn} is at least $2n-5$.
Now suppose that $u \in G$ is unipotent of prime order $p$, and with associated canonical Jordan form $(J_p^{a_p}, \dots J_2^{a_2},J_1^{a_1})$ on $\overline{W}$. By \cite[Lemma 1.3.3]{bigpaper}, the codimension of $C_{V}(u)$ is $n^2-\sum_{i<j} i a_i a_j-\sum ia_i^2$, while by \cite[Table B.3]{bg}, $|uF(G)^{\mathrm{PGL}_n(q)}| < 2^m q^{n^2-c}$, where $c= 2\sum_{i<j} ia_ia_j - \sum ia_i^2$ and $m$ is the number of distinct Jordan block sizes. So 
\begin{align}
 k\, \mathrm{codim}C_V(u)-\log_q(|uF(G)^{\mathrm{PGL}_n(q)}|) &\geq (k-2)(n^2-\sum ia_i^2-\sum_{i<j} ia_ia_j)+n^2-\frac{m}{2}-\sum ia_i^2 \label{pa+1ueqn}
 \end{align}
 %last line since n^2-\sum ia_i^2 \geq 4n-6
 If $k\geq 2$, then the right hand side of  \eqref{pa+1ueqn} is at least  $n^2-n/2-(n-2)^2-2 = \frac{7}{2}n-6$. If $k=3/2$, then $m \leq \sum_{i<j}ia_ja_j +1$ and if $u$ is not a root element, then the right hand side of \eqref{pa+1ueqn} is at least $2n-4$.
 By \cite[Lemma 3.7]{MR2888238}, the number of conjugacy classes in $\mathrm{PGL}_n(q)$ is at most $q^{n-1}+5q^{n-2}$. Therefore, if $k\geq 2$ and $G$ has no regular orbit on $V$,
 \[
 q^{k n^2} \leq 2(q^{n-1}+5q^{n-2})(q^{k n^2-(7n/2-6)}+q^{k n^2-((k-1)(2n-2)-n/2)})+ 2\log(\log_2q+2)q^{kn^2/2+n^2/2}
 \]
 This gives a contradiction for $n\geq 3$ and $q\geq 4$, except when $(n,k,q) = (4,2,4)$, where $E(G)/Z(E(G)) \cong Alt_8$, or $(3,2,q_1)$, where $q_1 \leq 9$. For the remaining cases where $n=3$, we compute that no element of projective prime order has an eigenspace of dimension greater than 5, and substituting this in gives a contradiction for the remaining $q$. When $n=2$, then each element of prime order has eigenspaces of dimension at most 2 on $V$, so if $G$ has no regular orbit on $V$ and $k\geq 2$,
 \[
 q^{4k} \leq 2|\pgl_2(q)|q^{k(4-2)} + 2\log(\log_2 q+2)q^{2k+2}
 \]
 This gives a contradiction for all $q\geq 4$, completing the proof. 
 If now $k=3/2$ and $G$ has no regular orbit on $V$, then
  \begin{align*}
 q^{k n^2} \leq & 2(q^{n-1}+5q^{n-2})(q^{k n^2-(2n-4)}+q^{k n^2-(2n-5)})+2q^{2n-1}q^{\tfrac{3}{2}(n^2-2n-2)}+ q^{2n-1}q^{\tfrac{3}{2}(n^2-3n+4)}\\
 &+ 2\log(\log_2q+2)q^{kn^2/2+n^2/2}
 \end{align*}
This gives a contradiction for all $n,q$ except $(n,q)=(5,4)$ and $n\leq 4$ for all $q$. For these cases, we perform a more detailed analysis of the classes of elements of prime order and again apply Proposition \ref{tools}\ref{crude} to give the result.
 The proof for $V = V(p^a \lam1+\lam l)$ with $k\geq 2$ is similar.
\end{proof}
	\begin{proposition}
		\label{adj_ext}
	Let $G \leq \Gamma\mathrm{L}(V)$ be almost quasisimple with $E(G)/Z(E(G)) \cong \mathrm{PSL}_n(q)$, and let $V=V(\lam1+\lam l)$ be the adjoint module for $E(G)$ over $\mathbb{F}_{q^k}$ for $k>1$. Then either $G$ has a regular orbit on $V$, or $(n,k)=(3,2)$, $G$ contains graph automorphisms and $b(G)\leq 2$.
\end{proposition}
\begin{proof}
	Recall that $\dim(V)=n^2-1-\ep_n$, where $\ep_n=1$ if $(n,q)=1$ and $\ep_n=0$ otherwise.
	We can consider $V$ as a composition factor of 
	$\hat{V}=V(\lambda_1) \otimes V(\lam l)$. Now $\hat{V}$ has an $(n^2-1)$-dimensional $\F_{q^k}\mathrm{GL}_n(q)$-submodule $V_0$. When $\ep_n=0$, then $V_0$ is irreducible, and we set $V=V_0$ and have  $\dim C_V(g) \leq  \dim C_{\hat{V}}(g)$ for all $g \in G$. We will aim to show the same when $\ep_n=1$. 
	So assume $\ep_n=1$, then $V_0$ has a one-dimensional $\F_{q^k}\mathrm{GL}_n(q)$-submodule $V_1= \langle w \rangle$, which is fixed by every element of $\gl_n(q)$, and $V=V_0/V_1$. Every element $g \in \gl(V_0)$ preserving $V_1$ can be considered as an element of $\mathrm{GL}(V)$ by setting $(v+V_1)g = vg+V_1$. Conversely, every element of $\mathrm{GL}(V)$ arises this way. Suppose $g \in \gl(V_0)$ is unipotent of prime order $p$ and let $v+V_1 \in C_V(g)$. Clearly if $v \in C_{V_0}(g)$ then $v+V_1 \in C_V(g)$, so $\dim C_{V_0}(g)-1 \leq C_V(g)$, since $g$ fixes $V_1$ pointwise. Now suppose $v+V_1 \in C_V(g)$, but $v \not\in C_{V_0}(g)$. Then $vg= v+tw$ for some $t \in \F_{q^k}^\times$. So $v(g-I) = tw$, and $t^{-1}v(g-I)=w$. Since $g-I$ is linear, it follows that every coset representative of an element of $C_V(g)$ must lie in $\langle C_{V_0}(g), v\rangle$. Therefore, recalling that $g$ fixes $V_1$ pointwise, we have  $\dim C_V(g) \leq \dim C_{V_0}(g)\leq \dim C_{\hat{V}}(g)$ as required.
	We now consider $g \in G$ semisimple of prime order. Recall that  we can think of the action of $G \cap \mathrm{GL}_n(q)$ on $V$ as conjugation on the  vector space  of $n \times n$ trace 0 matrices over $\F_{q^k}$, modulo the subspace $Z$ generated by the identity matrix. Suppose $g \in G$ is semisimple of  prime order. If $h+Z \in C_V(g)$, then as before $g^{-1}hg = h+jI$, for some $j \in \F_{q^k}$. Therefore, $g^{-p}hg^p = h$, so $h \in C_{V_0}(g^p)$ and $C_{V_0}(g^p)=C_{V_0}(g)$, since $g$ has projective prime order coprime to $p$. So $\dim C_{V}(g)\leq \dim C_{V_0}(g) \leq  C_{\hat{V}}(g)$ as required.  
	Let $\overline{W}$ denote the natural module for $G$ over $\overf_q$.
	By the proof of Proposition \ref{exttensor}, if $s \in G$ is semisimple, then as long as $s$ does not have eigenvalue multiplicities $(n-1,1)$ or $(n-2,2)$ on $\overline{W}$, then 
	\[
	k \, \mathrm{codim}C_V(s)-\log_q(|sF(G)^{\mathrm{PGL}_n(q)}|)\geq (k-1)(4n-6)-n/2-k(1+\epsilon_n).\\
	\]
	If $s\in G$ has eigenvalue multiplicities $(n-1,1)$ or $(n-2,2)$ on $\overline{W}$, then the lower bound is replaced with $ (k-1)(2n-2)-1-k(1+\epsilon_n)$ or $ (k-1)(4n-8)-1-k(1+\epsilon_n)$ respectively. The number of conjugacy classes of elements of prime order in $\pgl_n(q)$ that have eigenvalue multiplicities $(n-1,1)$ or $(n-2,2)$ on $\overline{W}$ is at most $q-1$ in each case, as long as $n\geq 5$.
	Moreover, if $u\in G$ is unipotent, then by the proof of Proposition \ref{exttensor}, 
	\[
	k\, \mathrm{codim}C_V(u)-\log_q(|uF(G)^{\mathrm{PGL}_n(q)}|) \geq (k-2)(3n-4)+\frac{7}{2}n-6 -k(1+\epsilon_n). 
	\]
	By \cite[Lemma 3.7]{MR2888238}, the number of conjugacy classes in $\mathrm{PGL}_n(q)$ is at most $q^{n-1}+5q^{n-2}$. We must also consider graph and graph-field automorphisms here. We find that a graph automorphism $\tau$ has $\dim C_V(\tau)\leq \frac{1}{2}(n^2-1-\ep_n)+ \frac{1}{2}(2\floor{n/2}-1-\ep_n)$
	
	Therefore, if $G$ has no regular orbit on $V$, letting $d=k(n^2-1-\ep_n)$, we have 
	\begin{align*}
	q^{d} &\leq  2(q^{n-1}+5q^{n-2})(q^{d-(k-2)(3n-4)-(7n/2-6)+k(1+\ep_n)}+q^{d-(k-1)(4n-6)+n/2+k(1+\epsilon_n)})\\
	&2(q-1)(q^{d-(k-1)(2n-2)+1+k(1+\epsilon_n)}+q^{d-(k-1)(4n-8)+1+k(1+\epsilon_n)})\\
	&+2\log(\log_2q+2)q^{n^2/2+d/2}+ 4q^{(n^2+n)/2-1}q^{d/2+\frac{k}{2}(2\floor{n/2}-1-\ep_n)}+2q^{(n^2-1)/2}q^{d/2}
	\end{align*}
	This inequality is false except for some cases with $k=3$ and $n=3$, as well as $k=2$ and $n=3,4,5$. For the remaining cases with $n=4,5$, we either replace $q^{n-1}+5q^{n-2}$  with a tighter upper bound for the number of $\pgl_n(q)$-classes of elements of prime order to obtain a contradiction, or we use the precise counts of elements with eigenvalue multiplicities $(n-1,1)$ or $(n-2,2)$ on $\overline{W}$. So suppose $n=3$. By Proposition \ref{permsquare} and the proof of Proposition \ref{adjoint}, if $G$ has no regular orbit on $V$, then letting $d=k(8-\ep_3)$, we have 
	\begin{align*}
	q^d \leq&  (q-1)\frac{|\gl_3(q)|}{|\gl_2(q)||\gl_1(q)|}(q^{4k}+q^{d-4k})+
	\frac{|\gl_3(q)|}{q^3|\gl_2(q)||\gl_1(q)|}q^{4k}+\frac{|\gl_3(q)|}{q^2|\gl_1(q)|}q^{3k}\\
	&+
	\frac{1}{2}\left(\binom{q-1}{2}+\frac{q+1}{2}+\frac{q^2+q+1}{3}\right)\frac{|\gl_3(q)|}{|\gl_1(q)|^3}(2q^{3k}+q^{(2-\ep_3)k}+\field\\
	&+4q^{\frac{n^2+n}{2}-1}q^{d/2+\frac{k}{2}(2\floor{n/2}-1-\ep_n)}+ 4q^{\frac{n^2-1}{2}}q^{d/2}
	\end{align*}
	This gives a contradiction for the remaining cases with $k=3$. When $k=2$, removing the terms that account for graph automorphisms gives a contradiction for all $q$ except $q=2$ when $\ep_3=0$, and $q\leq 13$ when $\ep_3=1$. For these remaining cases where $G$ does not contain graph automorphisms, we confirm the existence of a regular orbit of $G$ on $V$ using GAP. When $G$ contains graph automorphisms, we do not achieve a contradiction for $k=2$, so $b(G)\leq 2$ here by Lemma \ref{fieldext}.
\end{proof}

\begin{proposition}
	\label{2L1_base_size}
Let $G$ be almost quasisimple with $E(G)/Z(E(G)) \cong \mathrm{PSL}_n(q)$, and let $V=V(2\lam1)$ be the symmetric square for $E(G)$ over $\mathbb{F}_{q^k}$. If $k\geq 3$, then $G$ has a regular orbit on $V$. If instead $k=2$, then $b(G)\leq 2$. 
\end{proposition}
\begin{proof}
We will aim to apply Proposition \ref{tools}\ref{crude}. Let $d=k \binom{n+1}{2}$.
We will again proceed by considering the action of $G$ on $\overline{V} = V\otimes \overline{\mathbb{F}}_q$.
Suppose $s\in G$ is semisimple of prime order $r$ in $G$, and write the eigenvalues of $s$ on the natural module $W$ for $G$ as $t_1, t_2, \dots , t_m$, ordered so that the multiplicities $a_i$ of the $t_i$ are weakly decreasing. Suppose that $\{v_i \mid 1\leq i \leq n\}$ is a basis of eigenvectors of $s$ on the natural module $\overline{W}$ for $G$ over $ \overline{\mathbb{F}}_q$ so that $\{v_i v_j \mid 1\leq i\leq j \leq n\}$ forms a basis of $\overline{V}$ consisting of eigenvectors for $s$. Fix $t \in \overline{\mathbb{F}}_q$. Now, $V$ is isomorphic to the symmetric square of $W$, so 
by Proposition \ref{eigsp_from_permsquare}, $\dim V_t(s) \leq \tfrac{1}{2} \sum a_i^2+a_1$.
Therefore, 
\begin{align*}
k\, \mathrm{codim}V_t(s) - \log_q(|s^{\pgl_n(q)}|) \geq & k\left(\binom{n+1}{2}-\frac{1}{2}\sum a_i^2- a_1\right)-\left(n^2-\sum a_i^2+m/2 + \delta_{q,3} \frac{m}{6}\right)\\
=& \left( \frac{k}{2}-1\right) \left( n^2-\sum a_i^2\right)+\frac{kn}{2} - a_1-\frac{m}{2}-\delta_{q,3} \frac{m}{6}
%> & \left( \frac{k}{2}-1\right)(2n-2) +\frac{(k-2)n}{2}\\ %Since m \leq n-a1+1 (# eigvals) and n>= a1+1
\end{align*}
This is at least $(k/2-1)(3n-2)-\delta_{q,3} \frac{n}{6}$, since $m \leq n-a_1+1$ and $n\geq a_1+1$.

Now let $u \in G$ be unipotent of prime order $p$, and with associated canonical Jordan form $(J_p^{a_p}, \dots J_2^{a_2},J_1^{a_1})$ on the natural module of $G$. Suppose that $u$ does not have Jordan blocks $(J_2,J_1^{n-2})$ or $(J_2^2, J_1^{n-4})$; we will deal with these elements separately. Now,
\[
S^2((J_p^{a_p}, \dots J_2^{a_2},J_1^{a_1}))  = \sum_i a_i S^2(J_i) +  \sum_{i<j} a_ia_j J_i \otimes J_j + \sum_i \binom{a_i}{2} J_i\otimes J_i.
\]
By \cite[Lemmas 1.3.3, 1.3.4]{bigpaper}, $\dim C_V(u)  \leq \sum a_i \lceil i/2 \rceil + \sum_{i<j} ia_ia_j + \sum \binom{a_i}{2}i$. 
%Now suppose $u \in G$ is unipotent of prime order $p$, and with associated canonical Jordan form $(J_p^{a_p}, \dots J_2^{a_2},J_1^{a_1})$ on the natural module of $G$. By \cite[Lemma 1.3.4]{bigpaper}, $\dim C_V(u)  \leq \sum a_i \lceil i/2 \rceil$. 
Therefore,
\begin{align*}
k\, \mathrm{codim}C_V(u) -& \log_q(|u^{\pgl_n(q)}|) \geq \\
 & k\left(\binom{n+1}{2}-\sum_{i \textrm{ odd}} \frac{a_i}{2}-\sum_{i<j}ia_ia_j-\frac{1}{2}\sum ia_i^2 \right)-(n^2-2\sum_{i<j}ia_ia_j-\sum ia_i^2+m)\\
 & \geq (\frac{k}{2}-1)(n^2-2\sum_{i<j}ia_ia_j-\sum ia_i^2)+ k(\frac{n}{2}-\sum_{i \textrm{ odd}} \frac{a_i}{2})-m,
\end{align*}
which is at least $(k/2-1)(4n-6)$ since $m < 3\sqrt{n}/2$. Also note that if $u$ has Jordan form $(J_2,J_1^{n-2})$ or $(J_2^2, J_1^{n-4})$, then $\dim C_V(u) \leq d-n$  and $d-(2n-2)$ respectively. The number of such elements is given in the proof of Proposition \ref{unipcounts}.
%This is at least $(k/2-1)n^2$ since $\sum ia_i^2 \geq m$.
%
 If $G$ has no regular orbit on $V$, then
\begin{align*}
q^d \leq & 2(q^{n-1}+5q^{n-2})(q^{d-(k/2-1)(4n-6)}+q^{d-(k/2-1)(3n-2)-\delta_{q,3} \frac{n}{6}})+2\log(\log_2q+2)q^{n^2/2+d/2}\\
& +q^{2n-1}q^{d-kn}+4q^{4n-8}q^{d-k(2n-2)}.
\end{align*}
This is a contradiction for all $(k,n,q)$ with $k \geq 3$, $n\geq 3$ and $q\geq 3$, except for $(n,q) = (3,3), (3,5),(4,3)$ and $(5,3)$. We confirm using GAP \cite{GAP4} that $G$ has a regular orbit on $V$ in each of the remaining cases. %TODO check (3,5),(4,3),(5,3)
When $n=2$, we remove the $4q^{4n-8}q^{d-k(2n-2)}$ term, and find that there is a regular orbit of $G$ on $V$ for $k\geq 3$ and $q\geq 7$. Finally, if $k=2$, then $b(G)\leq 2$ by Proposition \ref{fieldext}.
%{\color{red}---NEED TO DO $k=2$ -----}
\end{proof}
	\begin{proposition}
	\label{L2_base_size}
	Let $G$ be almost quasisimple with $E(G)/Z(E(G)) \cong \mathrm{PSL}_n(q)$ with $n\geq 4$, and let $V=V(\lam2)$ be the exterior square for $E(G)$ over $\mathbb{F}_{q^k}$. Then $G$ has a regular orbit on $V$ unless one of the following holds.
	\begin{enumerate}
		\item $(n,k)=(4,4)$ and $b(G)\leq 2$, 
		\item $k=3$ with $n\in [4,6]$ and $b(G)\leq 2$
		\item $k=2$ with $n\geq 5$ and $b(G)\leq 2$, or 
		\item $(n,k)=(4,2)$ and $b(G)\leq 3$.
	\end{enumerate}
\end{proposition}
\begin{proof}
	We again consider the action of $G$ on $\overline{V} = V\otimes \overline{\mathbb{F}}_q$. Assume that $k\geq 3$.
	Suppose $s\in G$ is semisimple of prime order $r$ in $G$, with eigenvalues $t_1, t_2, \dots , t_m$ on the natural module $\overline{W}$ of $G$ over $\overf_q$, ordered so that their multiplicities $a_i$ are weakly decreasing. Also suppose that $(t_1,t_2) \neq (n-1,1)$ or $(n-2,2)$; we will deal with these cases separately. Now, $V$ is isomorphic to the exterior square of $W$, so by Proposition \ref{eigsp_from_permsquare}, for $t\in \overline{\mathbb{F}}_q$, we have $\dim V_t(s) \leq \frac{1}{2}\sum a_i^2$. Therefore, 
	\begin{align*}
	k\, \mathrm{codim}V_t(s) - \log_q(|s^{\pgl_n(q)}|) \geq &  k\left(\binom{n}{2}-\frac{1}{2}\sum a_i^2\right)-\left(n^2-\sum a_i^2+m\right) \\
	&> \left(\frac{k}{2}-1\right)(n^2-n-\sum a_i^2) \geq \left(\frac{k}{2}-1\right)(3n-6),
	\end{align*}
	since by assumption, $n^2-\sum a_i^2 \geq n^2-(n-2)^2-2 = 4n-6$. If $(t_1,t_2) = (n-1,1)$ or $(n-2,2)$, then we note that by the $\Psi$-net tables for $V(\lambda_2)$ in Table \ref{L2_tabs}, we have $V_t(s) \leq n-2$ and $2n-4$ respectively.
	Now suppose $u \in G$ is unipotent of prime order $p$, and with associated canonical Jordan form $(J_p^{a_p}, \dots J_2^{a_2},J_1^{a_1})$ on the natural module of $G$. %Suppose that the Jordan blocks of $u$ do not have $n^2-2\sum_{i<j}i a_ia_j-\sum ia_i^2<8n-20$ and/or largest Jordan block of size 2; there are at most $n/2 +9$ such classes, and we will deal with these elements separately.
	Now,
	\[
	\wedge^2((J_p^{a_p}, \dots J_2^{a_2},J_1^{a_1}))  = \sum_i a_i \wedge^2(J_i) +  \sum_{i<j} a_ia_j J_i \otimes J_j + \sum_i \binom{a_i}{2} J_i\otimes J_i.
	\]
	By \cite[Lemmas 1.3.3, 1.3.4]{bigpaper}, $\dim C_V(u)  \leq \sum a_i \lfloor i/2 \rfloor + \sum_{i<j} ia_ia_j + \sum \binom{a_i}{2}i$. 
	Therefore,
	\begin{align*}
	k\, \mathrm{codim}C_V(u) -& \log_q(|u^{\pgl_n(q)}|) \geq \\
	& k\left(\binom{n}{2}+\sum_{i \textrm{ odd}} \frac{a_i}{2}-\sum_{i<j}ia_ia_j-\frac{1}{2}\sum ia_i^2 \right)-(n^2-2\sum_{i<j}ia_ia_j-\sum ia_i^2+m)\\
	& \geq (\frac{k}{2}-1)(n^2-2\sum_{i<j}ia_ia_j-\sum ia_i^2) + k\sum_{i \textrm{ odd}} \frac{a_i}{2}-\frac{kn}{2}-m.
	\end{align*}
	Now, since $k \geq 3$,  we have $m \leq \sqrt{n}+ k\sum_{i \textrm{ odd}} \frac{a_i}{2}$. Moreover, $n^2-2\sum_{i<j}ia_ia_j-\sum ia_i^2 \geq 6n$, unless $u$ has an associated partition given in Table \ref{badparts}, which we treat separately. Let $\mathcal{S}$ denote a set of unipotent conjugacy class representatives of prime order with corresponding partition in Table \ref{badparts}. Excluding classes with representatives in $\mathcal{S}$, we have 
	\[
	k\, \mathrm{codim}C_V(u) - \log_q(|u^{\pgl_n(q)}|) \geq (\frac{5k}{2}-6)n-\sqrt{n}.
	\]

	\begin{table}[!htbp]
		\centering \begin{tabular}{ccc}
			\toprule
			Partition & $n^2-2\sum_{i<j} i a_i a_j-\sum i a_i^2$ & $\mathrm{codim}C_V(u) \geq$ \\ \midrule
			$(6,1^{n-6})$ & $10n-30$ & $5n-18$ \\
			$(5,2,1^{n-7})$ & $10n-32$ & $5n-19$ \\
			$(5,1^{n-5})$ & $8n-20$ & $4n-12$ \\
			$(4,3,1^{n-7})$ & $10n-34$ & $5n-20$ \\
			$(4,2^2,1^{n-8})$ & $10n-36$ & $5n-22$ \\
			$(4,2,1^{n-6})$ & $8n-22$ & $4n-14$ \\
			$(4,1^{n-4})$ & $6n-12$ & $3n-8$ \\
			$(3^2, 2,1^{n-8})$ & $10n-38$ & $5n-22$ \\
			$(3^2,1^{n-6})$ & $8n-24$ & $4n-14$ \\
			$(3,2^3,1^{n-9})$ & $10n-42$ & $5n-25$ \\
			$(3,2^2,1^{n-7})$ & $8n-26$ & $4n-16$ \\
			$(3,2,1^{n-5})$ & $6n-14$ & $3n-9$ \\
			$(3,1^{n-3})$ & $4n-6$ & $2n-4$ \\
			$(2^5,1^{n-10})$ & $10n-50$ & $5n-30$ \\
			$(2^4,1^{n-8})$ & $8n-32$ & $4n-20$ \\
			$(2^3,1^{n-6})$ & $6n-18$ & $3n-12$ \\
			$(2^2,1^{n-4})$ & $4n-8$ & $2n-6$ \\
			$(2,1^{n-2})$ & $2n-2$ & $n-2$ \\ \bottomrule
		\end{tabular}
		\caption{The centraliser dimensions and fixed point space codimensions of some unipotent elements $u$ acting on $V=V(\lam2)$.\label{badparts}}
	\end{table}
	%which is at least $(k/2-1)(8n-20)-kn/2-\lfloor3\sqrt{n}/2 \rfloor$.
	Therefore, if $G$ has no regular orbit on $V$,
	\begin{align*}
	q^d &\leq 2(q^{n-1}+5q^{n-2})(q^{d-(k/2-1)(3n-6)}+q^{d-( (\frac{5k}{2}-6)n-\sqrt{n})})+2\log(\log_2q+2)q^{n^2/2+d/2}\\
	&+2q^{2n-1}q^{d-k(n-2)}  +2q^{4n-7}q^{d-2k(n-2)}+ \sum_{u \in \mathcal{S}} 2^{t_u}q^{n^2-2\sum_{i<j} i a_i a_j-\sum i a_i^2}q^{d-k\,  \mathrm{codim}C_V(u)}\\
	\end{align*}
	where $d = k \binom{n}{2}$ and $t_u$ is the number of parts of distinct sizes in the partition corresponding to $u$.
	
	This gives a contradiction for all $k,n,q$ except for $4\leq n\leq 6$, as well as $k=4$ with $n=7$ and $q\geq 2$, and also $k=3$ with $n=7,8$ with $q\geq 2$, and $(n,2)$ for $n\leq 12$. For the remaining cases, with $k=4$ and $n\geq 5$, or $k=3$ and $n\geq 7$, we remove terms in the inequality based on elements of $\mathcal{S}$ that do not exist for a given $n$. This leaves us with a smaller list of $(n,q)$ that satisfy the inequality. All of these have $q\leq 5$, and by either computing the number of elements of each prime order in $G/F(G)$ precisely or constructing the module explicitly using GAP \cite{GAP4}, we determine that there is a regular orbit. Finally, if $k=2$, then $b(G)=2$ by Lemma \ref{fieldext} for $n\geq 5$, and $b(G)\leq 3$ if $n=4$.
	%{\color{red}----NEEDS FINISHING-----}
\end{proof}
\begin{proposition}
	\label{3L1_n=2}
Let $G$ be almost quasisimple with $E(G)/Z(E(G)) \cong \mathrm{PSL}_2(q)$, and let $V=V(3\lam1)$ or $V(4\lam1)$ over $\mathbb{F}_{q^k}$ for $k>1$. Then $G$ has a regular orbit on $V$. 
\end{proposition}
\begin{proof}
First let $V= V(3\lam1)$ over $\mathbb{F}_{q^k}$ for $k>1$. Then by Proposition \ref{alphas}, %l2lemma
if $G$ has no regular orbit on $V$,
\[
q^{4k} \leq 2|\mathrm{PGL}_2(q)|q^{2k}+2\log(\log_2q+2)q^{2+2k}
\]
This gives a contradiction for $k\geq 2$ and $q\geq 7$.
Now let $V= V(4\lam1)$ over $\mathbb{F}_{q^k}$. If $G$ has no regular orbit on $V$, then 
\[
q^{5k} \leq 2|\mathrm{PGL}_2(q)|q^{2k}+4(q^2+q)q^{3k}+2\log(\log_2q+2)q^{2+5k/2}
\]
This gives a contradiction for $q\geq 7$ and the result follows.
\end{proof}

\begin{proposition}
	\label{L3_n=6}
Let $G$ be almost quasisimple with $E(G)/Z(E(G)) \cong \mathrm{PSL}_n(q)$ for $n\in \{6,9\}$, and let $V=V(\lam3)$ over $\mathbb{F}_{q^k}$ for $k>1$. Then $G$ has a regular orbit on $V$, unless possibly when either $(n,k)=(6,3)$ and $G$ contains a graph automorphism, or $(n,k)=(6,2)$, in which cases $b(G)\leq 2$.
\end{proposition}
\begin{proof}
%{\color{red}----TO FINISH-----}
Suppose $n=9$. If $G$ has no regular orbit on $V$, then by the proof of Proposition \ref{l3},
\begin{align*}
q^{84k} \leq & 2|\mathrm{PGL}_9(q)|q^{84k-42k}+2\times 2q^{81/2+3/2}q^{84k-32k}+(2q^{2n-1}+4\left( \frac{q^{81/2+2}-1}{q^2-1}\right))q^{84-21k}\\
&+2\log(\log_2q+2)q^{81/2+42k}
\end{align*}
This gives a contradiction except for $q=2$, where replacing $4\left( \frac{q^{81/2+2}-1}{q^2-1}\right)$ with the precise number of unipotent elements of order 2 in $\mathrm{PGL}_9(2)$ gives a contradiction. Now suppose that $n=6$. Here there are graph automorphisms $\tau$ to consider. By Proposition \ref{graphfix}, $\dim C_V(\tau) \leq 14$. Therefore, if $G$ contains a graph automorphism and has no regular orbit on $V$, then
\begin{align*}
q^{20k} \leq & 2|\mathrm{PGL}_6(q)|q^{k(20-12)}+ 2(4q^{20}+4\left( \frac{q^{20}-1}{q^2-1}\right))q^{k(20-8)}+3q^{11}q^{k(20-6)}+2\log(\log_2q+2)q^{18+10k}\\
&+(2q^{20}+2q^{35/2})q^{14k}.
\end{align*}
This is a contradiction for $k\geq 4$ and $q\geq 2$. So $G$ has a regular orbit on $V$, and moreover $b(G)\leq 2$ for $k=2,3$ by Proposition \ref{fieldext}, since $b(G) =1$ for $k=4,6$ respectively. A similar argument shows that if $G$ does not contain a graph automorphism, then $G$ has a regular orbit on $V$ when $k\geq 3$ and $q\geq 2$, and $b(G)=2$ when $k=2$.
\end{proof}
\begin{proposition}
Let $G$ be almost quasisimple with $E(G)/Z(E(G)) \cong \mathrm{PSL}_3(q)$, and let $V=V(3\lam1)$ over $\mathbb{F}_{q^k}$ for $k>1$. Then $G$ has a regular orbit on $V$. 
\end{proposition}
\begin{proof}
By the proof of Proposition \ref{3l1}, if $G$ has no regular orbit on $V$, then
\[
q^{20k} \leq 2|\mathrm{PGL}_4(q)|q^{20k-8k}+4(q^{5}+q^{4})q^{20k-6k}+2(q^{5}+q^{5})q^{20k-4k}+2\log(\log_2q+2)q^{9/2+10k}
\]
This is a contradiction for $k\geq 2$ and $q\geq 5$.
\end{proof}
\begin{proposition}
	\label{ext_L4}
Let $G$ be almost quasisimple with $E(G)/Z(E(G)) \cong \mathrm{PSL}_8(q)$, and let $V=V(\lam4)$ over $\mathbb{F}_{q^k}$ for $k>1$. Then $G$ has a regular orbit on $V$. 
\end{proposition}
\begin{proof}
By the proof of Proposition \ref{l4}, if $G$ has no regular orbit on $V$ over $\mathbb{F}_{q^k}$ for $k>1$, then
\begin{align*}
q^{d} \leq & 2|\pgl _n(q)| q^{d-2k \binom{l-1}{3}} + 2(\atwoeven)q^{d-28k} +2(q^{2n-1}+q^{2n-1})q^{d-k\binom{n-2}{3}}\\
&+  \itwo q^{d-28k} + \field,
\end{align*}
where $d= k \binom{n}{4}$.  This gives a contradiction for $k\geq 2$ and $q\geq 2$ so the result follows.
\end{proof}
\begin{proposition}
	\label{ext_2L2}
Let $G$ be almost quasisimple with $E(G)/Z(E(G)) \cong \mathrm{PSL}_4(q)$, and let $V=V(2\lam2)$ over $\mathbb{F}_{q^k}$ for $k>1$. Then $G$ has a regular orbit on $V$. 
\end{proposition}
\begin{proof}
 This module is preserved by graph automorphisms, and there are 9 weights fixed by a graph automorphism.
By the proof of Proposition \ref{2l2}, if $G$ has no regular orbit on $V=V(2\lam2)$ over $\mathbb{F}_{q^k}$, then letting $d=k(20-\ep_3)$, we have
\[
q^d \leq 2 |\pgl_4(q) |q^{d-10k}+4(q^9+q^8)q^{d-(8-\ep_3)k}+2\log(\log_2q+2)q^{8+d/2}+(2q^{9}+2q^{15/2})q^{9k}
\]
This gives a contradiction for $k\geq 2$ and $q\geq 5$.
\end{proof}

We next take care of some absolutely irreducible modules from Section \ref{subfields} over extension fields. Here we replace $q_0$ with $q$ in our notation to be  consistent with Section \ref{subfields}.
\begin{proposition}
	\label{sub_ext_prop}
	Let $V=V_d(q)$ be a finite dimensional vector space over $\mathbb{F}_{q}$. Let $G\leq \Gamma \mathrm{L}(V)$ be almost quasisimple with $E(G)/Z(E(G)) \cong \mathrm{PSL}_n(q^c)$ for $n,c$ in Table \ref{ext_sub_table}. Also suppose that the restriction of $V$ to $E(G)$ is an absolutely irreducible module of highest weight $\lambda$ corresponding to $n,c$ in Table \ref{ext_sub_table}. Then $G$ has a regular orbit on $V \otimes \F_{q^t}$ for all integers $t >1$.
\end{proposition}
\begin{table}[h]
	\begin{tabular}{cccc}
		\toprule
		$\lambda$ & $c$ & $n$ & Section \ref{subfields} reference\\
		\midrule
		$(p^{e/3}+p^{2e/3}+1)\lam1$ & $3$ & $[2,3]$ & Propositions \ref{all_subnoro}, \ref{subcube}\\
		$(p^{e/4}+p^{2e/4}+p^{3e/4}+1)\lam1$& $4$ & 2 & Proposition \ref{sub_n^4}\\
		$2(p^{e/2}+1)\lam1$ & $2$ & $2$ & Proposition \ref{sub_symm_square}\\
		$(p^{e/2}+1)\lam2$ & $2$ & $4$ & Proposition \ref{sub_ext_square}\\
		\bottomrule
	\end{tabular}
	\caption{Some absolutely irreducible $\mathrm{SL}_n(q)$-modules realised over proper subfields of $\F_{p^e}$. \label{ext_sub_table}}
\end{table}

\begin{proof}
First suppose that $\lambda=(p^{e/3}+p^{2e/3}+1)\lam1$, so $c=3$ and let $n=3$. We have $d=27$. By Proposition \ref{permsquare}, a semisimple element $s \in G$ has $\dim C_V(s) \leq 15$ or $9$ if $s$ has eigenspaces of dimensions $2,1$ and $1,1,1$ respectively on the natural module of $G$. The number of prime order semisimple elements in $\pgl_3(q)$ whose preimages in $\mathrm{GL}_3(q^3)$ have eigenspace dimensions $2,1$  on the natural module for $G$ is at most $(q^3-1)q^6(q^6+q^3+1)$. The number of prime order unipotent elements in $\pgl_3(q^3)$ is at most $q^{18}$, and by \cite[Corollary 1]{MR2764924}, we have $\dim C_V(u) \leq 14$ for all unipotent projective prime order $u \in G$. Therefore, if $G$ has no regular orbit on $V \otimes \F_{q^t}$,
\[
q^{27t} \leq 2 |\pgl_3(q^3)|q^{9t}+2(q^3-1)q^6(q^6+q^3+1)q^{15t}+q^{18}q^{14t}+2\log(\log_2q+2)q^{27/2+\frac{1}{2}(27+\sqrt{27})t}.
\]
This inequality is false for $t\geq 2$ and $q\geq 2$, so $G$ has a regular orbit on $V \otimes \F_{q^t}$.
Now suppose that $n=2$. We compute that all semisimple and unipotent elements $x$ of projective prime order in $G$ have $\dim C_V(x) \leq 4$, with $\dim C_V(x) \leq 3$ if the projective prime order of $x$ is at least 5. We also compute that the number of odd prime order elements in $\pgl_2(q^3)$ is at most $q^3(q^3+1)(q^3+q^2+2q-1)$, that the number of involutions of $\pgl_2(q^3)$ is $q^6$, and the number of unipotent elements is $q^6-1$. Therefore, if $G$ has no regular orbit on $V\otimes \F_{q^t}$, then by Propositions \ref{invols} and the proof of Proposition \ref{field}, $q^{8t}$ is less than or equal to
\[
 2q^3(q^3+1)(q^3+q^2+2q-1)q^{3t}+4(q^7+q^4)q^{4t}+2q^6q^{4t}+q^6q^{4t}+2\log(\log_2q^3+2)q^{4+5t}+2\log(\log_2q^3+2)q^{6+4t},
\]
which is false, except for $t=2$ and $q\leq 7$. When $t=2$ and $3\leq q \leq 7$, we delete terms from the inequality as appropriate and substitute in the precise number of elements of order 3 and elements of order at least 5 in $\pgl_2(q^3)$ into the inequality. This gives us a contradiction in each case, so $G$ has a regular orbit on $V\otimes \F_{q^2}$. When $t=q=2$, we determine that there is a regular orbit of $G$ on $V\otimes \F_{q^2}$ by explicit construction in GAP \cite{GAP4}.

Now let $\lambda=(p^{e/4}+p^{2e/4}+p^{3e/4}+1)\lam1$, $c=4$ and $n=2$. We have $d=16$ and by the proof of Proposition \ref{sub_n^4}, if $G$ has no regular orbit on $V \otimes \F_{q^{t}}$, then 
\[
q^{16t} \leq 2 |\pgl_2(q^4)|q^{8t}+2\log(\log_2q^4+2)q^{8+10t},
\]
which gives a contradiction for $t\geq 2$ and  $q\geq 2$, and so the result follows.
%Following the notation of \cite[\S 3.2.2]{bg}, $t_1$ involutions have two eigenspaces of dimension 8 on $V$, while $t_1'$ involutions do not have any eigenspaces on $V$. 
%%%realised over $\mathbb{F}_q$. 
%All involutions $x\in G$ are also conjugate to $-x$. Moreover, unipotent elements $u \in G$ of prime order have $C_V(u) \leq 6+2\ep_2$  and semisimple elements of odd projective prime order $r$ have eigenspaces of dimension at most $4+2\ep$, where $\ep =1$ if $r \mid q^2-1$, and $\ep=0$ otherwise. The number of semisimple elements in $G/F(G)$ whose preimages in $G$ have $\ep=1$ is less than $2q^5(q^4+1)$, and the number for which $\ep=0$ is less than $(q^2+1)(q^4-1)(q^4+1)+\frac{1}{2}q^4(q^4-1)(q^4+1)$. Moreover, each prime order unipotent element $u$ of $G$ is conjugate to at least $q^4-1$ elements with the same eigenspace on $V$. Therefore, if $p$ is odd and $G$ has no regular orbit on $V$, then by Proposition \ref{eigsp2} and the proof of Proposition \ref{field}, %See this prf in book 10 for explanation
%\begin{align*}
%q^{16} \leq & \frac{1}{4}q^4(q^4+1)(2q^8)+q^5(q^4+1)(q^6+2q^5)+\frac{3}{2}(q^2+1)q^4(q^4+1)q^4+\left(\frac{q^4+1}{12}\right) q^4(q^4-1)(4q^4)\\
%&+\left(\frac{q^8-1}{q^4-1}\right) q^6+\frac{2}{q^2+1}\log(\log_2q+2)q^{16}
%\end{align*}
%This gives a contradiction for $q\geq 13$, while if $q\leq 11$, then using the precise numbers of different types of elements of prime order gives the result. An analogous calculation gives the result for $p=2$.

Now suppose that $\lambda = 2(p^{e/2}+1)\lam1$, with $ c=2$ and $n=2$. Here $d=9$ and $q$ is odd. By Propositions \ref{invols}, \ref{alphas} and \ref{tools}\ref{crude}), if $G$ has no regular orbit on $V\otimes \F_{q^t}$, then 
\[
q^{9t}\leq 2 |\pgl_2(q^2)|q^{\floor{\frac{9}{2}}t}+ 4(q^4+q^2)q^{\floor{\frac{3\times 9}{4}}t},
\]
which gives a contradiction for $t \geq 2$ and $q\geq 3$.
%Finally, suppose $n=2$. We summarise some information about projective prime order elements in $G$ in Table \ref{2L1_n=2_tens}.
%
%\begin{table}[h]
%	\begin{tabular}{ccc}
%		\toprule
%		Type & Number in $\mathrm{PGL}_2(q^2)$ & $\dim C_V(x)\leq$\\
%		\midrule
%		Involutions & $q^4$& 5\\
%		Order 3 & $2(q^{14/3}+q^{11/3})$ & 3\\
%		Unipotent & $q^4-1$ & 3\\
%		Semisimple $r\mid q-1$ & $2q^3(q^2+1)$ & $3$\\
%		Semisimple $r\mid q+1$ & $\frac{1}{2}q^2(q^2-1)(q^2+1)$ & $2$\\
%		\bottomrule
%		\end{tabular}
%	\caption{Information about elements of projective prime order $r$ in $G$ with $E(G)/Z(E(G))=\mathrm{PSL}_2(q^2)$. \label{2L1_n=2_tens}}
%	\end{table}
% We also observe that an involution $x\in G$ is conjugate to $-x\in G$. Therefore, if $G$ has no regular orbit on $V$ then by Proposition \ref{eigsp2} and the proof of Proposition \ref{field}, 
%\begin{align*}
%q^9 \leq &\frac(q^{14/3}+q^{11/3})(3q^3)+\frac{2q}{4}q^2(q^2+1)(3q^3)+\frac{5(q^2+1)}{8}q^2(q^2-1)q^2\\
%&+(q^4-1)q^3+\frac{1}{2}q^4(q^5+q^4)+\frac{2}{q+1}\log(\log_2 q+2)q^{4+9/2}
%\end{align*}
%This gives a contradiction for $q\geq 11$. When $q\leq 9$ then we use GAP to construct the respective modules and find that there is a regular orbit in the remaining cases.
Lastly, suppose that $n=4$, $c=2$ and $\lambda = (p^{e/2}+1)\lam2$, so $d=36$. We must consider graph automorphisms $\tau$ here, and determine that $\dim C_V(\tau) \leq 24 $.
 Therefore if $G$ has no regular orbit on $V \otimes \F_{q^t}$, then by the proof of Proposition \ref{sub_ext_square},
\[
q^{36t}\leq 2\left[2q^{14}+ 2q^{20}+4 \left( \frac{q^{20}-1}{q^4-1}\right)\right]q^{24t} +2|\mathrm{PGL}_4(q^2)| q^{12t}+ 2\log(\log_2q^2+2)q^{16+21t}+4(q^{18}+q^{15})q^{24t}.
\]
This gives a contradiction for $t\geq 2$ and $q\geq 2$, so $G$ has a regular orbit in all cases.

\end{proof}

We have now completed the proof of Theorem \ref{extmain}.
Finally, many of the base size results in this section for $V=V_0 \otimes F_{q_0^k}$ give upper bounds on the base size of $V_0$ by Proposition \ref{fieldext}. The next proposition finishes the proof of Theorem \ref{main} by completing Table \ref{rem1}.

\begin{proposition}
	\label{k=1_leftovers}
Let $V=V_d(q)$ be a $d$-dimensional vector space over $\mathbb{F}_q$, and let $G \leq \Gamma \mathrm{L}(V)$ be almost quasisimple with $E(G)/Z(E(G)) \cong \mathrm{PSL}_n(q)$. Further suppose that the restriction of $V=V(\lambda)$ to $E(G)$ is an absolutely irreducible module of highest weight $\lambda$, with $\lambda$ (up to quasiequivalence) in Table \ref{ext_sub_tab}. Let $\delta = 1$ if $G$ contains a graph automorphism, and zero otherwise. Then $b(G)$ is given in Table \ref{ext_sub_tab}.
\end{proposition}
\begin{table}[!htbp]
\begin{tabular}{cccc}
	\toprule
	$\lambda$ & $n$ & $b(G)$ & Reference \\ \midrule
	$2\lam1$ & $[2,\infty)$ & $2\leq b(G)\leq 3$ & Proposition \ref{2L1_base_size} \\
	$\lam2$ & $[7,\infty)$ & 3 & Proposition \ref{L2_base_size} \\
	& $[5,6]$ & $3 \leq b(G) \leq 4$& Proposition \ref{L2_base_size} \\
		& $4$ & $3 \leq b(G) \leq 5$& Proposition \ref{L2_base_size} \\
	$3\lam1$ , $p\geq 5$ & $2$ & 2 & Proposition \ref{3L1_n=2} \\
 &3 &  $1\leq b(G) \leq 2$ &Proposition \ref{3L1_n=2}  \\
	$\lam3$ & $6$ & $2\leq b(G) \leq 3+\delta$ & Proposition \ref{L3_n=6} \\
	&9& $1\leq b(G) \leq 2$ & Proposition \ref{L3_n=6}\\	
		$\lam4$ &8& $1\leq b(G)\leq 2$ & Proposition \ref{ext_L4} \\
				$2\lam2$, $p\geq 3$ &4& $1\leq b(G)\leq 2$&Proposition \ref{ext_2L2} \\
	$(p^a+1)\lam1$ & $[2,\infty)$ & 2 & Proposition \ref{exttensor} \\
	$\lam1+p^a\lam{n-1}$ & $[3,\infty)$ & 2 & Proposition \ref{exttensor} \\

	 \bottomrule
\end{tabular}
\caption{Some base sizes of $\mathbb{F}_qG$-modules $V(\lambda)$. \label{ext_sub_tab}}
\end{table}
\begin{proof}
This follows from Proposition \ref{fieldext} and each of the references given in Table \ref{ext_sub_tab}.
\end{proof}
 This concludes the proof of Theorem \ref{main}.
\section*{Acknowledgements}
This paper comprises a part of the PhD work of the author under the supervision of Professor Martin Liebeck. The author would like to thank Professor Liebeck for his guidance and careful reading of this paper. The author acknowledges the support of an EPSRC International Doctoral Scholarship at Imperial College London.
\bibliographystyle{plain}

%\bibliography{writeup}
\end{document}